\documentclass{amsart}
\usepackage[utf8]{inputenc}

\numberwithin{equation}{section}

\usepackage[a4paper]{geometry}
\usepackage{pinlabel}
\usepackage{amssymb,amsfonts,amsmath}
\usepackage{comment} 
\usepackage[all,arc]{xy}
\usepackage{enumerate, enumitem}
\usepackage{float}
\usepackage[caption = false]{subfig}
\usepackage{mathrsfs,mathtools}
\usepackage{todonotes,booktabs}
\usepackage{stmaryrd}
\usepackage{marvosym}
\usepackage{graphicx}
\usepackage{hyperref}
\usepackage{pdflscape}
\usepackage{microtype}
\usepackage[all]{xy}
\usepackage{tikz-cd}
\usepackage[noabbrev, capitalise, nameinlink]{cleveref}
\usepackage{marginnote}
\usepackage{todonotes}
\usepackage{multicol}

\definecolor{crimson}{rgb}{0.86, 0.08, 0.24}
\definecolor{darkcyan}{rgb}{0.0, 0.55, 0.55}

\hypersetup{
colorlinks={true},
linkcolor= {darkcyan},
citecolor= darkcyan}

\newtheorem{thm}{Theorem}[section]

\newtheorem*{thm*}{Theorem}
\newtheorem*{cor*}{Corollary}
\newtheorem*{prop*}{Proposition}

\newtheorem*{notation*}{Notation}

\newtheorem{example}[thm]{Example}
\newtheorem{defn}[thm]{Definition}
\newtheorem*{defn*}{Definition}
\newtheorem{prop}[thm]{Proposition}
\newtheorem{lem}[thm]{Lemma}
\newtheorem{rem}[thm]{Remark}

\newtheorem*{conj*}{Conjecture}
\newtheorem*{quest*}{Question}
\newtheorem{quest}[thm]{Question}
\newtheorem{thmx}{Theorem}

\newcommand{\dtone}{\bar{\mathfrak{d}}_{\bar{t}_1}}

\newcommand{\bt}{\bar{t}}
\newcommand{\tone}{\bar{t}_1}
\newcommand{\gtone}{\gamma\bar{t}_1}
\newcommand{\bv}{\bar{v}}
\newcommand{\BPG}{BP^{(\!(G)\!)}}
\newcommand{\BPR}{BP_\mathbb{R}}
\newcommand{\BPCn}{BP^{(\!(C_{2^n})\!)}}
\newcommand{\BPCi}{BP^{(\!(C_{2^i})\!)}}

\newcommand{\MUG}{MU^{(\!(G)\!)}}
\newcommand{\Sh}{\mathbb{S}_h}
\newcommand{\Gh}{\mathbb{G}_h}
\newcommand{\SliceSS}{\mathrm{SliceSS}}
\newcommand{\HFPSS}{\mathrm{HFPSS}}

\newcommand{\Z}{\underline{\mathbb{Z}}}
\newcommand{\BPCfour}{BP^{(\!(C_{4})\!)}}
\newcommand{\BPCfourone}{BP^{(\!(C_{4})\!)}\langle 1 \rangle}

\newcommand{\Cn}{C_{2^n}}
\newcommand{\Ci}{C_{2^i}}
\newcommand{\Ck}{C_{2^k}}
\newcommand{\Ckplusone}{C_{2^{k+1}}}

\newcommand{\Cnminusone}{C_{2^{n-1}}}
\newcommand{\WW}{\mathbb{W}}
\newcommand{\RN}{^R{N}_H^G}
\newcommand{\La}{\mathbb{L}}

\newcommand{\Ind}{\mathrm{Ind}}
\newcommand{\res}{\mathrm{res}}
\newcommand{\tr}{\mathrm{tr}}
\newcommand{\Res}{\mathrm{Res}}

\newcommand{\TateSS}{\mathrm{TateSS}}

\newcommand{\ind}{\mathrm{Ind}}

\newcommand{\SpG}{\mathrm{Sp}^G}
\newcommand{\SpH}{\mathrm{Sp}^H}
\newcommand{\SpFilH}{\mathrm{Sp}^{H,\mathrm{Fil}}}
\newcommand{\SpFilG}{\mathrm{Sp}^{G,\mathrm{Fil}}}

\newcommand{\Aut}{\mathrm{Aut}}
\newcommand{\Gal}{\mathrm{Gal}}
\newcommand{\Pic}{\mathrm{Pic}}
\newcommand{\colimit}{\mathrm{colim}}
\newcommand{\sk}{\mathrm{sk}}
\makeatletter
\newcommand*{\rom}[1]{\expandafter\@slowromancap\romannumeral #1@}
\makeatother

\author[Duan]{Zhipeng Duan}\address{School of Mathematical Sciences, Nanjing Normal University}\email{zhipeng@njnu.edu.cn}
\author[Hill]{Michael A. Hill}\address{Department of mathematics, University of Minnesota} \email{mahill@umn.edu}
\author[Li]{Guchuan Li}\address{School of Mathematical Sciences, Peking University}\email{liguchuan@math.pku.edu.cn}
\author[Liu]{Yutao Liu}\address{Department of Mathematics, University of Washington} \email{yutao492@uw.edu}
\author[Shi]{XiaoLin Danny Shi}\address{Department of Mathematics, University of Washington}\email{dannyshi@uw.edu}
\author[Wang]{Guozhen Wang} \address{Shanghai Center for Mathematical Sciences, Fudan University} \email{wangguozhen@fudan.edu.cn}
\author[Xu]{Zhouli Xu} \address{Department of Mathematics, UCLA} \email{xuzhouli@ucla.edu}

\title{Periodicity and finite complexity in higher real $K$-theories}

\begin{document}
\begin{abstract}
In this paper, we establish periodicity results for higher real $K$-theories at all heights and for all finite subgroups of the Morava stabilizer group at the prime 2. We further analyze the $RO(G)$-periodicity lattice of the height-$h$ Lubin--Tate theory, proving new $RO(G)$-graded periodicities and explicit finiteness results for the $RO(G)$-graded homotopy groups of $E_h$. Together, these results provide a foundation for both the structural and computational study of higher real $K$-theories.

\end{abstract}
\maketitle

\setcounter{tocdepth}{1}
\tableofcontents

\section{Introduction}

\subsection{Motivation and main results}

Historically, topological $K$-theories, together with their Bott periodicities and Adams operations, have resolved several fundamental problems in algebraic and geometric topology, including the Hopf invariant one problem \cite{Adams60, AdamsAtiyah66} and the vector fields on spheres problem \cite{Adams62}. Adams also used real $K$-theory to analyze the image of the $J$-homomorphism $J: \pi_k SO \to \pi_k S^0$, showing that the 8-fold Bott periodicity in real $K$-theory induces corresponding periodic families in the stable homotopy groups of spheres \cite{Adams66ImgJ} (see also \cite{MahowaldImageJ}). This stands as one of the earliest influential applications of periodic generalized cohomology theories to the study of the stable homotopy groups of spheres.

$K$-theories are part of a broader family of generalized cohomology theories central to chromatic homotopy theory. The moduli stack of formal groups is stratified by height, and Quillen’s foundational work \cite{Quillen69}, together with subsequent developments \cite{MRW77,DHS88,RavenelOrange,HS98,COCTALOS,RavenelGreen,LurieChromatic}, identifies this stratification with localizations with respect to the Lubin–Tate theories. These theories generalize complex $K$-theory $KU$ and carry symmetries that are higher-height analogues of the Adams operations on $KU$.

Let $E(k, \Gamma)$ be the Lubin--Tate theory associated with a height-$h$ formal group $\Gamma$ over a perfect field $k$ of characteristic $p$.  We will often suppress $(k, \Gamma)$ in the notation and write $E_h = E(k, \Gamma)$.  Let $\mathbb{S}_h = \Aut_k(\Gamma)$ be the small Morava stabilizer group, and let $\mathbb{G}_h = \mathbb{S}_h \rtimes \Gal(k/\mathbb{F}_p)$ be the big Morava stabilizer group.  By the Goerss--Hopkins--Miller Theorem \cite{Rez98,GH04,Lur18}, the continuous action of $\mathbb{G}_h$ on $\pi_* E_h$ refines uniquely to an $\mathbb{E}_\infty$-action of $\mathbb{G}_h$ on $E_h$. For example, at height $h = 1$, we have $E_1 = KU^{\wedge}_p$ (the $p$-completed complex $K$-theory), $\mathbb{G}_1 = \mathbb{Z}_p^\times$, and the action of $\mathbb{Z}_p^\times$ on $KU^{\wedge}_p$ recovers the Adams operations. 


\begin{notation*} \rm
For a finite subgroup $G \subset \mathbb{G}_h$, set 
\[EO_h(G) := E_h^{hG} = F(EG_+, E_h)^G,\]
the homotopy fixed points of $E_h$ under the action of $G$. 
\end{notation*}
The theories $EO_h(G)$ are called the \textit{higher real $K$-theories}, named because when $p = 2$, $h = 1$, and $G = C_2$, one recovers the 2-completed real $K$-theory $KO^{\wedge}_2$ as $EO_1(C_2)$.  As the height $h$ and group $G$ increase, the theories $EO_h(G)$ encode increasingly refined information about the stable homotopy category, but their analysis becomes correspondingly more intricate.

By the Nilpotence Theorem of Devinatz, Hopkins, and Smith \cite{DHS88, HS98}, each $EO_h(G)$ is a periodic cohomology theory, with \textit{some} finite period. However, beyond low heights where explicit computations of $\pi_*EO_h(G)$ are available, the periods for general $h$ and $G$ remain unknown.

\vspace{0.05in}
\begin{quest}\rm \label[quest]{quest:Periodicity}
What is the periodicity of $EO_h(G)$?
\end{quest}
\vspace{0.05in}

Establishing periodicity results for $EO_h(G)$ is of central importance for several reasons. Although computations involving $EO_h(G)$ are often intricate, the instances in which its periodicities are known have been instrumental in resolving major problems in algebraic and differential topology.  Examples include the following:

\begin{enumerate}

    \item The resolution of the Kervaire invariant one problem by Hill, Hopkins, and Ravenel \cite{HHR16} relies on the periodicity of $EO_4(C_8)$. The theory $EO_4(C_8)$ detects the potential Kervaire invariant elements, and its $256$-periodicity forces these elements to vanish in $\pi_* S^0$ (see also \cite{Ravenel78, LSWX2019}).
    
    \item At the prime 2, $\textup{TMF}$, whose $K(2)$-localization is the higher real $K$-theory $EO_2(G_{48})$, is 192-periodic. This periodicity governs both its structural features and geometric applications. In particular, Behrens, Mahowald, and Quigley \cite{BMQ20} determine the Hurewicz image of the unit map $\pi_* S^0 \to \pi_* \textup{tmf}$, producing 192-periodic families in $\pi_* S^0$. Through the Kervaire–Milnor classification of homotopy spheres \cite{KervaireMilnor63}, these families, together with the work of Behrens, Hill, Hopkins, and Mahowald \cite{BHHM20}, establish the existence of exotic spheres in congruence classes of dimensions modulo 192, accounting for more than half of the even-dimensional cases.

\item The theories $EO_h(G)$ serve as fundamental building blocks for the $K(h)$-local sphere via the theory of finite resolutions. The theory of finite resolutions was developed
by Goerss--Henn--Mahowald--Rezk \cite{GHMR05}, followed by the works of Beaudry, Behrens, Bobkova, Goerss, and Henn \cite{Behrens2006, Henn2007, Beaudry2015, BG18, Henn2019, BGH2022}.
    
    \item In addition to the geometric applications of $K$-theories discussed at the beginning of this section, $KO$ plays a central role in the study of spin manifolds, as established by Atiyah, Bott, and Shapiro \cite{ABS1964}. This leads directly to the $KO$-valued index of the Dirac operator via the Atiyah–Singer Index Theorem \cite{AtiyahSingerI}.  In this way, $KO$ connects geometric questions about curvature and symmetry to stable homotopy theory. In addition, $KO$ also plays a significant role in studying the Geography Problem for 4-manifolds through the analysis of $\textup{Pin(2)}$-equivariant Mahowald invariants \cite{Furuta01, HLSX2022}.
    
\end{enumerate}

In view of its significance, resolving \cref{quest:Periodicity} is not only a goal in its own right, but also a central step toward addressing every question concerning $EO_h(G)$, ranging from coefficients and detection theorems to duality and geometry.

\vspace{0.07in}

\textit{In this paper, we will establish periodicity results for $EO_h(G)$ at all heights $h \geq 1$ and all finite subgroups $G \subset \mathbb{G}_h$ at the prime 2, answering \cref{quest:Periodicity} when $p = 2$.}  

\vspace{0.07in}

To state our main result, we will first recall the classification of the finite subgroups of $\mathbb{G}_h$. The finite subgroups of $\mathbb{S}_h$ and $\mathbb{G}_h$ have been completely classified at all heights $h$ \cite{Hew95, Hew99, Buj2012}.  To summarize this classification at the prime 2, let $h = 2^{n-1}m$, where $m$ is an odd number.  If $n\neq 2$, the maximal finite $2$-subgroups of $\mathbb{S}_h$ are isomorphic to $C_{2^n}$, the cyclic group of order $2^n$.  When $n=2$, the maximal finite $2$-subgroups of $\mathbb{S}_h$ are isomorphic to the quaternion group $Q_8$.  As a result, for any $G \subset \mathbb{G}_h$ a finite subgroup, a 2-Sylow subgroup $H$ of $G \cap \mathbb{S}_h$ is isomorphic to either $C_{2^n}$ or $Q_8$.


\begin{defn} \rm \label[defn]{defn:PhG}
For $h \geq 1$ and For $G \subset \mathbb{G}_h$ a finite subgroup, let $H$ be a 2-Sylow subgroup of $G_0 = G \cap \mathbb{S}_h$.  Define
\[P(h, H) := \left \{\begin{array}{ll} 2 & \text{if } H = e \\ 2^{h+n+1} & \text{if } H = \Cn \\
2^{h+4} & \text{if } H = Q_8\end{array} \right. \]
and define 
\[P(h, G) := \frac{|G_0|}{|H|} \cdot P(h, H).\] 
\end{defn}

\begin{thmx}[Periodicity]\label{thm:A} \hfill \\
The cohomology theory $EO = EO_h(G)$ is $P$-fold periodic, where $P = P(h, G)$: for all $X$, 
\[EO^{*+P}(X) \cong EO^*(X).\]
\end{thmx}

\cref{thm:A} recovers all previously known periodicities of $EO_h(G)$ established through explicit computations:
{\begin{multicols}{2}
\begin{itemize}
\item $EO_1(C_2)$: 8 (\cite{Bott59, Atiyah66KR})
\item $EO_2(C_2)$: 16 (\cite{MahowaldRezk09}) 
\item $EO_2(C_4)$: 32 (\cite{HHR16, BO16}))
\item $EO_2(C_6)$: 48 (\cite{MahowaldRezk09})
\item $EO_2(Q_8)$: 64 (\cite{DFHH14, DKLLW24})
\item $EO_2(G_{24})$: 192 (\cite{Bau08,DFHH14})
\columnbreak
\item $EO_2(G_{48})$: 192 (\cite{Bau08,DFHH14})
\item $EO_4(C_2)$: 64 (\cite{HS20, LSWX2019})
\item $EO_4(C_4)$: 128 (\cite{HSWX23})
\item $EO_4(C_8)$: 256 (\cite{HHR16,LWX24})
\item $EO_4(C_{12})$: 384 (\cite{HSWX23})
\item $EO_h(C_2)$: $2^{h+2}$ (\cite{HS20, LSWX2019})
\end{itemize}
\end{multicols}}

In particular, when $h = 1$ and $G = C_2$, we have $P(1,C_2) = 8$, and \cref{thm:A} recovers the Bott periodicity of real $K$-theory $KO$. At height $h = 2$, for $G = G_{48}$, one has $G_0 = G_{24}$, $H = Q_8$, and $P(2,G_{48}) = \frac{|G_{24}|}{|Q_8|} \cdot 2^{2+4} = 192$. This yields the 192-periodicity of $EO_2(G_{48})$, which is the $K(2)$-localization of $\textup{TMF}$ at the prime 2. When $h = 4$ and $G = C_8$, the theory $EO_4(C_8)$ is closely related to the Hill–Hopkins–Ravenel detection spectrum $\Omega$, and is $P(4,C_8) = 256$-periodic, sharing the same periodicity as $\Omega$.

Proving \cref{thm:A} requires the following results as input: 
\begin{enumerate}
\item The equivariant slice spectral sequence for norms of Real bordism theories $\MUG$ and $\BPG$, established by Hill, Hopkins, and Ravenel \cite{HHR16}.
\item The Real orientation of $E_h$, established by Hahn and the fifth author \cite{HS20}.  This serves as a bridge to transport all the computational results for localizations of $\BPG$ to $E_h$.
\item The strong vanishing line results of $E_h^{hG}$, for all finite subgroups $G \subset \mathbb{G}_h$, proved in \cite{DLS2022} by the first, third, and fifth authors.
\end{enumerate}  
In particular, \cref{thm:A} and its proof works for all forms of $E_h$, for any choice of $(k,\Gamma)$.


A crucial step in proving \cref{thm:A} is to study $E_h$ as an equivariant $G$-spectrum, with homotopy groups $\underline{\pi}_\star E_h$ graded by $RO(G)$, the real representation ring of $G$.  The $G$-spectrum $E_h$ admits various $RO(G)$-graded periodicities that are not directly detected by $E_h^{hG}$, which is a nonequivariant spectrum and can only detect integer-graded periodicities.  

This leads us to define and study the $RO(G)$-periodicity lattice $\La_{h, G}$ (\cref{sec:lattice}).  For $G$ a finite subgroup of $\mathbb{G}_h$, the $RO(G)$-periodicity lattice $\La_{h, G}$ is the free abelian subgroup of $RO(G)$ that is generated by all $V \in RO(G)$ such that $E_h$ is $V$-periodic as a $G$-spectrum. 

For a given pair $(h, G)$, the lattice $\La_{h, G}$ encodes the complexity of the $RO(G)$-graded homotopy groups $\pi_\star^G E_h$. More precisely, to compute all of $\pi^G_\star E_h$, it is enough to compute the homotopy groups in the quotient grading $RO(G)/\La_{h, G}$. 
\begin{defn} \rm \label[defn]{defn:complexity}
The \textit{complexity} of $\pi_\star^G E_h$ is the quotient $RO(G)/\La_{h, G}$. 
\end{defn}

The lattice $\La_{h, G}$ is also useful for determining the Spanier--Whitehead dual of $EO_h(G)$ in the $K(h)$-local category \cite{BGHS2022}, as well for computing $K(h)$-local Picard groups \cite{BBHS20,HLS21}.

The lattice $\La_{h, G}$ is determined by $\La_{h, H}$, where $H$ is a 2-Sylow subgroup of $G_0 = G \cap \mathbb{S}_h$.  From this point on, we will assume that $G$ is a finite 2-subgroup of $\mathbb{S}_h$, which is isomorphic to either $\Cn$ or $Q_8$.

\begin{thmx}[$RO(G)$-Periodicity, \cref{theorem:ROGperiodicity}] \label{thm:B} \hfill \\
As a $G$-spectrum, $E_h$ has the following $RO(G)$-graded periodicities:
\begin{enumerate}
\item  When $(h, G) = (2^{n-1} m, \Cn)$: 
\begin{enumerate}[label=(\alph*)]
        \item $\rho_{2^n} = 1 + \sigma + \lambda_1 + 2 \lambda_2 + \cdots + 2^{n-2}\lambda_{n-1}$;
        \item $2^{2^{n-i}m +n-i+1}-2^{2^{n-i}m+n-i}\lambda_{n-i}$, $1\leq i\leq n$.
\end{enumerate}
\item When $(h, G) = (4\ell+2, Q_8)$: 
\begin{enumerate}[label=(\alph*)]
\item $\rho_{Q_8} = 1+\sigma_i+\sigma_j+\sigma_k+\mathbb{H}$;
\item $2^{2\ell+2}+2^{2\ell+2}\sigma_i-2^{2\ell+2}\sigma_j-2^{2\ell+2}\sigma_k$;\\
$2^{2\ell+2}-2^{2\ell+2}\sigma_i+2^{2\ell+2}\sigma_j-2^{2\ell+2}\sigma_k$;\\
$2^{2\ell+2}-2^{2\ell+2}\sigma_i-2^{2\ell+2}\sigma_j+2^{2\ell+2}\sigma_k$;
\item $2^{h+2}+2^{h+2}\sigma_i-2^{h+1}\mathbb{H}$; \\
$2^{h+2}+2^{h+2}\sigma_j-2^{h+1}\mathbb{H}$;\\
$2^{h+2}+2^{h+2}\sigma_k-2^{h+1}\mathbb{H}$.
\end{enumerate}
\end{enumerate}
Here, the representations in $RO(\Cn)$ and $RO(Q_8)$ are defined as in \cref{subsec:Notations}.
\end{thmx}
To prove the $RO(G)$-periodicities in \cref{thm:B}, we use two inductive methods to analyze the lattice $\La_{h,G}$.  
\begin{enumerate}
    \item Transchromatic isomorphism: as a consequence of \cite[Theorem~B]{MSZ24}, if $G = \Cn$ and $V \in \La_{h/2, G/C_2}$ with $|V| = 0$, then $V \in \La_{h,G}$.  Here, $V$ is viewed as an element of $RO(G)$ via the pullback $G \to G/C_2$. 
    \item The norm: the collection of abelian groups $\{\La_{h,H} \mid H \subset G\}$ forms a Mackey functor $\underline{\La}_{h,G}$, induced by the Hill--Hopkins--Ravenel norm functor $N_H^G(-)$ \cite{HHR16}.  Therefore, if $V \in \La_{h,H}$, then $\Ind_H^G(V) \in \La_{h,G}$. 
\end{enumerate}
Applying these two inductive methods, we prove all the $RO(G)$-periodicities in \cref{thm:B}. 

In general, determining the full periodicity lattice $\La_{h,G}$ and being able to compute the quotient $RO(G)/\La_{h,G}$ for arbitrary $h$ and $G$ is a difficult task, as it requires extensive computations of the $RO(G)$-graded homotopy groups $\pi_\star^G E_h$. Nevertheless, \cref{thm:B} provides enough $RO(G)$-periodicities to produce a significant sublattice $\La'_{h,G} \subset \La_{h,G}$. In particular, the sublattice \textit{$\La'_{h,G}$ has full rank in $RO(G)$}, and the finite quotient group $RO(G)/\La'_{h,G}$ can be explicitly identified. Since $RO(G)/\La'_{h, G}\to RO(G)/\La_{h, G}$ is a further quotient,  $\La_{h,G}$ itself is a full-rank sublattice of $RO(G)$, leading to the following finiteness result on the complexity of $\pi_\star^G E_h$.

\begin{thmx}[Finite Complexity, \cref{thm:finiteComplexity}] \label{thm:C} \hfill \\
The complexity of $\pi_\star^G E_h$ is finite, given by the following bounds: 
\begin{enumerate}
\item When $(h,G)=(2^{n-1} m, \Cn)$, 
\[RO(G)/\La'_{h,G} \cong  \bigoplus_{i=1}^{n-1} \mathbb{Z}/2^{2^{n-i-1}m+n-i} \oplus \mathbb{Z}/2^{h+n+1}.\]
\item When $(h,G)=(4\ell+2, Q_8)$, 
\[RO(G)/\La'_{h,G} \cong \mathbb{Z}/2^{2\ell+2}\oplus \mathbb{Z}/2^{2\ell+3}\oplus\mathbb{Z}/2^{2\ell+3}\oplus \mathbb{Z}/2^{4\ell+6}.\]
\end{enumerate}
\end{thmx}

The strength of \cref{thm:C} is that it reduces computations with $EO_h(G)$ to a finite window. In particular, the differentials and (hidden) extensions in the spectral sequences computing $\pi_\star^G E_h$ repeat with a fixed shift, so it is enough to determine them in a bounded range and then extend the pattern accordingly. We will illustrate this in \cref{sec:E2C4} through an explicit example.

In \cref{sec:norm} and \cref{sec:sharpness}, we turn our attention to the relationship between periodicities, vanishing lines, and differentials. The norm functor $N_H^G(-)$ will continue to play a crucial role. In particular, the $RO(G)$-periodicities arise from orientation classes that are permanent cycles. Since the norm of a permanent cycle is still a permanent cycle, the norm converts $RO(H)$-graded periodicities into $RO(G)$-graded ones.

To make further use of the norm functor, we study norm structures on filtered equivariant spectra and their effect on differentials. This leads to differential formulas for classes that were not permanent cycles before applying the norm. We refer to this as the \textit{norm-transfer differential formula} (\cref{thm:predicttransfer}). A variant of this formula played an important role in establishing vanishing lines in \cite{DLS2022} and in computations in \cite{HHR17, HSWX23, MSZ23}. As a consequence of the norm-transfer formula, we obtain a family of differentials on orientation classes, which in turn provides finer control over periodicities.

\begin{thmx}[Differentials on Orientation Classes, \cref{thm:predicteddiff}]\label{thm:D}\,
\begin{enumerate}
\item When $(h, G) = (2^{n-1}m, \Cn)$, we have the following differentials in the $\Cn$-slice spectral sequence and the $\Cn$-homotopy fixed point spectral sequence of $E_h$: 
\[d_{2^{\ell+1}-1} \left(u_{2^{\ell+n-2}\lambda_{n-1}}\right)=\tr_{C_2}^{C_{2^{n}}}\left(\bar{v}_{\ell} a_{(2^{\ell+1}-1)\sigma_2}u_{(2^{\ell+n-1}-2^\ell )\sigma_2}\right)u_{W_{\ell}}, \,\,\, 1\leq \ell \leq h,\]
where $W_{\ell} = 2^\ell \sigma_{2^{n}}+2^\ell \lambda_1+2^{\ell+1}\lambda_2 + \cdots + 2^{\ell+n-3}\lambda_{n-2}.$
\item When $(h, G) = (4m+2, Q_8)$, we have the following differentials in the $Q_8$-homotopy fixed point spectral sequence of $E_h$: 
\[d_{2^{\ell+1}-1} \left(u_{2^\ell\mathbb{H}}\right)=\tr_{C_2}^{Q_8}\left(\bar{v}_{\ell} a_{(2^{\ell+1}-1)\sigma_2}u_{3 \cdot 2^{\ell} \sigma_2}\right)u_{2^\ell \sigma_i} u_{2^\ell \sigma_j}u_{2^\ell \sigma_k}, \,\,\, 1\leq \ell \leq h.\]
\end{enumerate}
\end{thmx}


Finally, in \cref{sec:E2C4}, we present a computational application of the Periodicity Theorem (\cref{thm:A}), using it as a central input to establish differentials and resolve extension problems in the equivariant slice spectral sequence. More specifically, we compute the $C_4$-slice spectral sequence of $D^{-1}\BPCfour \langle 1 \rangle$, where $D = N_{C_2}^{C_4}(\bt_1)$. This periodic Hill–Hopkins–Ravenel theory is a model for the $C_4$-equivariant Lubin–Tate theory $E_2$ \cite{BHSZ21}.

A highlight of this computation is that all higher differentials can be derived from a single $d_3$-differential by exploiting qualitative features of the equivariant slice spectral sequence of $E_2$, such as periodicity, transchromatic phenomena, and strong vanishing lines. We remark that these structural theorems extend to $EO_h(G)$ at higher heights as well. However, computations for $EO_h(G)$ become increasingly intricate as $h$ and $G$ grow, and knowing only the lower differentials together with these structural results is unlikely to yield a complete picture. Nevertheless, these principles offer a more conceptual and streamlined framework for extending computations to higher heights, revealing substantial new information about their structure.


\subsection{Organization of the paper}
In \cref{sec:intP2}, we establish periodicities for $EO_h(G)$ for any finite subgroup $G \subset \mathbb{G}_h$, proving \cref{thm:A}.  In \cref{sec:lattice}, we analyze the $RO(G)$-periodicities of $E_h$ and prove \cref{thm:B} (\cref{theorem:ROGperiodicity}) and \cref{thm:C} (\cref{thm:finiteComplexity}). 

In \cref{sec:norm}, we define norm structures on filtered equivariant spectra and analyze their effect on differentials. We show that the presence of a norm structure relates differentials across spectral sequences associated to various subgroups (\cref{thm:predicttransfer}). In \cref{sec:sharpness}, we apply this framework to the homotopy fixed point spectral sequences of Lubin–Tate theories, proving \cref{thm:D}.

Finally, in \cref{sec:E2C4}, we present a computational application of \cref{thm:A}, using it as a central input to establish differentials and resolve extension problems in the $C_4$-slice spectral sequence of $D^{-1}\BPCfour \langle 1 \rangle$, which serves as a model for the $C_4$-equivariant Lubin–Tate theory $E_2$. All higher differentials are obtained from a single $d_3$-differential by using periodicity, transchromatic phenomena, and strong vanishing lines present in the equivariant slice spectral sequence.

\subsection{Acknowledgements}
The authors would like to thank Agn\`{e}s Beaudry, Mark Behrens, Andrew Blumberg, Paul Goerss, Mike Hopkins, J.D. Quigley, Doug Ravenel, Yuchen Wu, and Lucy Yang for helpful conversations. 

This material is based upon work supported by the National Natural Science Foundation of China under Grant No.~12401084 (first author), by the National Science Foundation under Grant No.~DMS-2528364 (second author), DMS-2313842 and DMS-2404828 (fifth author), and the AMS Centennial Fellowship (seventh author).
The sixth author is partially supported by the following Grants: NSFC-12325102, NSFC-12226002, the New Cornerstone Science Foundation, Shanghai Pilot Program for Basic Research–Fudan University 21TQ1400100 (21TQ002), Shanghai Municipal Science and Technology Commission 248002139.

\subsection{Notations and conventions} \label{subsec:Notations}

\begin{itemize}
\item Let $RO(G)$ denote the real representation ring of $G$. Given a $G$-spectrum $X$ and a subgroup $H$ in $G$, we use $\pi_\star^H X$ to denote the $H$-equivariant homotopy group $[S^\star, X]^H$, where $\star\in RO(G)$. This group $\pi_\star^H X$ admits an action by the Weyl group $W_G(H)=N_G(H)/H$. In particular, when $H=e$, the underlying homotopy group $\pi_\star^e X$ is a $G$-module. 

\item When $G$ is a cyclic 2-group and its order is clear from the context, we denote its sign representation by $\sigma$. In \cref{sec:lattice} and \cref{sec:norm}, when there are multiple cyclic groups and we need to distinguish the sign representation associated to different cyclic group $G$, we will use $\sigma_{2^i}$ to indicate it is the sign representation of the cyclic group $C_{2^i}$.  

\item When $G = \Cn$, let $\lambda_i$ be the $2$-dimensional $\Cn$-representation corresponding to rotation by $(\frac{\pi}{2^i})$, where $0 \leq i \leq n-1$. Note that $\lambda_0 = 2 \sigma_{2^n}$. When working $2$-locally, the representation spheres associated to rotations by $(\frac{j\pi}{2^{n-1}})$ and $(\frac{j'\pi}{2^{n-1}})$ are equivalent whenever $j$ and $j'$ have the same $2$-adic valuation. As a result, when we consider $RO(\Cn)$-graded homotopy groups, it is enough to work with the gradings generated by the irreducible representations $\{1, \sigma_{2^{n}}, \lambda_1, \dots, \lambda_{n-1}\}$. 

\item We will write $\rho_{2^n}$ for the regular representation of $\Cn$. When working 2-locally, the representation sphere associated to $\rho_{2^n}$ is equivalent to the representation sphere associated to $1 + \sigma_{2^n} + \lambda_1 + 2\lambda_2 + \cdots + 2^{n-2}\lambda_{n-1}$. By an abuse of notation, we will write $\rho_{2^n} = 1 + \sigma_{2^n} + \lambda_1 + 2\lambda_2 + \cdots + 2^{n-2}\lambda_{n-1}$.
    
\item When $G=Q_8$, let $\sigma_i$, $\sigma_j$, and $\sigma_k$ be three $1$-dimensional sign representations of $Q_8$ with kernels $C_4\langle i \rangle$, $C_4\langle j \rangle$, and $C_4\langle k \rangle$, respectively. 
We also denote $\mathbb{H} = \mathbb{R} \oplus \mathbb{R}i \oplus \mathbb{R}j \oplus \mathbb{R}k$ as the quaternion algebra, which is a $4$-dimensional irreducible representation of $Q_8$ given by left multiplication. 

\item We will denote the $G$-slice spectral sequence by $G$-SliceSS and the $G$-homotopy fixed point spectral sequence by $G$-HFPSS. 

\end{itemize}
\section{Periodicities of \texorpdfstring{$EO_h(G)$}{text}}\label{sec:intP2}
In this section, we will establish periodicities for $EO_h(G)$ when $G$ is any finite subgroup of $\mathbb{G}_h$, proving \cref{thm:A}.  We will do so by first establishing periodicities for $EO_h(G)$ when $G$ is a finite 2-subgroup of $\mathbb{S}_h$ (\cref{thm:cyclicperiodicity} and \cref{thm:Q8periodicity}), and then extending to the general case.





\begin{defn}\label[defn]{defn:Gperiodicity}\rm
Let $R$ be a $G$-equivariant homotopy commutative ring spectrum, and let $V$ be a virtual $G$-representation. We say that $R$ is \textit{$V$-periodic} if there exists a $G$-equivariant map
    \[
    \varphi_V: S^V\longrightarrow R
    \]
    such that the composition map 
    \[
   \widetilde{\varphi}_V \colon  R\wedge S^V \xrightarrow{ \mathrm{id}\wedge  \varphi_V} R\wedge R\xrightarrow{m} R
    \]
    is a $G$-equivalence.
    \end{defn}

It is immediate from \cref{defn:Gperiodicity} that if $R$ is $V$-periodic, then there is an invertible element $\varphi_V \in \pi_V^G R$.  When the virtual representation $V$ is of the special form $(W-|W|)$ for a $G$-representation $W$, \cref{defn:Gperiodicity} coincides with the notion of an $R$-orientation for $W$ in \cite[Definition~3.3]{BBHS20}.

\begin{rem}\rm \label[rem]{rem:moduleEquivGivesHtpyMap}
Suppose $f: R \wedge S^V \to R$ is a $G$-equivalence of $R$-modules, then we can define $\varphi_V$ to be the composition
\[\varphi_V: S^V \simeq S^0 \wedge S^V \xrightarrow{u \wedge \text{id}} R \wedge S^V \xrightarrow{f} R.  \]
A direct verification shows that $\widetilde{\varphi}_V = f$.  Thus, the equivalence $f$ establishes $V$-periodicity for $R$ as defined in \cref{defn:Gperiodicity}. 
\end{rem}

It is straightforward to check that if $R$ is both $V_1$-periodic and $V_2$-periodic, then it is also $(V_1 + V_2)$-periodic. Moreover, if $R$ is $V$-periodic, then it is also $(-V)$-periodic by \cref{rem:moduleEquivGivesHtpyMap}. This shows that given a set of periodicities for $R$, any integer linear combination of them is again a periodicity for $R$. 

In this section, let $\SpG$ denote the category of equivariant orthogonal $G$-spectra following \cite{MM02}.

\begin{lem}\label[lem]{lemma:normperiodicity}
Let $G$ be a finite group, $H$ a subgroup of $G$, $V$ an $H$-representation, and $R$ a $G$-commutative ring spectrum in $\SpG$. If $i_H^* R$ is $V$-periodic, then $R$ is $W$-periodic, where $W =\mathrm{\Ind}_H^G V$ is the induced representation of $V$. 
\end{lem}
\begin{proof} 
Since $i^*_H R$ is $V$-periodic, there exists an $H$-equivariant map 
\[\varphi_V: S^V \longrightarrow i_H^*R \]
such that the $i_H^*R$-module map
\[
\widetilde{\varphi}_V \colon  i^*_H R\wedge S^V \rightarrow i^*_H R
\]
is an $H$-equivalence. 

By \cite[Theorem~1.1, Theorem~5.10]{BH20}, there is a homotopy functor $\RN(-)$ from the category of $i^*_H R$-modules to the category of $R$-modules, given by
\[
\RN(M):= R\wedge_{N_H^G i^*_H R} N_H^G(M).
\]
We can apply this functor to the $H$-equivalence $\widetilde{\varphi}_V$ to obtain the $G$-equivalence
\[
\RN(\widetilde{\varphi}_V): R \wedge S^W\rightarrow R.
\]
Here, we have used the fact that 
\[
\RN (i^*_H R\wedge S^V)=  R\wedge_{N_H^G i^*_H R} N_H^G (i^*_H R \wedge S^V)\simeq R \wedge S^W.
\]
Since the $G$-equivalence $\RN(\widetilde{\varphi}_V)$ is an $R$-module map, the discussion in \cref{rem:moduleEquivGivesHtpyMap} shows that $R$ is $W$-periodic. 
\end{proof}

Suppose now that $X$ is an $\mathbb{E}_\infty$-ring spectrum equipped with a $G$-action by $\mathbb{E}_\infty$-ring maps. The cofree completion $F(EG_+, X)$ is then a commutative ring spectrum in $\SpG$, as shown in \cite[Theorem~2.4]{HM17}. For these cofree spectra, one can establish $G$-periodicity by producing a $G$-equivariant map whose underlying nonequivariant map is an equivalence. In practice, as the following lemma shows, such a map can be found by analyzing the homotopy fixed-point spectral sequence (HFPSS).



\begin{lem}\label[lem]{lemma:pctoperiodicity}
Let $G$ be a finite group, $V$ a virtual $G$-representation, and $X$ a cofree $G$-commutative ring spectrum in $\textup{Ho}(\SpG)$. The $G$-spectrum $X$ is $V$-periodic if and only if, on the $\mathcal{E}_2$-page of the $G$-$\mathrm{HFPSS}$ for $X$, there is an invertible class $u$ with bidegree $(V,0)$ that is a permanent cycle. 
\end{lem}

\begin{proof}
For the ``if'' part, choose an element $\varphi_V \in \pi_V^G X$ in homotopy that is represented by $u$. Since $u$ is invertible on the $\mathcal{E}_2$-page of the HFPSS, it is also invertible in $(\pi^e_\star X)^G$ (the $0$-line of the $\HFPSS$) and hence invertible in $\pi^e_\star X$. This implies that the $G$-map
\[
\widetilde{\varphi}_V: X\wedge S^V \to X \wedge X \to X
\]
is a nonequivariant equivalence. Since $X$ is cofree, $\widetilde{\varphi}_V$ is a $G$-equivalence.  This shows that $X$ is $V$-periodic.

For the ``only if'' part, if $X$ is $V$-periodic, consider the invertible class $\varphi_V \in \pi_V^G X$ that induces the equivalence in \cref{defn:Gperiodicity}. This is represented by an invertible permanent cycle $u$ on the $\mathcal{E}_2$-page. Note that the restriction of $\varphi_V$ to $\pi^e_V(X)$ lies in the image of the map $(\pi^e_V(X))^G \to \pi^e_V(X)$. This shows that the bidegree of $u$ is $(V, 0)$.
\end{proof}




In practice, we can construct a family of invertible elements on the $\mathcal{E}_2$-page of the $G$-HFPSS for $X$ as follows. Let $V$ be an orientable $G$-representation of dimension $n$ (i.e. it factors through $SO(n)$). The key observation is that for a $G$-spectrum $X$, the homotopy group $\pi_{|V|}^e \Sigma^V X$ is isomorphic to $\pi_0^e X$ as a $G$-module. Therefore, the unit $1 \in \pi_0^e X$ corresponds to an orientation class
\[u_V \in (\pi_{|V|}^e \Sigma^V X)^G \cong H^0(G, \pi_{|V|}^e \Sigma^V X)\]
on the $\mathcal{E}_2$-page of the HFPSS. The orientation class $u_V$ has bidegree $(|V| - V, 0)$. Periodicities arising from the classes $u_V$ are discussed in \cite[Proposition~3.10]{BBHS20}.

In \cite[Section~3]{HHR16}, Hill--Hopkins--Ravenel have also defined orientation classes $u_V \in \pi_{|V|-V}^G H\underline{\mathbb{Z}}$, which lie on the $\mathcal{E}_2$-page of the slice spectral sequence of $\MUG$ and $\BPG$ ($G = \Cn$).  These classes map to the unit $1$ under the restriction map, which is an isomorphism:
\[
H^G_{|V|}(S^V, \underline{\mathbb{Z}}) \stackrel{\cong}{\longrightarrow}  H_{|V|}(S^{|V|}, \mathbb{Z}) \cong \mathbb{Z}.
\]
These orientation classes, for various orientable $G$-representations $V$, play key roles in computing the slice spectral sequences for $\MUG$ and $\BPG$.

When $X = \MUG$ or $\BPG$, the $u_V$ classes in the slice spectral sequence and the homotopy fixed point spectral sequence can be chosen compatibly. This is done by first choosing an orientation class $u_V$ on the $\mathcal{E}_2$-page of the slice spectral sequence. The image of this class under the map $\SliceSS \to \HFPSS$ will then be an orientation class in the homotopy fixed point spectral sequence.


We now focus on $E_h$, the Lubin--Tate spectrum associated with a height-$h$ formal group $\Gamma$ over $k$. By the works of Goerss--Hopkins--Miller and Lurie \cite{GH04, Lur18}, there is a unique $\mathbb{E}_\infty$ $\Gh$-action on $E_h$, where $\Gh$ is the corresponding Morava stabilizer group. From now on, we shall abuse notation and use $E_h$ to denote the $G$-spectrum $F(EG_+, E_h)$, where $G \subset \Gh$ is a finite subgroup.

We will first consider the case when $G$ is a cyclic $2$-subgroups of $\Sh$.  More specifically, when $h = 2^{n-1}m$ and $E_h$ admits a $\Cn$-action.

\begin{prop}\label[prop]{corollary:periodicityfromslicediffthm}
At height $h= 2^{n-1}m$, the $C_{2^{n}}$-spectrum $E_h$ is $(2^{h+1} - 2^{h+1} \sigma)$-periodic.
\end{prop}

\begin{proof}
By \cite[Theorem~1.1]{HS20}, there is a $C_{2^{n}}$-equivariant map from 
$N_{C_2}^{C_{2^{n}}}(\bar{v}_h)^{-1}BP^{(\!(C_{2^{n}})\!)}$ to $E_h$, which induces a map between their corresponding $C_{2^{n}}$-HFPSS.  Consider the following maps of spectral sequences 
\[\begin{tikzcd}
\Cn\text{-}\SliceSS\left(MU^{(\!(C_{2^{n}})\!)}\right) \ar[r] & \Cn\text{-}\SliceSS\left(N_{C_2}^{C_{2^{n}}}(\bar{v}_h)^{-1}BP^{(\!(C_{2^{n}})\!)}\right) \ar[d] \\
& \Cn\text{-}\HFPSS \left(N_{C_2}^{C_{2^{n}}}(\bar{v}_h)^{-1}BP^{(\!(C_{2^{n}})\!)}\right) \ar[r] & \Cn\text{-}\HFPSS \left(E_h\right) 
\end{tikzcd}\]
By the Slice Differentials Theorem \cite[Theorem~9.9]{HHR16}, the class $u_{2\sigma}^{2^{h}}$ is a $2^{n} \cdot (2^{h+1}-1)$-cycle in the $\Cn$-$\SliceSS$ for $MU^{(\!(C_{2^{n}})\!)}$.  Therefore, this class is also a $2^{n} \cdot (2^{h+1}-1)$-cycle in the $\Cn$-$\SliceSS$ for $N_{C_2}^{C_{2^{n}}}(\bar{v}_h)^{-1}BP^{(\!(C_{2^{n}})\!)}$, and furthermore in the $\Cn$-$\HFPSS$ for $N_{C_2}^{C_{2^{n}}}(\bar{v}_h)^{-1}BP^{(\!(C_{2^{n}})\!)}$ by naturality. 

By \cite[Theorem~6.1]{DLS2022}, the $\Cn$-HFPSS for $N_{C_2}^{C_{2^{n}}}(\bar{v}_h)^{-1}BP^{(\!(C_{2^{n}})\!)}$ admits a strong vanishing line of filtration $(2^{h+n}-2^{n}+1)$. Since $2^{n} \cdot (2^{h+1}-1) \geq 2^{h+n}-2^{n}+1$, the class $u_{2\sigma}^{2^{h}}$ is a permanent cycle in the $\Cn$-HFPSS for $N_{C_2}^{C_{2^{n}}}(\bar{v}_h)^{-1}BP^{(\!(C_{2^{n}})\!)}$.  Therefore, $u_{2\sigma}^{2^{h}}$ is a permanent cycle in the $C_{2^{n}}$-HFPSS for $E_h$. It follows that the spectrum $E_h$ is $(2^{h+1} - 2^{h+1} \sigma)$-periodic by \cref{lemma:pctoperiodicity}.
\end{proof}

\begin{rem}\rm
The statement and proof of \cref{corollary:periodicityfromslicediffthm} hold for all forms of $E_h$ because both the Real orientation theorem in \cite{HS20} and the strong vanishing line theorem in \cite{DLS2022} apply to all forms of Lubin--Tate theories.
\end{rem}
    
\begin{rem}\rm
The positive integer $m$ in \cref{corollary:periodicityfromslicediffthm} can also be even. For example, when $n = 1$ and $m = 2$, \cref{corollary:periodicityfromslicediffthm} implies that the $C_2$-spectrum $E_2$ is $(8 - 8\sigma)$-periodic.
\end{rem}




\begin{thm}\label[thm]{thm:cyclicperiodicity}
At height $h=2^{n-1} m$, the spectrum $EO_h(\Cn)$ is $2^{h+n+1}$-periodic.  
\end{thm}

\begin{proof}
We will prove the claim by using induction on $n$, which corresponds to the group $\Cn$ acting on $E_h$.  The base case, when $n = 1$, follows from the complete computation of $EO_h(C_2) = E_h^{hC_2}$ in \cite[Theorem~1.4]{HS20}.  

For the induction step, assume that the claim is true for $k < n$, and for all $h$.  In particular, this means that as a $\Ck$-spectrum, $i_{\Ck}^* E_h$ is $2^{h+k+1}$-periodic.  By \cref{lemma:normperiodicity}, the $\Ckplusone$-spectrum $i_{\Ckplusone}^* E_h$ is 
\[\ind_{C_{2^k}}^{C_{2^{k+1}}}(2^{h+k+1}) = 2^{h+k+1} + 2^{h+k+1} \sigma_{2^{k+1}}\]
periodic.  By \cref{corollary:periodicityfromslicediffthm}, $i_{\Ckplusone}^* E_h$ is also $(2^{h+1} - 2^{h+1} \sigma_{2^{k+1}})$-periodic, and therefore 
\[(2^{h+1} - 2^{h+1} \sigma_{2^{k+1}}) \cdot 2^k= 2^{h+k+1} - 2^{h+k+1}\sigma_{2^{k+1}}\]
periodic.  Combining these two periodicities shows that $i_{\Ckplusone}^* E_h$ is 
\[(2^{h+k+1} + 2^{h+k+1} \sigma_{2^{k+1}}) + (2^{h+k+1} - 2^{h+k+1}\sigma_{2^{k+1}}) = 2^{h+(k+1)+1} \]
periodic.  This completes the induction step. 
\end{proof}

We will now consider the case when $h = 4k + 2$ and $E_h$ admits a $Q_8$-action. To obtain periodicities for the $Q_8$-fixed points of $E_h$, we will first introduce some real representations of $Q_8$. 

Let $\sigma_i$, $\sigma_j$, and $\sigma_k$ denote the three $1$-dimensional representations of $Q_8$ with kernels $C_4\langle i \rangle$, $C_4\langle j \rangle$, and $C_4\langle k \rangle$, respectively.  Here, $C_4\langle i \rangle$, $C_4\langle j \rangle$, and $C_4\langle k \rangle$ are the $C_4$-subgroups generated by $i$, $j$, and $k$, respectively.  We will also denote $\mathbb{H} = \mathbb{R} \oplus \mathbb{R}i \oplus \mathbb{R}j \oplus \mathbb{R}k$ as the quaternion algebra, which is a $4$-dimensional irreducible real representation of $Q_8$, with $Q_8$ acting via left multiplication. We will need the following facts about these representations:

\begin{multicols}{2}
\begin{enumerate} \item $\ind_{C_2}^{Q_8}(1) = 1 + \sigma_i + \sigma_j + \sigma_k$ \item $\ind_{C_2}^{Q_8}(\sigma_2) = \mathbb{H}$ \item $\ind_{C_4\langle i \rangle}^{Q_8}(1) = 1 + \sigma_i$ \item $\ind_{C_4\langle i \rangle}^{Q_8}(\sigma_4) = \sigma_j + \sigma_k$
\item $\ind_{C_4\langle j \rangle}^{Q_8}(1) = 1 + \sigma_j$ \item $\ind_{C_4\langle j \rangle}^{Q_8}(\sigma_4) = \sigma_k + \sigma_i$
\item $\ind_{C_4\langle k \rangle}^{Q_8}(1) = 1 + \sigma_k$ \item $\ind_{C_4\langle k \rangle}^{Q_8}(\sigma_4) = \sigma_i + \sigma_j$
\end{enumerate}
\end{multicols}


\begin{thm}\label[thm]{thm:Q8periodicity}
At height $h=4\ell+2$, the spectrum $EO_h(Q_8)$ is $2^{h+4}$-periodic.
\end{thm}

\begin{proof}
By \cref{corollary:periodicityfromslicediffthm}, $E_{h}$ is $2^{h+1}(1-\sigma_2)$-periodic as a $C_2$-spectrum.  Applying \cref{lemma:normperiodicity} shows that $E_h$ has periodicity
\[
\ind_{C_2}^{Q_8}\left(2^{h+1}(1-\sigma_2)\right) = 2^{h+1}(1+\sigma_i+\sigma_j+\sigma_k-\mathbb{H})
\]
as a $Q_8$-spectrum.  

By \cite[Theorem~1.4]{HS20}, $E_{h}$ is $(1+\sigma_2)$-periodic as a $C_2$-spectrum. Applying \cref{lemma:normperiodicity} to this periodicity shows that as a $Q_8$-spectrum, $E_h$ has periodicity 
\[\ind_{C_2}^{Q_8}\left(1+\sigma_2\right)=1 + \sigma_i + \sigma_j + \sigma_k + \mathbb{H} .\]

Consider the $C_4$-subgroup $C_4\langle i \rangle$. By \cref{corollary:periodicityfromslicediffthm} again, $E_{h}$ is $2^{h+1}(1-\sigma_4)$-periodic as a $C_4\langle i \rangle$-spectrum. Applying \cref{lemma:normperiodicity} shows that $E_h$ has periodicity 
\[\ind_{C_4\langle i \rangle}^{Q_8}\left(2^{h+1}(1-\sigma_4)\right) = 2^{h+1}(1+\sigma_i - \sigma_j - \sigma_k)\]
as a $Q_8$-spectrum.  Applying the same argument to the $C_4$-subgroups $C_4\langle j \rangle$ and $C_4 \langle k \rangle$ shows we have the following three $Q_8$-periodicities: 
\begin{enumerate}
    \item $2^{h+1}(1+\sigma_i-\sigma_j-\sigma_k)$;
    \item $2^{h+1}(1-\sigma_i+\sigma_j-\sigma_k)$;
    \item $2^{h+1}(1-\sigma_i-\sigma_j+\sigma_k)$.
\end{enumerate}
The sum of these three periodicities is $2^{h+1}(3 - \sigma_i - \sigma_j - \sigma_k)$.  By combining these periodicities, we obtain the desired integer-graded periodicity for $EO_h(Q_8) = E_h^{hQ_8}$: 
\[2^{h+1}(1+\sigma_i+\sigma_j+\sigma_k-\mathbb{H}) + 2^{h+1}\cdot (1 + \sigma_i + \sigma_j + \sigma_k + \mathbb{H}) + 2 \cdot 2^{h+1}(3 - \sigma_i - \sigma_j - \sigma_k) = 2^{h+4}.\]
\end{proof}

\begin{rem}\rm
The periodicities established in \cref{thm:cyclicperiodicity} and \cref{thm:Q8periodicity} apply to all forms of $E_h$, whereas previously, periodicities were known only for certain specific forms. For instance, when $n = 3$ and $m = 1$, the Periodicity Theorem \cite[Theorem~1.7]{HHR16} of Hill, Hopkins, and Ravenel shows that for a particular form of $E_4$, the $C_8$-fixed points of the detection spectrum $\Omega$ are $256$-periodic.  When $n = 2$ and $m = 2$, Hill, Shi, Wang, and Xu proved that the $C_4$-spectrum $D^{-1}\BPCfour\langle 2\rangle$ is $384$-periodic \cite[Theorem~1.3]{HSWX23}, which implies that the $C_4$-fixed points of a specific form of $E_4$ are $2^{4+2+1} = 128$-periodic.  One can also combine the Hill--Hopkins--Ravenel Periodicity Theorem with the work of Hahn--Shi \cite{HS20} and Beaudry--Hill--Shi--Zeng \cite{BHSZ21} to deduce the periodicity of $E_h^{h\Cn}$ for a specific $\Cn$-equivariant model of $E_h$ (see \cite[Theorem~4.3.5]{Mor24}).
\end{rem}


We will now use \cref{thm:cyclicperiodicity} and \cref{thm:Q8periodicity} to establish periodicities for $EO_h(G)$ when $G$ is any finite subgroup of $\mathbb{G}_h$ (\cref{thm:A}).

\begin{lem}\label[lem]{lem:coprime}
Let $G_0$ be a finite subgroup of $\mathbb{S}_h$, and let $H \subset G_0$ be a Sylow $2$-subgroup with index $[G_0: H] = \ell$. Then $EO_h(H)$ is $P$-periodic if and only if $EO_h(G_0)$ is $(\ell P)$-periodic.
\end{lem}

\begin{proof}
For the ``if'' part, if $E_h^{hG_0}$ is $(\ell P)$-periodic, then the $G_0$-spectrum $E_h$ is also $(\ell P)$-periodic. Therefore, its restriction $i_H^* E_h$ is $(\ell P)$-periodic, which implies that $E_h^{hH}$ is $(\ell P)$-periodic. By \cref{thm:cyclicperiodicity} and \cref{thm:Q8periodicity}, $E_h^{hH}$ always admits a periodicity of some power of $2$. Since $\ell$ is coprime to $2$, this forces $E_h^{hH}$ to be $P$-periodic.

For the ``only if'' part, if $E_h^{hH}$ is $P$-periodic, then by \cref{lemma:pctoperiodicity}, there exists an invertible permanent cycle $u$ with bidegree $(P, 0)$ in the $H$-HFPSS for $E_h$. Recall from \cite[Table~3.14]{BB20} that $H$ is normal in $G_0$. In particular, the quotient group $G_0/H$ acts on $(\pi_*^e E_h)^H=H^0(H, \pi_*^eE_h)$.  Consider the product $\prod_{\bar{g} \in G_0/H} \bar{g}u$. Since $u$ is an invertible permanent cycle, each $\bar{g}u$ is also an invertible permanent cycle, so their product remains an invertible permanent cycle on the $\mathcal{E}_2$-page of the $H$-HFPSS for $E_h$. Furthermore, it is also an invertible element in $(\pi_*^e E_h)^{G_0} = H^0(G_0, \pi_*^e E_h)$, the 0-line of the $\mathcal{E}_2$-page of the $G_0$-HFPSS for $E_h$. 

Now, consider the composition
\[\begin{tikzcd}
G_0\text{-}\HFPSS(E_h) \ar[r, "res"] &H\text{-}\HFPSS(E_h) \ar[r, "tr"]& G_0\text{-}\HFPSS(E_h),
\end{tikzcd}\]
which is multiplication by $|G_0/H| = \ell$. Since the restriction map sends $\prod_{\bar{g} \in G_0/H} \bar{g}u$ to $\prod_{\bar{g} \in G_0/H} \bar{g}u$, the transfer map must send $\prod_{\bar{g} \in G_0/H} \bar{g}u$ to $\ell \cdot \prod_{\bar{g} \in G_0/H} \bar{g}u$.  Since $\ell$ is coprime to 2 and we are working 2-locally, $\ell \cdot \prod_{\bar{g} \in G_0/H} \bar{g}u$ is an invertible permanent cycle on the $\mathcal{E}_2$-page of the $G_0$-HFPSS for $E_h$, with bidegree $(\ell P, 0)$.  By \cref{lemma:pctoperiodicity}, this establishes the desired $(\ell P)$-periodicity for $E_h^{G_0}$.
\end{proof}

The following lemma shows that taking fixed points with respect to the Galois group does not affect the periodicity.

\begin{lem}\label[lem]{lem:Galois}
Let $G \subset \mathbb{G}_h$ be a finite subgroup, and let $G_0 = G \cap \mathbb{S}_h$. Then $EO_h(G)$ and $EO_h(G_0)$ have the same periodicity.      
\end{lem}

\begin{proof}
First note that the quotient group $G/G_0$ can be identified as a subgroup of the Galois group $\Gal(k/\mathbb{F}_2)$ through the inclusion $G \rightarrow \mathbb{G}_h$. By \cite[Lemma 1.32, Lemma 1.37, Remark 1.39]{BG18}, the $G$-HFPSS for $E_h$ is a base change
of the $G_0$-HFPSS for $E_h$ (see also \cite[Lemma~2.2.7]{BGH2022}). More precisely, let $k'=k^{G/G_0}$.  Then there is an isomorphism 
\[\WW(k) \otimes_{\WW(k')} H^*(G, \pi_*^e E_h) \stackrel{\cong}{\longrightarrow} H^*(G_0, \pi_*^e E_h)\]
of the corresponding $\mathcal{E}_2$-pages, and all the differentials in the $G_0$-HFPSS for $E_h$ are the $\mathbb{W}(k)$-linear extensions of those in the $G$-HFPSS for $E_h$. Therefore, if $u$ is an invertible permanent cycle in $G$-HFPSS that produces a periodicity for $E_h^{hG}$, then it also corresponds to an invertible permanent cycle in the $G_0$-HFPSS that produces the same periodicity for $E_h^{hG_0}$.  The converse also holds.

\end{proof}

\begin{proof}[Proof of \cref{thm:A}]
By \cref{lem:Galois}, $EO_h(G)$ and $EO_h(G_0)$ have the same periodicity.  Let $H$ be a 2-Sylow subgroup of $G_0$ with index $\ell$.  By \cref{thm:cyclicperiodicity} and \cref{thm:Q8periodicity}, $EO_h(H)$ is $P(h, H)$-periodic, where $P(h, H)$ is defined as in \cref{defn:PhG}.  The conclusion now follows by applying \cref{lem:coprime}, which shows that $EO_h(G_0)$ is $(\ell P(h, H))$-periodic.  
\end{proof}

\begin{rem}\rm \label[rem]{rem:PreviousResults}
The $P(h, G)$-periodicities for $EO_h(G)$ stated in \cref{thm:A} are sharp in all previously known cases where explicit computations have determined the exact periodicities of $EO_h(G)$:
\vspace{-0.15in}
{\begin{multicols}{2}
\begin{itemize}
\item $EO_1(C_2)$: 8 (\cite{Bott59, Atiyah66KR})
\item $EO_2(C_2)$: 16 (\cite{MahowaldRezk09}) 
\item $EO_2(C_4)$: 32 (\cite{HHR16, BO16}))
\item $EO_2(C_6)$: 48 (\cite{MahowaldRezk09})
\item $EO_2(Q_8)$: 64 (\cite{DFHH14, DKLLW24})
\item $EO_2(G_{24})$: 192 (\cite{Bau08,DFHH14})
\columnbreak
\item $EO_2(G_{48})$: 192 (\cite{Bau08,DFHH14})
\item $EO_4(C_2)$: 64 (\cite{HS20, LSWX2019})
\item $EO_4(C_4)$: 128 (\cite{HSWX23})
\item $EO_4(C_8)$: 256 (\cite{HHR16,LWX24})
\item $EO_4(C_{12})$: 384 (\cite{HSWX23})
\item $EO_h(C_2)$: $2^{h+2}$ (\cite{HS20, LSWX2019})
\end{itemize}
\end{multicols}}
\end{rem}

\section{The \texorpdfstring{$RO(G)$}{text}-periodicities of \texorpdfstring{$E_h$}{text}}\label{sec:lattice}

For $p = 2$ and $h \geq 1$, let $E_h = E(\bar{\mathbb{F}}_p, F)$ be the Lubin--Tate theory associated with a chosen height-$h$ formal group $F$ over $\bar{\mathbb{F}}_p$.  In this section, we will analyze the $RO(G)$-periodicities of $E_h$ and prove \cref{thm:B} (\cref{theorem:ROGperiodicity}) and \cref{thm:C} (\cref{thm:finiteComplexity}).  By \cref{lem:Galois}, it suffices to restrict our attention to the situation when $G$ is a finite subgroup of $\mathbb{S}_h := \Aut_{\bar{\mathbb{F}}_p}(F)$.  

\begin{defn}\rm\label[defn]{defn:lattice}
For $h \geq 1$ and $G \subset \mathbb{S}_h$ a finite subgroup, the \textit{$RO(G)$-periodicity lattice} $\La_{h,G}$ of $E_h$ is the free abelian subgroup of $RO(G)$ containing all $V$ such that $E_h$ is $V$-periodic.
\end{defn}

The periodicity lattice $\La_{h, G}$ encodes the complexity of the $RO(G)$-graded homotopy groups $\pi_\star^G E_h$.  More specifically, to compute all the $RO(G)$-graded homotopy groups $\pi^G_\star E_h$, it suffices to compute the homotopy groups of the grading in $RO(G)/\La_{h, G}$.  For example, $\La_{1, C_2}=\mathbb{Z}\langle 1+\sigma, 4-4\sigma \rangle$, and $RO(C_2)/\La_{1, C_2} = \mathbb{Z}/8$.  This is the complexity of $\pi_\star^{C_2}E_1$, as the computation of the grading in $\mathbb{Z}/8$ completely determines all the $RO(C_2)$-graded homotopy groups of $E_1$. 



\begin{rem}\rm
The $RO(G)$-periodicity lattice $\La_{h,G}$ is closely related to $K(h)$-local Picard groups. There is a group homomorphism $J: RO(G) \rightarrow \Pic_G(E_h)$ defined by sending a virtual $G$-representation $V$ to $\Sigma^V E_h$ \cite[Section~3]{BBHS20}. Here, $\Pic_G(E_h)$ is the Picard group of the category of all $K(h)$-local $E_h$-modules in $\text{Sp}^G$. The kernel of $J$ is precisely the lattice $\La_{h,G}$.
\end{rem}

The following proposition shows that $\La_{h,G}$  depends only on the pair $(h, G)$ and is independent of the choice of $E_h$ and the inclusion $G \hookrightarrow \Sh$. 

\begin{prop}\label[prop]{prop:modelfree}
Let $E(k, \Gamma)$ be the Lubin–Tate theory associated with a height-$h$ formal group $\Gamma$ over a finite field $k$ of characteristic $p$, and $G$ a finite subgroup of $\Aut_k(\Gamma)$.  For any $V \in RO(G)$, the $G$-spectrum $E(k, \Gamma)$ is $V$-periodic if and only if $V \in \La_{h, G}$. 
\end{prop}

\begin{proof}
Since $G \subset \Aut_k(\Gamma) \subset \Aut_{\bar{\mathbb{F}}_p}(\Gamma)$, the field extension $k \rightarrow \bar{\mathbb{F}}_p$ induces a $G$-equivariant map $E(k, \Gamma) \rightarrow E(\bar{\mathbb{F}}_p, \Gamma)$.  Furthermore, since $G$ is a subgroup of $\Aut_k(\Gamma)$, the Galois group $\Gal(\bar{\mathbb{F}}_p/k)$ acts trivially on $G$ in $\mathbb{G}_h = \Aut_{\bar{\mathbb{F}}_p}(\Gamma) \rtimes \Gal(\bar{\mathbb{F}}_p/k)$, and $G \times \Gal(\bar{\mathbb{F}}_p/k)$ is a subgroup of $\mathbb{G}_h$.  Taking the homotopy fixed points with respect to $\Gal(\bar{\mathbb{F}}_p/k)$ shows that there is a $G$-equivalence $E(\bar{\mathbb{F}}_p, \Gamma)^{h\Gal(\bar{\mathbb{F}}_p/k)} \simeq E(k, \Gamma)$.  Therefore, $E(k, \Gamma)$ and $E(\bar{\mathbb{F}}_p, \Gamma)^{h\Gal(\bar{\mathbb{F}}_p/k)}$ have the same $RO(G)$-periodicity. 

By \cref{lem:Galois}, we also know that $E(\bar{\mathbb{F}}_p, \Gamma)^{h\Gal(\bar{\mathbb{F}}_p/k)}$ and $E(\bar{\mathbb{F}}_p, \Gamma)$ have the same $RO(G)$-periodicity. Since formal groups with the same height are isomorphic over $\bar{\mathbb{F}}_p$, $\Gamma$ and $F$ are isomorphic, and an isomorphism between them induces a $G$-equivalence $E(\bar{\mathbb{F}}_p, \Gamma) \simeq E(\bar{\mathbb{F}}_p, F) = E_h$.  Here, the $G$-action on $E(\bar{\mathbb{F}}_p, \Gamma)$ comes from the inclusion $i_\Gamma: G \hookrightarrow \Aut_k(\Gamma) \hookrightarrow \Aut_{\bar{\mathbb{F}}_p}(\Gamma) = \mathbb{S}_h$, and may differ from the chosen injection $i_F: G \hookrightarrow \mathbb{S}_h$.  However, since two isomorphic subgroups of $\mathbb{S}_h$ are conjugate to each other in $\mathbb{S}_h$ \cite[Corollary~1.30]{Buj2012}, we indeed get a $G$-equivalence, and therefore $E(\bar{\mathbb{F}}_p, \Gamma)$ and $E_h$ have the same $RO(G)$-periodicity.  It follows that $E(k, \Gamma)$ and $E_h$ have the same $RO(G)$-periodicity.  
\end{proof}


\begin{rem} \label[rem]{rem:LaMackeyStructure} \rm
The lattice $\La_{h,G}$ has a richer structure than that of abelian groups.  It has a Mackey functor structure, which we denote by $\underline{\La}_{h,G}$, obtained by considering all abelian groups $\La_{h,H}$ for $H \subset G$.  To explain this Mackey functor structure, recall that the real representation ring $RO(G)$ has a Mackey functor structure, denoted by $\underline{RO}$, as follows: $\underline{RO}(G/H) = RO(H)$, the restriction map is given by the restriction of representations, and the transfer map is given by the induction of representations. For subgroups $K \subset H \subset G$, if $i_H^* E_h$ is $V$-periodic for some $V \in RO(H)$, then $i_K^* E_h$ is $\Res^H_K(V)$-periodic by definition; if $i_K^* E_h$ is $W$-periodic for some $W \in RO(K)$, then $i_H^* E_h$ is $\ind_K^H(W)$-periodic by \cref{lemma:normperiodicity}. Therefore, the restriction and transfer (induction) maps are closed in $\underline{\La}_{h,G} \subset \underline{RO}$, and $\underline{\La}_{h,G}$ inherits a Mackey functor structure from $\underline{RO}$.    
\end{rem}




By \cref{lem:coprime}, it suffices to restrict our attention to the cases where $G$ is $\Cn$ or $Q_8$ when studying the lattice $\La_{h, G}$. In general, determining $\La_{h,G}$ completely for arbitrary $h$ and $G$ is a difficult task. However, the Mackey functor structure of $\underline{\La}_{h,G}$ can be used to identify an important sublattice of $\La_{h,G}$. To describe this sublattice, we first recall some notation for elements in $RO(G)$.

When $G = \Cn$, the representation spheres associated with the $2$-dimensional representations corresponding to rotations by $(\frac{2k\pi}{2^{n}})$ and $(\frac{2k'\pi}{2^{n}})$ are $G$-equivalent 2-locally if the integers $k$ and $k'$ have the same $2$-adic valuation.  Therefore, 2-locally, $RO(\Cn)$ is generated by the irreducible representations \{$1, \sigma_{2^{n}}, \lambda_1, \dots, \lambda_{n-1}$\}, where $\lambda_i$ is the $2$-dimensional real representation corresponding to the rotation by $(\frac{\pi}{2^i})$.  Note that $\lambda_0 = 2 \sigma_{2^n}$.  We will also denote the regular representation of $\Cn$ by $\rho_{2^n}$, which is equal to 
\[\rho_{2^n}=1 + \sigma_{2^n} + \lambda_1 + 2\lambda_2 + \cdots + 2^{n-2} \lambda_{n-1}.\]
The real representations associated to $G = Q_8$ have already been discussed in the paragraph before \cref{thm:Q8periodicity}. 


\begin{thm}\label[thm]{theorem:ROGperiodicity}
As a $G$-spectrum, $E_h$ has the following $RO(G)$-graded periodicities: 
\begin{enumerate}
\item  When $(h, G) = (2^{n-1} m, \Cn)$: 
\begin{enumerate}[label=(\alph*)]
        \item $\rho_{2^n} = 1 + \sigma + \lambda_1 + 2 \lambda_2 + \cdots + 2^{n-2}\lambda_{n-1}$; 
        \item $2^{2^{n-i}m +n-i+1}-2^{2^{n-i}m+n-i}\lambda_{n-i}$, $1\leq i\leq n$.
\end{enumerate}
\item When $(h, G) = (4\ell+2, Q_8)$: 
\begin{enumerate}[label=(\alph*)]
\item $\rho_{Q_8} = 1+\sigma_i+\sigma_j+\sigma_k+\mathbb{H}$;
\item $2^{2\ell+2}+2^{2\ell+2}\sigma_i-2^{2\ell+2}\sigma_j-2^{2\ell+2}\sigma_k$;\\
$2^{2\ell+2}-2^{2\ell+2}\sigma_i+2^{2\ell+2}\sigma_j-2^{2\ell+2}\sigma_k$;\\
$2^{2\ell+2}-2^{2\ell+2}\sigma_i-2^{2\ell+2}\sigma_j+2^{2\ell+2}\sigma_k$;
\item $2^{h+2}+2^{h+2}\sigma_i-2^{h+1}\mathbb{H}$; \\
$2^{h+2}+2^{h+2}\sigma_j-2^{h+1}\mathbb{H}$;\\
$2^{h+2}+2^{h+2}\sigma_k-2^{h+1}\mathbb{H}$.
\end{enumerate}
\end{enumerate}
\end{thm}

\begin{proof}
For $(1a)$, \cite[Theorem~1.2]{HS20} shows that $i_{C_2}^*E_h$ is $\rho_2$-periodic. Therefore, by \cref{lemma:normperiodicity}, $E_h$ has periodicity $\ind_{C_2}^{\Cn}(\rho_2)= \rho_{2^{n}}$ when considered as a $\Cn$-spectrum. 

To prove the periodicities in $(1b)$, we will use \cref{prop:modelfree} and the specific form of $E_h$ in \cite{BHSZ21} that are constructed based on the generators for $\pi_{*\rho_2}^{C_2} \BPCn$ in \cite{HHR16}.  Recall that for $1 \leq i \leq n$, a collection of generators $\left\{\bar{t}_k^{\Ci} \, |\,k\geq 1 \right\}$ in $\pi_{*\rho_2}^{C_2} \BPCi$ are constructed in \cite{HHR16} (see also \cite{BHSZ21}).  For each $k\geq 1$, we will also use $\bar{t}^{\Ci}_k$ to denote the image of $\bar{t}^{\Ci}_k$ in $\pi_{*\rho_2}^{C_2} \BPCn$ along the left unit map $i_{C_2}^* \BPCi \rightarrow i_{C_2}^*\BPCn$. \cite[Theorem~1.8]{BHSZ21} constructs a form of $E_h$ that admits a $\Cn$-equivariant map
\[D^{-1}\BPCn \longrightarrow E_h,\]
where $D = \prod_{i=1}^{n} N_{C_2}^{\Cn} (\bar{t}^{C_{2^i}}_{2^{n-i} m})$. 

We will prove the periodicities in $(1b)$ by using induction on $G = \Cn$.  Suppose we have proven the claim for $G = C_{2^{n-1}}$ and for all heights $h$.  By \cite[Theorem~8.3]{MSZ24}, it is a consequence of the Transchromatic Isomorphism Theorem that for $V \in RO(\Cn/C_2)$, $u_V$ is a permanent cycle in the $RO(\Cn)$-graded homotopy fixed point spectral sequence of $E_{h}$ if and only if it is a permanent cycle in the $RO(\Cn/C_2)$-graded homotopy fixed point spectral sequence of $E_{h/2}$.  Here, $V$ is treated as an element in $RO(\Cn)$ through pullback along the map $\Cn \to \Cn/C_2$.  In other words, $(|V| - V)$ is a $RO(\Cn)$-periodicity for $E_h$ if and only if $(|V| - V)$ is a $RO(\Cnminusone)$-periodicity for $E_{h/2}$ \cite[Theorem~B]{MSZ24}.  

By the induction hypothesis, the following are $RO(\Cn/C_2)$-periodicities of $E_{h/2}$: 
\[2^{2^{n-1-i}m +n-i}-2^{2^{n-1-i}m+n-1-i}\lambda_{n-1-i}, \,\,\, 1\leq i\leq n-1\]
and therefore they are also $RO(\Cn)$-periodicities of $E_h$ by the discussion above. Note that these periodicities can also be written as 
\[2^{2^{n-i}m +n-i+1}-2^{2^{n-i}m+n-i}\lambda_{n-i}, \,\,\,2\leq i\leq n.\]
It remains to prove that when $i = 1$, the element $(2^{2^{n-1}m +n}-2^{2^{n-1}m+n-1}\lambda_{n-1})$ is a $RO(\Cn)$-periodicity for $E_h$.  To do so, consider the $\Cnminusone$-spectrum $i_{\Cnminusone}^*E_h$, which corresponds to the pair $(h, G) = (2^{n-2}\cdot(2m), \Cnminusone)$.  By the induction hypothesis, $i_{\Cnminusone}^*E_h$ has $RO(\Cnminusone)$-periodicity $(2^{2^{n-1}m + n-1} - 2^{2^{n-1}m + n-2}\lambda_{n-2})$. Applying \cref{lemma:normperiodicity}, $E_h$ has $RO(\Cn)$-periodicity 
\[\Ind_{\Cnminusone}^{\Cn}\left(2^{2^{n-1}m + n-1} - 2^{2^{n-1}m + n-2} \lambda_{n-2} \right) =  2^{2^{n-1}m + n-1} + 2^{2^{n-1}m + n-1} \sigma - 2^{2^{n-1}m + n-1} \lambda_{n-1}.\]
Combining this with the periodicity $2^{m+1} - 2^m \lambda_0 = 2^{m+1} - 2^{m+1}\sigma$ that we have already obtained, we get the desired $(2^{2^{n-1}m + n}- 2^{2^{n-1}m + n-1}\lambda_{n-1})$-periodicity.  This completes the induction step.

For $(2a)$, the proof for the $RO(Q_8)$-periodicities is similar to that of (1).  For $(a)$, since $i_{C_2}^*E_h$ is $\rho_2$-periodic, $E_h$ has periodicity 
\[\ind_{C_2}^{Q_8}(\rho_2)= \rho_{Q_8} = 1 + \sigma_i + \sigma_j + \sigma_k+\mathbb{H}\]
when considered as a $Q_8$-spectrum.

For $(2b)$, we have already shown in part $(b)$ of $(1)$ that when $n = 2$, $m = 2\ell+1$, and $i = 2$, $i_{C_4}^*E_h$ is $(2^{2\ell+2}-2^{2\ell+2}\sigma_4)$-periodic.  The quaternion group $Q_8$ has three cyclic $C_4$-subgroups: $C_4\langle i \rangle, C_4\langle j\rangle$ and $C_4\langle k\rangle$. Applying \cref{lemma:normperiodicity} shows that $E_h$ has the following $RO(Q_8)$-periodicities:
\begin{align*}
\ind_{C_4\langle i \rangle}^{Q_8}(2^{2\ell+2}-2^{2\ell+2}\sigma_4) &= 2^{2\ell+2}+2^{2\ell+2}\sigma_i-2^{2\ell+2}\sigma_j-2^{2\ell+2}\sigma_k, \\
\ind_{C_4\langle j \rangle}^{Q_8}(2^{2\ell+2}-2^{2\ell+2}\sigma_4) &= 2^{2\ell+2}+2^{2\ell+2}\sigma_j-2^{2\ell+2}\sigma_i-2^{2\ell+2}\sigma_k, \\
\ind_{C_4\langle k \rangle}^{Q_8}(2^{2\ell+2}-2^{2\ell+2}\sigma_4) &= 2^{2\ell+2}+2^{2\ell+2}\sigma_k-2^{2\ell+2}\sigma_i-2^{2\ell+2}\sigma_j. 
\end{align*}



For $(2c)$, we have shown in part $(b)$ of $(1)$ that when $n = 2$, $m = 2\ell+1$, and $i= 1$, $i_{C_4}^* E_h$ is $(2^{h+2} - 2^{h+1}\lambda_1)$-periodic. Applying \cref{lemma:normperiodicity} shows that $E_h$ has the following $RO(Q_8)$-periodicities:
\begin{align*}
\ind_{C_4\langle i \rangle}^{Q_8}(2^{h+2}-2^{h+1}\lambda_1)&=2^{h+2}+2^{h+2}\sigma_i-2^{h+1}\mathbb{H},  \\
\ind_{C_4\langle j \rangle}^{Q_8}(2^{h+2}-2^{h+1}\lambda_1)&=2^{h+2}+2^{h+2}\sigma_j-2^{h+1}\mathbb{H}, \\
\ind_{C_4\langle k \rangle}^{Q_8}(2^{h+2}-2^{h+1}\lambda_1)&=2^{h+2}+2^{h+2}\sigma_k-2^{h+1}\mathbb{H}. 
\end{align*}

\end{proof}


%

\begin{rem}\rm
The $RO(G)$-periodicities of $E_h$ in \cref{theorem:ROGperiodicity} can be used to recover the integer periodicities of $EO_h(G)$ established in \cref{thm:cyclicperiodicity} and \cref{thm:Q8periodicity}.  When $(h, G) = (2^{n-1}m, \Cnminusone)$, we have 

\begin{align*}
& \, \, 2^{2^{n-1}m+1} \cdot \rho_{2^n} + \sum_{i =1}^n 2^{2^{n-1}m - 2^{n-i}m} \cdot (2^{2^{n-i}m +n-i+1}-2^{2^{n-i}m+n-i}\lambda_{n-i}) \\
= & \, \,  2^{2^{n-1}m+1} + \sum_{i = 1}^n 2^{2^{n-1}m + n-i+1} \\
= & \, \, 2^{2^{n-1}m + n+1} = 2^{h+n+1}. 
\end{align*}

When $(h, G) = (4\ell+2, Q_8)$, we have 
\begin{align*}
&\,\,2^{4\ell+3} \cdot (1+\sigma_i+\sigma_j+\sigma_k+\mathbb{H}) \\
+&\,\, 2^{2\ell+1} \cdot (2^{2\ell+2}+2^{2\ell+2}\sigma_i-2^{2\ell+2}\sigma_j-2^{2\ell+2}\sigma_k)  \\
+& \,\, 2^{2\ell+2} \cdot (2^{2\ell+2}-2^{2\ell+2}\sigma_i+2^{2\ell+2}\sigma_j-2^{2\ell+2}\sigma_k) \\
+& \,\, 2^{2\ell+2} \cdot (2^{2\ell+2}-2^{2\ell+2}\sigma_i-2^{2\ell+2}\sigma_j+2^{2\ell+2}\sigma_k) \\
+& \,\,  1 \cdot (2^{4\ell+4}+2^{4\ell+4}\sigma_i-2^{4\ell+3}\mathbb{H})    \\
= & \,\, 2^{4\ell+6} = 2^{h+4}. 
\end{align*}
\end{rem}

\begin{defn}\rm\label[defn]{defn:normlattice}
When $(h, G) = (2^{n-1}m, \Cn)$ or $(4\ell+2, Q_8)$, let $\La'_{h,G}$ be the sublattice of $\La_{h, G}$ that is generated by the $RO(G)$-periodicities in \cref{theorem:ROGperiodicity}.
\end{defn}

Recall from \cref{defn:complexity} that the \textit{complexity of $\pi_\star^G E_h$} is defined as the quotient $RO(G)/\La_{h,G}$. Since $\La_{h,G}' \subset \La_{h,G}$, there is a natural quotient map
\[RO(G)/\La_{h, G}' \longrightarrow RO(G)/\La_{h, G},\] 
and hence $RO(G)/\La_{h,G}'$ provides an upper bound for $RO(G)/\La_{h,G}$.

\begin{thm}\label[thm]{thm:finiteComplexity}
The complexity of $\pi_\star^G E_h$ is finite, given by the following bounds: 
\begin{enumerate}
\item When $(h,G)=(2^{n-1} m, \Cn)$, 
\[RO(G)/\La'_{h,G} \cong  \bigoplus_{i=1}^{n-1} \mathbb{Z}/2^{2^{n-i-1}m+n-i} \oplus \mathbb{Z}/2^{h+n+1}.\]
\item When $(h,G)=(4\ell+2, Q_8)$, 
\[RO(G)/\La'_{h,G} \cong \mathbb{Z}/2^{2\ell+2}\oplus \mathbb{Z}/2^{2\ell+3}\oplus\mathbb{Z}/2^{2\ell+3}\oplus \mathbb{Z}/2^{4\ell+6}.\]
\end{enumerate}
\end{thm}
\begin{proof}


For $(1)$, the claim follows directly from computing the Smith normal form of the matrix
\[\begin{pmatrix}
1 & 1 & 1 & 2 & \cdots & 2^{n-2} \\
2^{m+1} & - 2^{m+1} \\
2^{2m+2} & & -2^{2m+1} \\
2^{4m+3}&&& -2^{4m+2} \\ 
\vdots &&&& \ddots \\ 
2^{2^{n-1}m + n} &&&&& -2^{2^{n-1}m + n-1}
\end{pmatrix}.\]
For $(2)$, note that for the $RO(Q_8)$-periodicities in \cref{theorem:ROGperiodicity}, the periodicities $(2^{h+2}+2^{h+2}\sigma_j-2^{h+1}\mathbb{H})$ and $(2^{h+2}+2^{h+2}\sigma_k-2^{h+1}\mathbb{H})$ are linear combinations of the other five periodicities:
\begin{align*}
2^{4\ell+4}+2^{4\ell+4}\sigma_j-2^{4\ell+3}\mathbb{H} & = (2^{4\ell+4}+2^{4\ell+4}\sigma_i-2^{4\ell+3}\mathbb{H}) \\
& - 2^{2\ell+1} \cdot (2^{2\ell+2}+2^{2\ell+2}\sigma_i-2^{2\ell+2}\sigma_j-2^{2\ell+2}\sigma_k) \\
& + 2^{2\ell+1} \cdot (2^{2\ell+2}-2^{2\ell+2}\sigma_i+2^{2\ell+2}\sigma_j-2^{2\ell+2}\sigma_k),    
\end{align*}
\begin{align*}
2^{4\ell+4}+2^{4\ell+4}\sigma_k-2^{4\ell+3}\mathbb{H} & = (2^{4\ell+4}+2^{4\ell+4}\sigma_i-2^{4\ell+3}\mathbb{H}) \\
& - 2^{2\ell+1} \cdot (2^{2\ell+2}+2^{2\ell+2}\sigma_i-2^{2\ell+2}\sigma_j-2^{2\ell+2}\sigma_k) \\
& + 2^{2\ell+1} \cdot (2^{2\ell+2}-2^{2\ell+2}\sigma_i-2^{2\ell+2}\sigma_j+2^{2\ell+2}\sigma_k).    
\end{align*}
The claim now follows directly from computing the Smith normal form of the matrix 
\[\begin{pmatrix}
1 & 1 & 1 & 1 & 1 \\ 
2^{2\ell+2} & 2^{2\ell+2} & -2^{2\ell+2} & -2^{2\ell+2} \\ 
2^{2\ell+2} & - 2^{2\ell+2} & 2^{2\ell+2} & -2^{2\ell+2} \\ 
2^{2\ell+2} &- 2^{2\ell+2} & -2^{2\ell+2} & 2^{2\ell+2} \\ 
2^{4\ell+4} &2^{4\ell+4} && &-2^{4\ell+3}
\end{pmatrix}.\]
\end{proof}

\begin{example}\rm
The lattices $\La'_{h, C_2}$ and $\La_{h,C_2}$ of $E_h$ are both 
\[\mathbb{Z}\langle\rho_{2} \rangle\oplus \mathbb{Z}\langle 2^{h+1}-2^{h+1}\sigma\rangle.\]
We have 
\[RO(C_2)/\La_{h,C_2} = RO(C_2)/\La'_{h, C_2} \cong \mathbb{Z}/2^{h+2}.\]
\end{example}

\begin{example}\rm 
The lattice $\La'_{2, C_4}$ of $E_2$ is 
\[\mathbb{Z}\langle\rho_{4} \rangle\oplus \mathbb{Z}\langle 4-4\sigma\rangle\oplus\mathbb{Z}\langle 16 - 8 \lambda_1\rangle.\]
It is a result of the computations in \cite{HHR17, BBHS20} that the lattice $\La_{2,C_4}$ is
\[\mathbb{Z}\langle\rho_{4} \rangle\oplus \mathbb{Z}\langle 4-4\sigma\rangle\oplus\mathbb{Z}\langle 10-2\sigma-4\lambda_1\rangle.\]
The periodicity $(10-2\sigma-4\lambda_1)$ arises from the fact that the class $u_{4 \lambda_1}u_{2\sigma}$ is a permanent cycle in the $C_4$-HFPSS of $E_2$.  
We have 
\begin{align*}
RO(C_4)/\La'_{2,C_4} &\cong \mathbb{Z}/4 \oplus \mathbb{Z}/32, \\
RO(C_4)/\La_{2,C_4} &\cong \mathbb{Z}/2 \oplus \mathbb{Z}/32.
\end{align*}
\end{example}

\begin{example}\rm
The lattice $\La'_{4, C_4}$ of $E_4$ is 
\[\mathbb{Z}\langle\rho_{4} \rangle\oplus \mathbb{Z}\langle 8-8\sigma\rangle\oplus\mathbb{Z}\langle 64 - 32 \lambda_1\rangle.\]
It is a result of the computations in \cite{HSWX23} that the lattice $\La_{4,C_4}$ is
\[\mathbb{Z}\langle\rho_{4} \rangle\oplus \mathbb{Z}\langle 8-8\sigma \rangle\oplus\mathbb{Z}\langle 36-4\sigma-16\lambda_1\rangle.\]
The periodicity $(36-4\sigma-16\lambda_1)$ arises from the fact that the class $u_{16 \lambda_1}u_{4\sigma}$ is a permanent cycle in the $C_4$-HFPSS of $E_4$.  
We have 
\begin{align*}
RO(C_4)/\La'_{4,C_4} &\cong \mathbb{Z}/8 \oplus \mathbb{Z}/128, \\
RO(C_4)/\La_{4,C_4} &\cong \mathbb{Z}/4 \oplus \mathbb{Z}/128.
\end{align*}
\end{example}










\section{Norm structures and norms of differentials}\label{sec:norm}

In this section, we show that a norm structure on filtered equivariant spectra induces relationships between differentials across various subgroups in the associated spectral sequences. The resulting differentials are called the \textit{norm-transfer differentials} (\Cref{thm:predicttransfer}). To accomplish this, we refine the notion of a norm structure on filtered $G$-spectra so that it encodes the relevant higher coherence data.

In the literature (\cite[Section~4]{HHR17} and \cite[Section~3.4]{MSZ23}), 
a norm structure on a filtered $G$-spectrum $P_{\bullet}^G \in F(\mathbb{Z}^{op}, Sp^G)$ consists of maps 
\[
N_H^G P^H_s \longrightarrow P^G_{ks} 
\quad \text{and} \quad  
N_H^G \bigl(P^H_s / P^H_{s+r}\bigr) \longrightarrow P^G_{ks} / P^G_{ks+r},
\]  
which are up to homotopy compatible with the structure maps $P_s \to P_{s-1}$ and $P_s \to P_s/P_{s+r}$.  More specifically, the following diagrams commute up to homotopy
\[\begin{tikzcd}
N_H^G P_s^H \ar[r] \ar[d] & P_{ks}^G \ar[d] \\
N_H^G P_{s-1}^H \ar[r] & P_{k(s-1)}^G
\end{tikzcd} \hspace{0.3in} 
\begin{tikzcd}
N_H^G P_s^H \ar[r] \ar[d] & P_{ks}^G \ar[d] \\ 
N_H^G \left(P_s^H/P_{s+r}^H\right) \ar[r] &P_{ks}^G/P_{ks+r}^G
\end{tikzcd}\]
Here, $H \subset G$ is a subgroup of index $k$, and $P^H_{\bullet} := i^*_H P^G_{\bullet}$. 

For the proof of the norm-transfer differential in \cref{thm:predicttransfer}, however, we will need the explicit compatibility data of the norm structure. 
To this end, it is convenient to formulate the norm structure in the setting 
of $\infty$-categories (\cref{defn:normcoherent}), following \cite[Section~1.2.2]{Lur17}.

 


In this section and the next section (\cref{sec:sharpness}), we will let $H$ be a subgroup of $G$ of index $k$. We denote by $\mathcal{C}^{\mathrm{Fil}} := F(N(\mathbb{Z}^{op}), \mathcal{C})$ the category of filtered objects in $\mathcal{C}$.  This $\infty$-category admits a symmetric monoidal structure given by the Day convolution. Our main example is $\mathcal{C} = \SpG$, the $\infty$-category of $G$-spectra.




Given a function $f: G/H \to \mathbb{Z}$, its \textit{weight} $|f|$ is the sum 
\[
|f| = \sum_{[g] \in G/H} f([g]).
\] 
Let $I_{\geq s}$ denote the poset of all functions $f \in \text{Fun}(G/H, \mathbb{Z})$ with weight $|f| \geq s$, 
ordered by $f_1 \to f_2$ if $f_1([g]) \geq f_2([g])$ for every $[g] \in G/H$. 
Given a filtered $H$-spectrum $P_{\bullet}^H \in \SpFilH$, define a functor  
\[
F_s: I_{\geq s} \longrightarrow \SpH
\]
by sending a function $f$ to the $H$-spectrum 
$
\bigwedge_{[g] \in G/H} P^H_{f([g])}$. The colimit of the functor $F_s$ over $I_{\geq s}$ admits a lift to a $G$-spectrum.  This is because it is the colimit of the following simplicial diagram, where all the simplices and all the maps can be made $G$-equivariant:
\begin{equation}\label{eq:simplicial}
\begin{tikzcd}[column sep=small]
\cdots
  \arrow[r, shift left=1.8]
  \arrow[r, shift left=0.6]
  \arrow[r, shift right=0.6]
  \arrow[r, shift right=1.8]
& \bigvee\limits_{(f_0 \to f_1 \to f_2) \in I_{\geq s}} F_s(f_0)
  \arrow[r, shift left=1.2]
  \arrow[r]
  \arrow[r, shift right=1.2]
& \bigvee\limits_{(f_0 \to f_1) \in I_{\geq s}} F_s(f_0)
  \arrow[r, shift left=0.6]
  \arrow[r, shift right=0.6]
& \bigvee\limits_{f_0 \in I_{\geq s}} F_s(f_0).
\end{tikzcd}
\end{equation}
More specifically, for $0$-simplices, we can organize the spectrum $\bigvee_{f \in I_{\geq s}} F_s(f)$ into a wedge sum of $G$-spectra as follows: given a summand $\bigwedge_{[g] \in G/H} P^H_{f([g])}$, if the stabilizer subgroup of the $G$-action on this summand is $K$, then the entire $G$-orbit of such summands has the form
\[
G_{+} \wedge_K \bigwedge_{[g] \in G/H} P^H_{f([g])}.
\]
Here, the  spectrum $\bigwedge_{[g] \in G/H} P^H_{f([g])}$ is a $K$-spectrum  and can be written as the norm $N_H^K (Y)$ for an $H$-spectrum $Y$. 

For $1$-simplices, we consider the $G$-action on the indexing morphisms. If the stabilizer subgroup of this action is $K$, then the corresponding orbit again takes the above form. Observe that the stabilizer of an indexing morphism is always contained in the stabilizers of its associated $0$-simplices. Consequently, the face maps from $1$-simplices to $0$-simplices are $G$-equivariant. By proceeding in the same way for higher simplices and face maps, the diagram \eqref{eq:simplicial} lifts to a $G$-diagram, and its colimit $\mathop{\colimit}\limits_{I_{\geq s}} F_s$ is therefore a $G$-spectrum.



\begin{defn}\rm\label[defn]{defn:normtower}
    The \textit{norm functor}
    \[
    \mathcal{N}_H^G \colon \SpFilH \longrightarrow  \SpFilG
    \]
 on filtered equivariant spectra is defined by sending a filtered $H$-spectrum $P^H_\bullet \in \SpFilH$ to the filtered $G$-spectrum $\mathcal{N}_H^G P^H_\bullet$, where $(\mathcal{N}_H^G P^H_\bullet)(s)$ is defined to be the colimit $\mathop{\colimit}\limits_{I_{\geq s}} F_s$ taken in the $\infty$-category $\SpG$. The map from $ (\mathcal{N}_H^G P^H_\bullet)(s+1)$ to $ (\mathcal{N}_H^G P^H_\bullet)(s)$ is induced by the universal property.
\end{defn}




\begin{rem}\rm \label[rem]{rem:global}

We define the norm functor on filtered objects for a fixed subgroup $H \subset G$ in order to minimize the structural assumptions required for the Norm-Transfer Differentials Theorem (\cref{thm:predicttransfer}). A global version of this construction can also be formulated. In \cite{BDGNS16}, Barwick--Dotto--Glasman--Nardin--Shah introduced the concept of a $G$-symmetric monoidal $\infty$-category, in which the multiplicative norm is incorporated into the underlying structure. In \cite{Car25}, Carrick constructed the category of filtered $G$-spectra as a $G$-symmetric monoidal $\infty$-category.


\end{rem}


 

\begin{defn} \label[defn]{defn:normcoherent}\rm
   Given a subgroup $H\subset G$, a \emph{norm structure from  $H$ to $G$} on a filtered $G$-spectrum $P_\bullet^G \in \SpFilG$ is a morphism from $\mathcal{N}_H^G (P_{\bullet}^H)$ to $P_{\bullet}^G$ in $\SpFilG$.
\end{defn}

When $H \subset G$ is fixed and clear from the context, a filtered $G$-spectrum with a norm structure is a filtered $G$-spectrum equipped with a chosen norm structure from $H$ to $G$. The following proposition shows that our notion of norm structure refines the classical one appearing in the literature \cite{HHR17, MSZ23}, in that it encodes all higher homotopies.

\begin{prop}\label[prop]{prop:explicitnorm}
For $P_{\bullet}^G$ a filtered $G$-spectrum with a norm structure, there exist maps  
\[
N_H^G P^H_s \longrightarrow P^G_{ks} 
\quad \text{and} \quad  
N_H^G \bigl(P^H_s / P^H_{s+r}\bigr) \longrightarrow P^G_{ks} / P^G_{ks+r},
\]  
which are coherently compatible with the structure map $P_s \to P_{s-1}$ and $P_s \to P_s / P_{s+r}$ 
for every integer $s$ and positive integer $r$.
\end{prop}
\begin{proof}
The existence of the maps $N_H^G P_s^H\rightarrow P_{ks}^G$ follows directly from \cref{defn:normcoherent}. The existence of the map $N_H^G (P^H_s /P^H_{s+r}) \rightarrow P^G_{ks}/P^G_{ks+r}$ is induced by the following diagram, where both rows are cofiber sequences: 
\[
\begin{tikzcd}
F \arrow[r] \arrow[d] & N_H^G P_s^H \arrow[r] \arrow[d] & N_H^G(P_s^H/P_{s+r}^H) \arrow[d,dashed] \\
P^G_{ks+r}/P^G_{ks+2r} \arrow[r] & P_{ks}^G/P^G_{ks+2r} \arrow[r] &P^G_{ks}/P^G_{ks+r}.
\end{tikzcd}
\]
In the diagram above, $F$ is the fiber of the map $N_{H}^G P_s^H \to N_H^G(P_s^H/P_{s+r}^H)$.  The left vertical map exists because $F$ can be constructed as a quotient of $G_+\wedge_H (P_{s+r}^H\wedge P_s^H\wedge \cdots\wedge P_{s}^H)$ by things of filtration at least $(ks+2r)$ \cite[Appendix~A.3.4]{HHR16} (see also \cite[Section~2.9]{HHR21}).  Therefore, the map 
\[G_+ \wedge_H \left(P_{s+r}^H\wedge P_{s}^H\wedge \cdots\wedge P_{s}^H\right) \longrightarrow P_{ks+r}^G\] 
gives the left vertical map $F \to P^G_{ks+r}/P^G_{ks+2r}$.
\end{proof}

Given $P^G_{\bullet} \in \SpFilG$ with a norm structure, let $H$-$\mathcal{E}_{*}^{*,\star}$ denote the $RO(H)$-graded spectral sequence 
associated with the tower $P_{\bullet}^H = i^*_H P_{\bullet}^G$. We refer to this as the $H$-level spectral sequence associated with the tower $P_{\bullet}^G$. The norm structure provides a comparison between the $H$-level spectral sequence and the $G$-level spectral sequence. In particular, the norm structure induces a norm map between the corresponding $\mathcal{E}_2$-pages:
\[
N_H^G \colon 
H\text{-}\mathcal{E}_2^{s,\,V+s} \;\longrightarrow\; 
G\text{-}\mathcal{E}_2^{ks,\, \ind_H^G(V)+ks}.
\]
Moreover, the norm structure can be used to produce differentials in the $G$-level spectral sequence from those in the $H$-level spectral sequence. For example, by \cite[Theorem~4.7]{HHR17} (see also \cite[Theorem~2.8]{HSWX23}), a differential $d_r(x) = y$ in the $H$-level spectral sequence induces a corresponding differential in the $G$-level spectral sequence of the form
\begin{equation}\label{eq:HHRnorm}
d_{\,k(r-1)+1}\bigl(N_H^G(x)\, a_{\bar{\rho}}\bigr) 
  \;=\; N_H^G(y),
\end{equation}
where $\bar{\rho} = \ind_H^G(1) - 1$.


This differential plays an important role in the computations of 
$E_2^{hC_4}$ and $E_4^{hC_4}$ \cite{HHR17,HSWX23}, as well as in establishing the bound for the strong vanishing line for $E_h^{hG}$ at all heights and all finite subgroups of $\mathbb{G}_h$ at the prime 2 \cite{DLS2022}. We will refer to it as the Hill--Hopkins--Ravenel \textit{norm differential}. In \cref{thm:predicttransfer} below, we will prove another kind of norm differential on $N_H^G(x)$, without the factor $a_{\bar{\rho}}$ (see also \cite[Theorem~5.23]{Ull13}).

\begin{rem}\rm
The conditions for a tower to admit a norm structure, as given in \cref{defn:normcoherent}, are stronger than the conditions in \cite[Section~3.4]{MSZ23}, which require only the case $r=1$ for the map
\[
N_H^G \bigl(P^H_s / P^H_{s+r}\bigr) \longrightarrow P^G_{ks} / P^G_{ks+r}.
\]
The key difference is that if only the case $r=1$ is required, then one can show that a class $y$ killed on the $H$-$\mathcal{E}_r$-page has its norm $N_H^G(y)$ killed on or before the $G$-$\mathcal{E}_{|G/H|(r-1)+1}$-page \cite[Proposition~3.12]{MSZ23}. The stronger norm structure for all $r$ produces an explicit differential killing $N_H^G(y)$ \cite[Theorem~4.7]{HHR17} \cite[Theorem~2.8]{HSWX23}.
These  are also necessary for  \cref{thm:predicttransfer}. Moreover, the proof of \cref{thm:predicttransfer} relies on higher morphisms in the coherent compatibility data, whereas the proofs of the Hill--Hopkins--Ravenel norm differential in 
\cite[Theorem~4.7]{HHR17} and \cite[Theorem~2.8]{HSWX23} do not require them.

\end{rem}

Recall that, analogous to the transfer functor $\tr_H^G$, the norm functor $N_H^G$ also satisfies a double coset formula. More precisely, let $H,K$ be subgroups of a finite group $G$, and let $X$ be an $H$-spectrum. Then
\begin{equation}\label{eq:doublecoset}
i^*_K N_H^G X \;\cong\; \bigwedge_{[g]\in K\backslash G/H} N_{K\cap H^g}^K \, i^*_{K\cap H^g} X^g.
\end{equation}
Here, $H^g$ denotes the conjugate subgroup $gHg^{-1}$, and $X^g=(gH)_+\wedge_H X$ is a copy of $X$ equipped with the $H^g$-action. Moreover, for an $H$-equivariant map $f\colon X\to Y$, we write $f^g$ for the induced map $X^g\to Y^g$ appearing as the summand map in 
\[
i^*_K N_H^G f\colon i^*_K N_H^G X \longrightarrow i^*_K N_H^G Y.
\]

\begin{thm}[Norm-Transfer Differentials]\label[thm]{thm:predicttransfer}
Let $H$ be a subgroup of $G$ of index $k$, and $P_{\bullet}^G$ a filtered $G$-spectrum equipped with a norm structure. Suppose there is a differential 
$d_r(x) = y$ in the $H$-level spectral sequence, where $x$ is in stem $V \in RO(H)$, then there is a norm-transfer differential
\[
d_r \bigl(N_H^G(x)\bigr) \;=\; \mathrm{tr}_H^G \left(y \cdot \prod_{[g]\in T} N_{H\cap H^g}^{H}\!\bigl(x^{g}\bigr)\right)
\]
in the $G$-level spectral sequence.  Here, $T = H\backslash G/H - [H]$, and the target of the differential is in stem $(\ind_H^G(V) - 1)$.  In particular, if $H$ lies in the center of $G$, then the formula above becomes 
\[
d_r \bigl(N_H^G(x)\bigr) \;=\; \mathrm{tr}_H^G\!\bigl(yx^{k-1}\bigr).
\]

\end{thm}
\begin{proof}
By the construction of spectral sequences, the differential $d_r(x)=y$ can be represented by the following commutative diagram:
\begin{equation}\label{eq:startdiff}
\begin{tikzcd}
S^V \ar[r] \ar[d, "y"]  & I\wedge S^V \ar[r] \ar[d,"\bar{x}"] & S^{1+V} \ar[d, "x"] \\ 
P_{s+r-1}^{H} \ar[r] &P_s^{H}  \ar[r] & P_{s}^{H}/P_{s+r-1}^{H}. 
\end{tikzcd}
\end{equation}
Here, $I$ denotes the interval $[0,1]$ with $0$ as the chosen base point. 
Applying the norm functor $N_{H}^G$ to the above diagram yields:
\[
\begin{tikzcd}
N_{H}^G S^V  \ar[r] \ar[d, "N_{H}^{G} (y)"]  & N_{H}^G I \wedge N_{H}^G S^V \ar[r] \ar[d,"N_H^G (\bar{x})"] & N_{H}^G S^{1+V} \ar[d, "N_{H}^{G}(x)"] \\ 
N_{H}^{G} P_{s+r-1}^{H}  \ar[r] &N_{H}^{G} P_s^{H} \ar[r] & N_{H}^{G}(P_{s}^{H}/P_{s+r-1}^{H}). 
\end{tikzcd}
\]
Note that the two rows above are not cofiber sequences. The cofiber of the map $N_{H}^G I\rightarrow N_H^G S^1 $ is $\Sigma \partial N_H^G I$. Therefore, the cofiber of the map $ N_{H}^G I \wedge N_{H}^G S^V\rightarrow N_{H}^G S^{1+V} $ is $\Sigma\partial N_H^G I\wedge N_H^G S^V$. 
Combining this with the structure map 
$$ N_{H}^G (P^{H}_s /P^{H}_{s+r-1} )\rightarrow P^G_{ks}/P^G_{ks+r-1}$$
from \cref{prop:explicitnorm}, induced by the norm structure on $P_{\bullet}^G$, 
we obtain the following commutative diagram, 
where both the top and bottom rows are cofiber sequences:

\[
\begin{tikzcd}
    \partial N_{H}^G I\wedge N_{H}^G S^V\ar[r]\ar[dd,"f"]& N_{H}^G I\wedge N_{H}^G S^V \ar[r]\ar[d] & N_{H}^G S^{1+V}\ar[d, "N_{H}^{G}(x)"]\\
  & N_{H}^G P^{H}_s \ar[r]\ar[d]& N_{H}^G (P^{H}_s /P^{H}_{s+r-1} )\ar[d]\\
    P^G_{ks+r-1}\ar[r]& P^G_{ks}\ar[r]& P^G_{ks}/P^G_{ks+r-1}.
\end{tikzcd}
\]
 The  long vertical map $f$ is induced by the second and third vertical maps. We will show that $f$ represents the transfer class $\tr_H^G \left(y\prod\limits_{[g]\in T} N_{H\cap H^g}^{H}(x^{g})\right)$ in the spectral sequence. If so, then the commutative diagram above will give the desired $d_r$-differential:
\[
d_r \left(N_{H}^G (x)\right)=\tr_{H}^G \left(y\prod\limits_{[g]\in T} N_{H\cap H^g}^{H}(x^{g})\right).
\] 
 Consider the $G$-CW structure of the cube $N_H^G I$. Let $\sk_i N_H^G I$ denote the $i$-th skeleton of $N_H^G I$. In particular, we have $\partial N_H^G I=\sk_{k-1} N_H^G I$. We claim that there exists a lifting map  
 $$\phi_{k-2}:\sk_{k-2}N_H^G I\wedge N_H^G S^V\rightarrow P^G_{ks+2r-2},$$
 which makes the following diagram commute. 
 \begin{equation}\label{eq:transferdiff}
 \begin{tikzcd}
     \sk_{k-2}N_H^G I\wedge N_H^G S^V \ar[d,"\phi_{k-2}"]\ar[r]& \partial N_H^G I\wedge N_H^G S^V\ar[d,"f"]\ar[r] & {G}_{+}\wedge_H S^{k-1}\wedge N_{H}^G S^V\ar[d]\\
     P^G_{ks+2r-2}\ar[r]& P^G_{ks+r-1}\ar[r,"\pi"] &P^G_{ks+r-1}/P^G_{ks+2r-2}.
      \end{tikzcd}
 \end{equation}
Here, the sphere $S^{k-1}$ is the one-point compactification of the face $I^{k-1}=\bigwedge_{[g]\in G/H-[H]}I^g$ and is equipped with the $H$-action on the indexing set $\{ G/H-[H]\}$. The induced third vertical map identifies the class represented by the composite map $\pi \circ f$.


We construct the lift inductively on  the skeleta of $N_H^G I$. For the base case $i=0$, we have the following maps induced by $N_H^G (y)$ on the $0$-th skeleton:
\[
\phi_0 \colon \sk_0 N_H^G I \wedge N_H^G S^V\xrightarrow{N_H^G (y)}N_H^G P^H_{s+r-1}\rightarrow P^G_{ks+k(r-1)}.
\]
Moreover, the map $\phi_0$ fits into a $3$-morphism as follows:
\begin{equation}\label{eq:threecell}
    \begin{tikzcd}[row sep=small, column sep=large]
      \sk_0 N_H^G I \wedge N_H^G S^V  \ar[rr]\ar[dd,"\phi_0"]\ar[dr] & & N_H^G I\wedge N_H^G S^V \ar[dd] \\
        & \partial N_{H}^G I\wedge N_{H}^G S^V\ar[ur]\ar[dd,"f"]&\\
        P^G_{ks+k(r-1)}\ar[rr,dashrightarrow]\ar[dr]& & P^G_{ks} \\
        &P^G_{ks+r-1}\ar[ur]&\\
        \end{tikzcd}
\end{equation}
This $3$-morphism implies that the following  diagram commutes: 
\begin{equation}\label{eq:diagram}
\begin{tikzcd}
    \sk_0 N_H^G I \wedge N_H^G S^V\ar[r]\ar[d,"\phi_{0}"]&\partial N_H^G I\wedge N_H^G S^V \ar[d,"f"]\\
    P^G_{ks+k(r-1)}\ar[r]&P^G_{ks+r-1}.
    \end{tikzcd}
\end{equation}
For the induction hypothesis, we assume that  for the $i$-th skeleton we have a lifting map $$\phi_{i}: \sk_i N_H^G I \wedge N_H^G S^V\rightarrow P^G_{ks+(k-i)(r-1)}$$ making the following diagram commute:
\[
\begin{tikzcd}
    \sk_i N_H^G I \wedge N_H^G S^V\ar[r]\ar[d,"\phi_{i}"]&\partial N_H^G I\wedge N_H^G S^V \ar[d,"f"]\\
    P^G_{ks+(k-i)(r-1)}\ar[r]&P^G_{ks+r-1}.
    \end{tikzcd}
\]
Now consider the extension from the $i$-th skeleton 
$\sk_i N_H^G I$ to the $(i+1)$-st skeleton 
$\sk_{i+1} N_H^G I$.  We then construct  a map $\bar{\phi}_{i+1}$ such that the following  diagram commutes:
\begin{equation}\label{eq:induction}
\begin{tikzcd}
    \coprod \partial I^{i+1}\wedge N_H^G S^V\ar[d,"\iota"]\ar[r]&\sk_i N_H^G I\wedge N_H^G S^V\ar[d]\ar[r,"\phi_{i}"] &P^G_{ks+(k-i)(r-1)}\ar[dd]\\
    \coprod I^{i+1}\wedge N_H^G S^V\ar[r]\ar[rrd,"\bar{\phi}_{i+1}", bend right=15] &\sk_{i+1}N_H^G I\wedge N_H^G S^V & \\
    &&P^G_{ks+(k-i-1)(r-1)}.
    \end{tikzcd}
\end{equation}

For any face $I^{i+1}$, let $K \subset G$ denote the stabilizer subgroup that fixes this face.  Then the $(i+1)$-st skeleton of $N_H^G I$ contains the orbit $K\backslash G_+ \wedge I^{i+1}$.  Let $$S=\{[g_1],[g_2],\dots, [g_{i+1}]\}$$ be the set of elements in $G/H$ indexing the face $I^{i+1}$ and let $\bar{S}=G/H-S$ denote its complement. By the adjunction, it suffices to construct the $K$-equivariant map  $i^*_K \bar{\phi}_{i+1}$ on the  summand $$\bigwedge_{[g]\in S }I^g\wedge i^*_K N_H^G S^V$$ as the composition of the following maps
\begin{equation*}
\begin{aligned}
\bigwedge\limits_{[g]\in S} I^g\wedge i^*_K N_H^G S^V &\cong \bigwedge_{[g]\in K\backslash S}N_{H^g\cap K}^K(I^{g}\wedge (S^V)^g)\wedge \bigwedge_{[g]\in K\backslash \bar{S}} N_{H^g\cap K}^{K}(S^V)^g\\
&\longrightarrow\bigwedge_{[g]\in K\backslash S}N_{H^g\cap K}^{K}(P_s^{H})^g \wedge \bigwedge_{[g]\in K\backslash \bar{S}}N_{H^g\cap K}^{K}(P_{s+r-1}^H)^g\\
&\longrightarrow i^*_K P^G_{ks+(k-i-1)(r-1)}.
\end{aligned}
\end{equation*}
Here, the first map in the second row above is $$\prod\limits_{[g]\in K\backslash S}N_{H^g\cap K}^{K}(\bar{x}^g)\wedge\prod\limits_{[g]\in K\backslash \bar{S}}N_{H^g\cap K}^{K}(y),$$ and the last map in the second row is induced by the norm structure on $P^G_{\bullet}$. Notice that the above composition map factors through
\[
\bigwedge_{[g]\in K\backslash S}N_{H^g\cap K}^{K}(P_s^{H})^g \wedge \bigwedge_{[g]\in K\backslash \bar{S}}N_{H^g\cap K}^{K}(P_{s+r-1}^H)^g.
\]
We apply the norm functor $N_{H^g\cap K}^{K}$ to the diagram~(\ref{eq:startdiff}) for each coset $[g]$ in $K\backslash S$ and $K\backslash \bar{S}$, then the norm structure on $P_{\bullet}^G$ implies the commutativity of the diagram~(\ref{eq:induction}).
Therefore, it induces a lifting map $$\phi_{i+1}\colon \sk_{i+1}N_H^G I\wedge N_H^G S^V\rightarrow P^G_{ks+(k-i-1)(r-1)}. $$ 
Similar to the previous $3$-morphism argument, we have  the following commutative diagram: 
\[
\begin{tikzcd}
    \sk_{i+1} N_H^G I \wedge N_H^G S^V\ar[r]\ar[d,"\phi_{i+1}"]&\partial N_H^G I\wedge N_H^G S^V \ar[d,"f"]\\
    P^G_{ks+(k-i-1)(r-1)}\ar[r]&P^G_{ks+r-1}.
    \end{tikzcd}
\]
Therefore, the inductive process establishes the claim.

To conclude the proof, we identify the class represented by the composite map 
 $\pi\circ f$ in the spectral sequence. The commutative diagram~(\ref{eq:transferdiff}) shows that the composite map $\pi\circ f$ represents a transfer class.
By the adjunction, the source of the transfer is determined by the $H$-equivariant map $\psi\colon S^{k-1}\wedge i^*_HN_H^G S^V\rightarrow P^H_{ks+r-1}/P^H_{ks+2r-2}$. By the inductive construction above  and the double coset formula~(\ref{eq:doublecoset}), the following composition
 \[
 \bigwedge_{[g]\in G/H-[H]}I^g\wedge i^*_H N_H^G S^V\rightarrow S^{k-1}\wedge i^*_HN_H^G S^V\xrightarrow{\psi} P^H_{ks+r-1}/P^H_{ks+2r-2},
 \] 
represents the class $y\prod\limits_{[g]\in T} N_{H\cap H^g}^{H}(\bar{x}^{g})$. It follows that $\psi$ represents the class $y\prod\limits_{[g]\in T} N_{H\cap H^g}^{H}(x^{g})$. Hence, there is a  $d_r$-differential from $N_H^G (x)$ to $\tr_H^G\left(y\prod\limits_{[g]\in T} N_{H\cap H^g}^{H}(x^{g})\right)$ as claimed.
\end{proof}


\section{Differentials on orientation classes}\label{sec:sharpness}


In this section, we will show that for a $G$-spectrum $X$, both its equivariant slice tower and its homotopy fixed point tower admit norm structures, in the sense of \cref{defn:normcoherent}. When we specialize to the case $X = E_h$ and apply \cref{thm:predicttransfer}, we obtain \cref{thm:D} (\cref{thm:predicteddiff}), which gives a family of differentials on orientation classes. 

Recall that for a $G$-spectrum $X$, its slice tower $P^G_\bullet X$ is defined by first taking $\tau^G_{\geq s}$ to be the localizing subcategory of $G$-spectra generated by all spectra of the form $G_+ \wedge_H S^{m\rho_H}$, where $H$ is a subgroup of $G$ and $m|H| \geq s$.  The functor $P_G^{n-1}(-)$ is the nullification functor which makes $\tau^G_{\geq n}$ acyclic, and $P^G_n X$ is the fiber of the natural map $X \to P_G^{n-1}X$.  The tower 
\[P^G_\bullet X = \{\cdots \to P^G_{n+1}X \to P^G_n X \to P^G_{n-1}X \to \cdots\} \]
is the \textit{slice tower of $X$}. In particular, $P_n^G X$ lies in $\tau_{\geq n}^G$, which means it is slice $n$-connective.





\begin{defn}\rm
    Let $X$ be a $G$-spectrum, and let $H$ be a subgroup of $G$. $X$ has a \textit{norm} from $H$ to $G$ if there exists a $G$-equivariant map $N_H^G i_H^*(X) \to X$.
\end{defn}

\begin{prop}\label[prop]{prop:slicetowernorm}
    Let $X$ be a $G$-spectrum with a norm from $H$ to $G$.  Then its slice tower $P_\bullet^G X$ admits a norm structure from $H$ to $G$ in the sense of \cref{defn:normcoherent}.
\end{prop}

\begin{proof}
Let $c_{\bullet}(X)$ denote the constant tower of $X$. Note that $\mathcal{N}_H^G i^*_H(c_{\bullet}(X))\simeq c_{\bullet}(N_H^G i^*_H X)$. The norm on $X$ gives a morphism  $c_{\bullet}(N_H^G i^*_H X) \rightarrow c_{\bullet}(X)$. This implies that the constant tower $c_{\bullet}(X)$ admits a norm structure from $H$ to $G$.  Combining this with the map of towers $P_\bullet^GX \to c_{\bullet}(X)$ gives the composite map 
\[
\mathcal{N}_H^Gi_H^*(P_{\bullet}^G X)\rightarrow \mathcal{N}_H^Gi_H^*(c_{\bullet}(X))\rightarrow c_{\bullet}(X).
\]

To construct the desired map $\mathcal{N}_H^Gi_H^*(P_{\bullet}^G X) \to P_\bullet^G X$, we will show that for each integer $s$, $\mathcal{N}_H^Gi_H^*(P_{\bullet}^G X)(s)$ is slice $s$-connective.  This is because each $G$-spectrum in the simplicial diagram (\ref{eq:simplicial}) for computing $\mathcal{N}_H^G(P_{\bullet}^G X)(s)$ is slice $s$-connective.  More specifically, by the discussion before \cref{defn:normtower}, the simplices in each dimension can be organized into wedges of $G$-spectra of the form $G_+\wedge_K N_H^K Y$, where $Y$ is a slice $H$-connective $H$-spectrum.  Therefore, $N_H^K Y$ is a slice $s$-connective $K$-spectra (\cite[Chapter~\rom{1}, Corollary~5.8]{Ull13} and \cite[Theorem~2.5]{HY18}). Therefore, the $G$-spectrum $G_+\wedge_K N_H^K Y$ is also slice $s$-connected. 

By the universal property of the slice tower, we have a map
\[
\mathcal{N}_H^G i_H^* (P_{\bullet}^G X)\rightarrow P_{\bullet}^G X.
\]
This gives the norm structure of the slice tower $P_{\bullet}^G X$ from $H$ to $G$. 
\end{proof}

\begin{rem}\rm
    As noted in \cref{rem:global}, one may simultaneously consider all subgroups of $G$ to formulate a global version of \cref{prop:slicetowernorm}, as in \cite[Theorem~1.2]{Car25}.
    \end{rem}

We will denote $F_\bullet^G(X) := F(EG_+, P_\bullet^G X)$. This is the \textit{homotopy fixed point tower}, and its associated spectral sequence is the homotopy fixed point spectral sequence of $X$. 







\begin{prop}\label[prop]{prop:HFPSSnorm}
Let $X$ be a $G$-spectrum with a norm from $H$ to $G$.  Then its homotopy fixed point tower $F_\bullet^GX$ admits a norm structure from $H$ to $G$. 
\end{prop}
\begin{proof}

We will construct a map 
\[\mathcal{N}_H^G i_H^*(F_\bullet^G X) = \mathcal{N}_H^G F(EH_+, P_\bullet^H X)\longrightarrow F\left(EG_+, \mathcal{N}_H^G P_\bullet^H X\right).\]
The desired norm structure map is obtained by postcomposing this map with
\[F\left(EG_+, \mathcal{N}_H^G P_\bullet^H X\right) \longrightarrow F(EG_+, P_\bullet^G X) = F_\bullet^G X,\]
which is obtained from \cref{prop:slicetowernorm}.

Consider the following diagram

\[
\begin{tikzcd}
\mathcal{N}_H^G F(EH_+, P_\bullet^H X)\ar[r,"f"]\ar[rd,dashed] &F\left(EG_+, \mathcal{N}_H^G F(EH_+,P_\bullet^H X)\right)\\
& F(EG_+, \mathcal{N}_H^G P^H_\bullet X)\ar[u,"g", "\simeq" swap]
\end{tikzcd}
\]
where $f$ is induced by the functor $F(EG_+, -) \colon \SpFilG \rightarrow \SpFilG$ and $g$ is induced by the map
$P^H_\bullet X \to F(EH_+,P_\bullet^H X)$.  Note that $P^H_\bullet X \rightarrow F(EH_+, P^H_\bullet X)$ is an underlying equivalence, and therefore so is $\mathcal{N}_H^G P^H_\bullet X \rightarrow \mathcal{N}_H^G F(EH_+, P^H_\bullet X)$. After applying $F(EG_+, -)$, the map $g$ is an equivalence, and $g^{-1}f$ gives the desired map from  $\mathcal{N}_H^G F(EH_+, P_\bullet^H X)$ to $F(EG_+,\mathcal{N}_H^G P_\bullet^G X)$.
\end{proof}

As discussed in \cref{sec:intP2}, the Lubin--Tate theory $E_h$ is an $\mathbb{E}_{\infty}$-ring spectrum with $G$-actions by $\mathbb{E}_{\infty}$-maps \cite{GH04, Lur18}, which, after cofree completion, gives rise to an object in $\SpG$ with all norms \cite[Theorem~2.4]{HM17}.  By \cref{prop:slicetowernorm} and \cref{prop:HFPSSnorm}, the slice tower and the homotopy fixed point tower of $E_h$ both admit norm structures from $H$ to $G$ for any subgroup $H \subset G$.  Therefore, we can apply \cref{thm:predicttransfer} to analyze differentials in the corresponding spectral sequences.

\begin{thm}[Differentials on Orientation Classes]\label[thm]{thm:predicteddiff} \,
\begin{enumerate}
\item When $(h, G) = (2^{n-1}m, \Cn)$, we have the following differentials in the $\Cn$-slice spectral sequence and the $\Cn$-homotopy fixed point spectral sequence of $E_h$: 
\[d_{2^{\ell+1}-1} \left(u_{2^{\ell+n-2}\lambda_{n-1}}\right)=\tr_{C_2}^{C_{2^{n}}}\left(\bar{v}_{\ell} a_{(2^{\ell+1}-1)\sigma_2}u_{(2^{\ell+n-1}-2^\ell )\sigma_2}\right)u_{W_{\ell}}, \,\,\, 1\leq \ell \leq h,\]
where $W_{\ell} = 2^\ell \sigma_{2^{n}}+2^\ell \lambda_1+2^{\ell+1}\lambda_2 + \cdots + 2^{\ell+n-3}\lambda_{n-2}.$
\item When $(h, G) = (4m+2, Q_8)$, we have the following differentials in the $Q_8$-homotopy fixed point spectral sequence of $E_h$: 
\[d_{2^{\ell+1}-1} \left(u_{2^\ell\mathbb{H}}\right)=\tr_{C_2}^{Q_8}\left(\bar{v}_{\ell} a_{(2^{\ell+1}-1)\sigma_2}u_{3 \cdot 2^{\ell} \sigma_2}\right)u_{2^\ell \sigma_i} u_{2^\ell \sigma_j}u_{2^\ell \sigma_k}, \,\,\, 1\leq \ell \leq h.\]
\end{enumerate}
\end{thm}
\begin{proof}
By \cref{thm:predicttransfer}, when $H\subset G$ is central,  the norm-transfer  differential take a differential $d_r(x) = y$ on the $H$-level and produces the differential
\begin{equation}\label{eq:cyclicnorm}
d_r (N_H^G (x))=\tr_H^G (y x^{k-1}),
\end{equation}
on the $G$-level, where $k=|G/H|$.  

To prove (1), note that we have the following differentials in the $C_2$-SliceSS and the $C_2$-HFPSS of $E_h$ (\cite[Theorem~1.4]{HS20}):
    \[
    d_{2^{\ell+1}-1} (u_{2^\ell\sigma_2})=\bar{v}_{\ell} a_{(2^{\ell+1}-1)\sigma_2}, \,\,\, 1\leq \ell \leq h. 
    \]
This gives the desired formula when $n = 1$.  For $n \geq 2$, since $C_2 \subset \Cn$ is central, applying \cref{thm:predicttransfer} to these differentials with $H=C_2$ and $G=C_{2^n}$ produces the following differentials in the $\Cn$-SliceSS and the $\Cn$-HFPSS of $E_h$: 
\[d_{2^{\ell+1}-1} \left(N_H^G(u_{2^\ell\sigma_2})\right)=\tr_{C_2}^{C_{2^{n}}}\left(\bar{v}_{\ell} a_{(2^{\ell+1}-1)\sigma_2}u_{(2^{\ell+n-1}-2^\ell )\sigma_2}\right), \,\,\, 1 \leq \ell \leq h.\]

Recall that for $V$ an $H$-representation, the norm of the orientation class $u_V$ is given by the formula $N_H^G(u_V)=u_{\ind_H^G V}/u_{\ind_H^G |V|}$ \cite[Lemma~3.13]{HHR16}. Therefore, we have 
\[
N_H^G(u_{2^{\ell}\sigma_2})=\frac{u_{2^{\ell+n-2}\lambda_{n-1}}}{u_{2^\ell\sigma_{2^n}}u_{2^\ell\lambda_1}u_{2^{\ell+1}\lambda_2}\cdots u_{2^{\ell+(n-3)}\lambda_{n-2}}}.
\]
To finish the proof, it suffices to show that the factors in the denominator, namely 
\[u_{2^\ell\sigma_{2^n}},  u_{2^\ell\lambda_1}, u_{2^{\ell+1}\lambda_2}, \ldots, u_{2^{\ell+(n-3)}\lambda_{n-2}},\] 
are all $d_{2^{\ell+1} -1}$-cycles.  We will do so by using induction on $n$.  The base case, when $n = 2$, is immediate, because we only have the class $u_{2^{\ell} \sigma_{4}}$, which is a $d_{2^{\ell+1} -1}$-cycle by the Slice Differentials Theorem \cite[Theorem~9.9]{HHR16}.  For the inductive step, suppose that at the $\Cnminusone$-level, the classes 
\[u_{2^\ell\sigma_{2^{n-1}}},  u_{2^\ell\lambda_1}, u_{2^{\ell+1}\lambda_2}, \ldots, u_{2^{\ell+(n-4)}\lambda_{n-3}}\] 
are all $(d_{\ell+1}-1)$-cycles.  At the $\Cn$-level, the class $u_{2^\ell\sigma_{2^n}}$ is a $d_{2^{\ell+1} -1}$-cycle again by the Slice Differentials Theorem.  To prove that the class $u_{2^{\ell+i-1}\lambda_{i}}$ is a $d_{2^{\ell+1} -1}$-cycle for $1 \leq i \leq n-2$, consider the class $u_{2^{\ell+i-2}\lambda_{i-1}}$ at the $\Cnminusone$-level.  By the induction hypothesis, this is a $d_{2^{\ell+1} -1}$-cycle.  Applying \cref{thm:predicttransfer} on the class $u_{2^{\ell+i-2}\lambda_{i-1}}$ from $H = \Cnminusone$ to $G = \Cn$, we deduce that the class 
\[N_{\Cnminusone}^{\Cn} \left(u_{2^{\ell+i-2}\lambda_{i-1}} \right) = \frac{u_{2^{\ell+i-1}\lambda_i}}{u_{2^{\ell+i-1}}u_{2^{\ell+i-1} \sigma_{2^n}}} = \frac{u_{2^{\ell+i-1}\lambda_i}}{u_{2^{\ell+i-1} \sigma_{2^n}}}\]
is also a $d_{2^{\ell+1} -1}$-cycle.  Since the denominator $u_{2^{\ell+i-1} \sigma_{2^n}}$ is already a $d_{2^{\ell+1} -1}$-cycle, the class $u_{2^{\ell+i-1}\lambda_i}$ is a $d_{2^{\ell+1} -1}$-cycle.  This completes the induction step and proves that each factor in 
$$u_{W_\ell} :=u_{2^\ell\sigma_{2^n}}u_{2^\ell\lambda_1}u_{2^{\ell+1}\lambda_2}\cdots u_{2^{\ell+(n-3)}\lambda_{n-2}}$$
is $d_{2^{\ell+1} -1}$-cycle. It follows that we have the differential 
\[d_{2^{\ell+1}-1} \left(u_{2^{\ell+n-2}\lambda_{n-1}}\right)=\tr_{C_2}^{C_{2^{n}}}\left(\bar{v}_{\ell} a_{(2^{\ell+1}-1)\sigma_2}u_{(2^{\ell+n-1}-2^\ell )\sigma_2} \right)u_{W_\ell}, \,\,\, 1 \leq \ell \leq h.\]
This proves (1).

For $(2)$, applying \cref{thm:predicttransfer} for $H = C_2$, $G = Q_8$ on the $C_2$-differentials 
\[
d_{2^{\ell+1}-1} (u_{2^\ell\sigma_2})=\bar{v}_{\ell} a_{(2^{\ell+1}-1)\sigma_2}, \,\,\, 1\leq \ell \leq h 
\]
produces the $Q_8$-differentials 
\[d_{2^{\ell+1}-1} \left(N_{C_2}^{Q_8}(u_{2^\ell \sigma_2})\right)=\tr_{C_2}^{Q_8}\left(\bar{v}_{\ell} a_{(2^{\ell+1}-1)\sigma_2}u_{3 \cdot 2^{\ell} \sigma_2}\right), \,\,\, 1\leq \ell \leq h.\]
We have 
\[
N_{C_2}^{Q_8}(u_{2^\ell \sigma_2})
  = \frac{u_{2^\ell \mathbb{H}}}{u_{2^\ell \sigma_i} u_{2^\ell \sigma_j} u_{2^\ell \sigma_k}}.
\]
Since the class $u_{2^\ell\sigma_4}$ is a $d_{2^{\ell+1}-1}$-cycle in the $C_4$-HFPSS for $E_h$, so is the product 
\[
N_{C_4\langle i\rangle}^{Q_8}(u_{2^\ell\sigma_4})N_{C_4\langle j\rangle}^{Q_8}(u_{2^\ell\sigma_4})N_{C_4\langle k\rangle}^{Q_8}(u_{2^\ell\sigma_4})= \frac{u_{2^\ell \sigma_j}u_{2^\ell \sigma_k}}{u_{2^\ell \sigma_i}} \cdot \frac{u_{2^\ell \sigma_i}u_{2^\ell \sigma_k}}{u_{2^\ell \sigma_j}} \cdot \frac{u_{2^\ell \sigma_i}u_{2^\ell \sigma_j}}{u_{2^\ell \sigma_k}} = u_{2^\ell \sigma_i} u_{2^\ell \sigma_j} u_{2^\ell \sigma_k}.
\]
The desired differential now follows immediately from applying the Leibniz rule.
\end{proof}

\begin{rem}\rm
When $(h, G) = (2^{n-1}m, \Cn)$, \cref{thm:predicteddiff} gives differential formulas for powers of the top orientation class $\lambda_{n-1}$. The differentials for powers of the other orientation classes $\lambda_{n-i}$, where $2 \leq i \leq n$, follow from the same formula after setting $(h, G) = (2^{n-i}m, C_{2^{n-i+1}})$ (so that $\lambda_{n-i}$ becomes the top orientation class) and then applying the Transchromatic Isomorphism Theorem \cite[Theorem~A]{MSZ24}. Therefore, \cref{thm:predicteddiff} gives differential formulas for powers of every orientation class $\lambda_{n-i}$, for $1 \leq i \leq n$.
\end{rem}



\begin{rem}\rm
Recall that having a differential $d_r(x)=y$ implies the following:
\begin{enumerate}
\item The class $x$ is nontrivial on the $\mathcal{E}_r$-page.
\item The class $y$ is a permanent cycle, and must be killed on or before the $\mathcal{E}_r$-page.
\end{enumerate}
Furthermore, if $y$ is nontrivial on the $\mathcal{E}_r$-page, then the differential $d_r(x)=y$ is called \textit{essential}. In practice, it is possible that the class $y$ is already killed by a shorter differential before the $\mathcal{E}_r$-page, so the differential $d_r(x)=y$ becomes trivial on the $\mathcal{E}_r$-page. For example, \cite{DKLLW24} shows that in the $Q_8$-$\HFPSS(E_2)$, the class $D^{-2}g^7$ is killed by a $d_{23}$-differential, even though the Hill--Hopkins--Ravenel norm differential~(\ref{eq:HHRnorm}) proves that it is the target of a $d_{25}$-differential.
\end{rem}

\begin{example}\label[example]{exam:dseven}\rm
    In the $C_2$-$\HFPSS(E_2)$, there is a $d_7$-differential
    \[
    d_7(u_{4\sigma_2})=\bar{v}_2 a_{7\sigma_2}. 
    \]
 Applying the Hill--Hopkins--Ravenel norm-differential formula~(\ref{eq:HHRnorm}) with $H = C_2$ and $G = C_4$ to this $d_7$-differential produces the following $d_{13}$-differential in the $C_4$-$\HFPSS$:
\[
d_{13}(u_{4\lambda_1}a_{\sigma_4})= N_{C_2}^{C_4}(\bar{v}_2) u_{4\sigma_4}a_{7\lambda_1}.
\]
Applying \cref{thm:predicteddiff} with $H = C_2$ and $G =C_4$ to the this $d_7$-differential produces the following $d_7$-differential in the $C_4$-$\HFPSS$:
    \[
    d_7(u_{4\lambda_1})= \tr_{C_2}^{C_4}(\bar{v}_2 a_{7\sigma_2}u_{4\sigma_2})u_{4\sigma_4}=\tr_{C_2}^{C_4}(\bar{v}_2 a_{\sigma_2})u_{4\sigma_4}u_{2\lambda_1}a_{3\lambda_1}.
    \]
Both of these differentials are essential \cite{HHR17,HSWX23}. 

Alternatively, if we apply \cref{thm:predicteddiff} with $H = C_2$ and $G =Q_8$ to the same $d_7$-differential, we obtain a $d_7$-differential on $u_{4\mathbb{H}}$ in the $Q_8$-HFPSS:
\[
d_7(u_{4\mathbb{H}})=\tr_{C_2}^{Q_8}(\bar{v}_2u_{12\sigma_2}a_{7\sigma_2})u_{4\sigma_i}u_{4\sigma_j}u_{4\sigma_k}.
\]
This $d_7$-differential is essential, and it induces the $d_7$-differential on the class $(96, 0)$ in $Q_8$-$\HFPSS(E_2)$. In other words, this $d_7$-differential is responsible for breaking the $96$-periodicity of $EO_2(Q_8)$, making the theory exactly $192$-periodic. In \cite{DKLLW24}, this $d_7$-differential is one of the most difficult differentials to establish. By contrast, \cref{thm:predicteddiff} directly yields this differential.  
\end{example}

Letting $\ell = h$ in \cref{thm:predicteddiff} shows that we have the following differentials:
\begin{enumerate}
\item When $(h, G) = (2^{n-1}m, \Cn)$, 
\[d_{2^{h+1}-1} \left(u_{2^{h+n-2}\lambda_{n-1}}\right)=\tr_{C_2}^{C_{2^{n}}}\left(\bar{v}_{h} a_{(2^{h+1}-1)\sigma_2}u_{(2^{h+n-1}-2^\ell )\sigma_2}\right)u_{W_{h}},\]
where $W_{h} = 2^h \sigma_{2^{n}}+2^h \lambda_1+2^{h+1}\lambda_2 + \cdots + 2^{h+n-3}\lambda_{n-2}.$
\item When $(h, G) = (4m+2, Q_8)$, 
\[d_{2^{h+1}-1} \left(u_{2^h\mathbb{H}}\right)=\tr_{C_2}^{Q_8}\left(\bar{v}_{h} a_{(2^{h+1}-1)\sigma_2}u_{3 \cdot 2^{h} \sigma_2}\right)u_{2^h \sigma_i} u_{2^h \sigma_j}u_{2^h \sigma_k}.\]
\end{enumerate}
These differentials are essential in all known computations, when $(h,G) = (2, C_4)$, $(4, C_4)$, and $(2, Q_8)$ \cite{HHR17, HSWX23, DKLLW24}. By \cref{lemma:pctoperiodicity}, $(2^{h+n-1}-2^{h+n-2}\lambda_{n-1})$ and $(2^{h+2}-2^h\mathbb{H})$ are $RO(G)$-periodicities of $E_h$ if and only if the orientation classes $u_{2^{h+n-2}\lambda_{n-1}}$ and $u_{2^h\mathbb{H}}$ are permanent cycles in the $\Cn$- and $Q_8$-homotopy fixed point spectral sequences of $E_h$, respectively. The following proposition establishes a close relationship between the essentiality of the above differentials and the sharpness of the periodicities of $EO_h(G)$.

\begin{prop}\label[prop]{lemma:cyclictight}
\,

\begin{enumerate}
 \item At height $h=2^{n-1} m$,  the periodicity of $EO_h(\Cn)$ in \cref{thm:cyclicperiodicity} is sharp if and only if  $(2^{h+n-1}-2^{h+n-2}\lambda_{n-1})$ is not a $RO(\Cn)$-periodicity of $E_h$. 
 \item At height $h=4m+2$, the periodicity of $EO_h(Q_8)$ in \cref{thm:Q8periodicity} is sharp if and only if  $(2^{h+2}-2^{h}\mathbb{H})$ is not a $RO(Q_8)$-periodicity of $E_h$. 
 \end{enumerate}
\end{prop}
\begin{proof}
For (1), the statement is equivalent to that the periodicity of $EO_h(C_{2^n})$ is not sharp if and only if $( 2^{h+n-1} - 2^{h+n-2}\lambda_{n-1} )$ is a $RO(\Cn)$-periodicity for $E_h$. If this class is indeed a periodicity, then consider the following linear combination of $RO(C_{2^n})$-periodicities (see \cref{theorem:ROGperiodicity}):
\[
\begin{aligned}
&\hspace{0.2in} 2^h\Big(1+\sigma_{2^n}+\sum_{i=1}^{\,n-1}2^{\,i-1}\lambda_i\Big)
+2^{\,h-m-1}\big(2^{m+1}-2^{m+1}\sigma_{2^n}\big)\\
&\; + \left( 2^{\,h+n-1}-2^{\,h+n-2}\lambda_{n-1} \right)
+\sum_{s=1}^{\,n-2} 2^{\,h-2^{s}m-1}\big(2^{2^{s}m+s+1}-2^{2^{s}m+s}\lambda_s\big)\\
&\;=\;\sum_{j=1}^{\,n-1}2^{\,h+j}+2^{\,h+1}\\
&\;=\;2^{\,h+n}.
\end{aligned}
\]
This would show that $EO_h(\Cn)$ is $2^{h+n}$-periodic, which is smaller than the $2^{h+n+1}$-periodicity proven in \cref{thm:cyclicperiodicity}.

Conversely, if the periodicity in \cref{thm:cyclicperiodicity} is not sharp, i.e., if  $EO_h(C_{2^n})$ is $2^{h+n}$-periodic. Then by considering the above linear combination again, we deduce that $E_h$ is $(2^{h+n-1} - 2^{h+n-2}\lambda_{n-1})$-periodic.

To prove (2), note that if $(2^{h+2}-2^h \mathbb{H})$ is a periodicity of $E_h$,  then the linear combination
\[
\begin{aligned}
&\hspace{0.2in} 2^h(1+\sigma_i+\sigma_j+\sigma_k+\mathbb{H})
   + \left(2^{h+2} - 2^{h}\mathbb{H}\right) \\[4pt]
&+\,2^{2m}
      \big(2^{2m+2}+2^{2m+2}\sigma_i
           -2^{2m+2}\sigma_j
           -2^{2m+2}\sigma_k\big) \\[6pt]
&+\,2^{2m}
      \big(2^{2m+2}+2^{2m+2}\sigma_j
           -2^{2m+2}\sigma_k
           -2^{2m+2}\sigma_i\big)\\[6pt]
           &+\,2^{2m}
      \big(2^{2m+2}+2^{2m+2}\sigma_k
           -2^{2m+2}\sigma_i
           -2^{2m+2}\sigma_j\big)\\[6pt]
&= 2^h + 2^{h+2} + 3\cdot2^{h} \\[6pt]
   &= 2^{h+3}
\end{aligned}
\]
shows that $EO_h(Q_8)$ is $2^{h+3}$-periodic, which is smaller than the $2^{h+4}$-periodicity proven in \cref{thm:Q8periodicity}.  Conversely, if $EO_h(Q_8)$ is $2^{h+3}$-periodic, then the above linear combination also implies that $E_h$ is $(2^{h+2}-2^h\mathbb{H})$-periodic.
\end{proof}

\section{Computational applications} \label{sec:E2C4}

In this section, we demonstrate a computational application of our periodicity theorem, using \cref{thm:A} as a key input to establish differentials and resolve extension problems in the equivariant slice spectral sequence. We compute the $C_4$-slice spectral sequence of $D^{-1}\BPCfour \langle 1 \rangle$, where $D = N_{C_2}^{C_4}(\bt_1)$.  Based on the work in \cite{BHSZ21}, this periodic Hill--Hopkins--Ravenel theory serves as a model for the $C_4$-equivariant Lubin--Tate theory $E_2$. 

A highlight of our computation is that, unlike the original computation in \cite{HHR17}, which uses many explicit formulas involving restriction, transfer, and norm maps, our computation relies on very few such formulas. In fact, except for the $d_3$-differentials, all other differentials are obtained by drawing on the “qualitative structures” present in the equivariant slice spectral sequence, such as periodicity, transchromatic phenomena, and strong vanishing lines.

For our application, we will specialize the statements below to $D^{-1}\BPCfour \langle 1 \rangle$. The structural theorems themselves apply to all localizations of $\BPCn \langle m \rangle$, not only to $D^{-1}\BPCfour \langle 1 \rangle$. In particular, they also give a more structural and streamlined approach to computations at higher heights (such as the height-4 theory computed in \cite{HSWX23}). Here we work out the case $D^{-1}\BPCfour \langle 1 \rangle$ in detail to illustrate how these structural theorems are used.

\begin{enumerate}

\item \textbf{32-Periodicity in $\pi_*$}. By the Periodicity Theorem (\cref{thm:A}), $\pi_n(D^{-1}\BPCfourone) = \pi_{n-32}(D^{-1}\BPCfourone)$. The presence of this 32-periodicity in $\pi_*$ allows us to deduce higher differentials from known lower differentials, as well as establish extensions on the $\mathcal{E}_\infty$-page, by leveraging the differing distributions of classes and patterns of differentials in the positive and negative cones.

\item \textbf{Transchromatic isomorphism}.  This is established in \cite{MSZ24} and \cite{LiuShiYan2025}. These results allow us to deduce differentials in a certain region of the $C_4$-slice spectral sequence of $D^{-1}\BPCfour \langle 1 \rangle$ ($C_4$-equivariant $E_2$) directly from the $C_2$-slice spectral sequence of $\bv_1^{-1} \BPR \langle 1 \rangle$ ($C_2$-equivariant $E_1$).

\item \textbf{Strong vanishing line}.  The strong horizontal vanishing lines for the equivariant slice spectral sequences of Lubin--Tate theories and Hill--Hopkins--Ravenel theories are proven in \cite{DLS2022}. Applied to the $C_4$-slice spectral sequence of $D^{-1}\BPCfour \langle 1 \rangle$, these results show that there is a strong horizontal vanishing line at filtration 13: all classes with filtration at least 13 must die, and the longest differential in the spectral sequence has length at most 13. 

\item \textbf{Tate periodicity and duality}.  In \cite{LiuShiYan2025}, periodicity elements are established in the Tate spectral sequences of Lubin--Tate theories and Hill--Hopkins--Ravenel theories. These are elements $x$ that are invertible on the $\mathcal{E}_2$-page and have the property that if $d_r(a) = b$ is an essential $d_r$-differential, then $d_r(x \cdot a) = x \cdot b$ and $d_r(x^{-1} \cdot a) = x^{-1} \cdot b$ are also essential $d_r$-differentials. 

The Tate periodicities associated with $D^{-1}\BPCfour \langle 1 \rangle$ are $\dtone^4 u_{4\sigma} a_{4\lambda}$ at $(8, 8)$ and $\dtone^8 u_{8\lambda} u_{8\sigma}$ at $(32, 0)$. These periodicity elements will be used to propagate nontrivial $d_r$-differentials in the positive cone, as well as to propagate the differentials in the positive cone through the generalized Tate spectral sequences to the negative cone. Given these Tate periodicity elements, it suffices to determine the differentials in the positive cone.

\item \textbf{Restriction, transfer, and norm maps}.  Traditionally, the restriction, transfer, and norm maps between the $C_4$-slice spectral sequence and the $C_2$-slice spectral sequence of $D^{-1}\BPCfour \langle 1 \rangle$ have been used extensively in \cite{HHR17} to compute differentials. While powerful, this approach requires explicit computations using equivariant formulas associated with these maps and can become computationally intensive as the lengths of differentials increase. It is manageable for short differentials but becomes increasingly difficult for longer ones. In our computation, we use this method only to establish the $d_3$-differentials, which are the shortest differentials in the spectral sequence and whose computations using this method are very straightforward.
\end{enumerate}

\subsection{The \texorpdfstring{$\mathcal{E}_2$}{}-page}

The slice associated graded of $D^{-1}\BPCfourone$ is $D^{-1}H\Z[\tone, \gtone]$, and the $RO(G)$-graded $\mathcal{E}_2$-page is $D^{-1}{H\Z}_\star[\tone, \gtone]$. The distributions of classes in the integer-graded slice spectral sequences are as follows: the $C_2$-slice spectral sequence is concentrated in the first and third quadrants, bounded by the lines of slope 0 and 1 through the origin. The $C_4$-slice spectral sequence is concentrated in the first and third quadrants, bounded between the lines of slopes 0 and 3 through the origin.

\subsection{The \texorpdfstring{$C_2$}{}-differentials}
Using Hu--Kriz's computation of the $C_2$-slice spectral sequence of $\BPR$ \cite{HuKriz, HHR16}, we obtain complete knowledge of the differentials in the positive cone of the $C_2$-slice spectral sequence of $D^{-1}\BPCfour \langle 1 \rangle$ by using the composition map of spectral sequences
\[C_2\text{-}\SliceSS(\BPR) \longrightarrow C_2\text{-}\SliceSS(\BPCfour) \longrightarrow C_2\text{-}\SliceSS(D^{-1}\BPCfour\langle 1 \rangle).\]
There are only $d_3$- and $d_7$-differentials. All the $d_3$-differentials in the positive cone are generated by
\[d_3(u_{2\sigma_2}) = \bv_1 a_{\sigma_2}^3 = (\tone + \gtone) a_{\sigma_2}^3,\]
and all the $d_7$-differentials are generated by
\[d_7(u_{4\sigma_2}) = \bv_2 a_{\sigma_2}^7 = \tone^3 a_{\sigma_2}^7.\]
Using Tate periodicity and duality ((4) above), we then obtain all the differentials in the negative cone from those in the positive cone.

\subsection{\texorpdfstring{$d_3$}{}-differentials} \label{subsec:d3diff}

The $d_3$-differentials in the $C_4$-$\SliceSS$ can all be directly determined via the restriction map to the $C_2$-$\SliceSS$, whose differentials have been completely computed.  In particular, there is a $d_3$-differential $d_3(x) = y$ in the $C_4$-$\SliceSS$ if and only if $d_3(\res(x)) = \res(y)$ in the $C_2$-$\SliceSS$. A more detailed explanation can be found in \cite[Section~3.2]{HSWX23}. 

\subsection{Transchromatic differentials}

By applying the transchromatic isomorphism ((2) above), the $C_2$-slice spectral sequence of $\bv_1^{-1}\BPR \langle 1 \rangle$ and the $C_4$-slice spectral sequence of $N(\bt_1)^{-1}\BPCfour \langle 1 \rangle$ are shearing isomorphic. Therefore, the $C_2$-differentials in $C_2$-$\SliceSS(\bv_1^{-1}\BPR)$ (which are all $d_3$-differentials) completely determine the differentials in the $C_4$-slice spectral sequence of $N(\bt_1)^{-1}\BPCfour \langle 1 \rangle$, and hence of $D^{-1}\BPCfour \langle 1 \rangle$, in the region between the lines of slope 1 and slope 3 (all $d_5$-differentials).

It remains to compute the rest of the differentials, which all originate in the region between the lines of slope 0 and slope 1. By using Tate periodicity and duality ((4) above), it suffices to compute all the differentials in the positive cone.

\subsection{Repeating \texorpdfstring{$bo$}{}-patterns} \label{subsec:boPatterns}

Before discussing the higher differentials, we will first take off repeating patterns that do not contribute to higher differentials.  More specifically, the $d_3$-differentials obtained in \cref{subsec:d3diff} produce a series of 8-periodic patterns in low filtrations ($-2$, $-1$, $0$, $1$, $2$).  To describe these patterns, we will use the notation $\langle a, b \rangle$, $a \geq b$ to denote the (regular) slice cell indexed by $\{t_1^a \gamma t_1^a\}$ if $a = b$, and the (induced) slice cell indexed by $\{t_1^a \gamma t_1^b, t_1^b \gamma t_1^a\}$ if $a \neq b$. 

There are two types of 8-periodic patterns.  The first type is shown in \cref{fig:E2C4RepeatPatternFirst} and has classes from both regular slice cells and induced slice cells.

\begin{figure}
\begin{center}
\makebox[\textwidth]{\includegraphics[trim={0cm 21.5cm 0cm 0cm}, clip, scale = 0.6, page = 1]{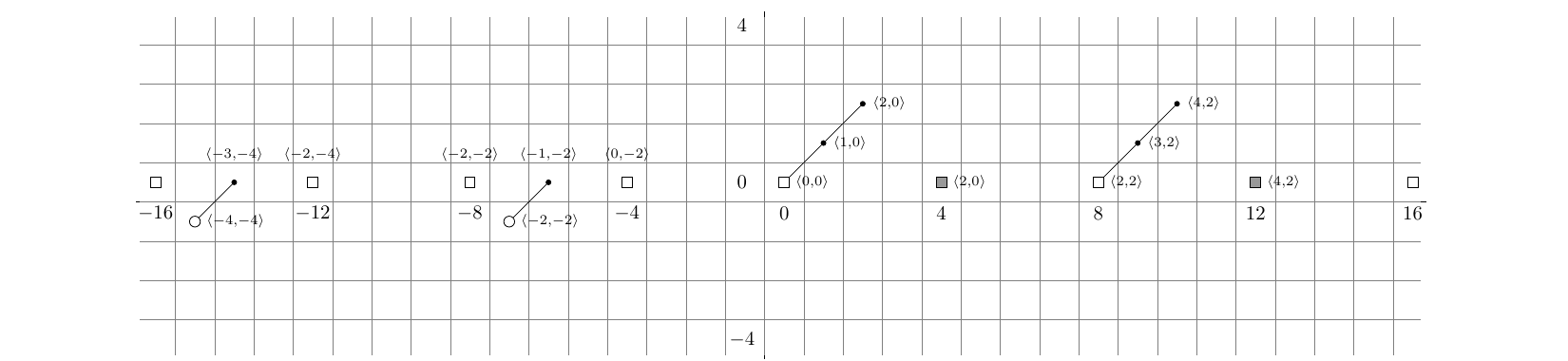}}
\caption{First type of 8-periodic pattern formed after $d_3$-differentials.}
\hfill
\label{fig:E2C4RepeatPatternFirst}
\end{center}
\end{figure}

The second type is shown in \cref{fig:E2C4RepeatPatternSecond}, and has classes from only induced slice cells.  The pattern that these classes form is called the \textit{$bo$-pattern}, as it matches the pattern of copies of $\pi_* KO$. These classes all survive to the $\mathcal{E}_\infty$-page. Consequently, we can ignore them when determining higher differentials.

\begin{figure}
\begin{center}
\makebox[\textwidth]{\includegraphics[trim={0cm 21.5cm 0cm 0cm}, clip, scale = 0.6, page = 2]{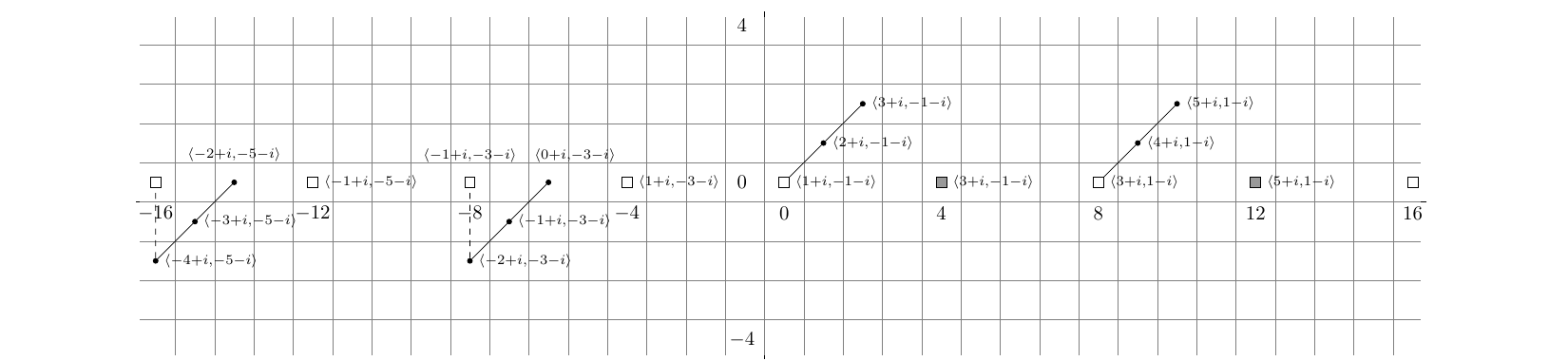}}
\caption{Second type of 8-periodic pattern ($bo$-pattern) formed after $d_3$-differentials ($i \geq 0$).}
\hfill
\label{fig:E2C4RepeatPatternSecond}
\end{center}
\end{figure}

The survival of these classes can be explained using the Tate spectral sequence. A general method for this is described in \cite{LiuShiYan2025} (see also \cite[Method~2.7, Proposition~4.12]{DKLLW24} for discussions of $bo$-patterns in the homotopy fixed point spectral sequence). In this case, since we have determined all the $d_3$-differentials in the slice spectral sequence, we also determine all the $d_3$-differentials in the Tate spectral sequence. Some $d_3$-differentials in the Tate spectral sequence have sources in filtration $<0$ and targets in filtration $\geq 0$ (see \cref{fig:E2C4TateSSd3}). These differentials ``cross'' the horizontal line $s = 0$ and do not occur in the slice spectral sequence.  Consider the maps
\[\SliceSS(X) \longrightarrow \TateSS(X),\]
which induce an isomorphism in the positive cone, and 
\[\Sigma^{-1}\TateSS(X) \longrightarrow \SliceSS(X),\]
which induce an isomorphism in the negative cone.  Using these maps, we deduce that the targets of the crossing differentials contribute the classes in the $bo$-pattern in the positive cone (via the first map), and the sources of the crossing differentials contribute to the classes in the $bo$-pattern in the negative cone (via the second map).

\begin{figure}
\begin{center}

\makebox[\textwidth]{\includegraphics[trim={0cm 21.5cm 0cm 0cm}, clip, scale = 0.6, page = 3]{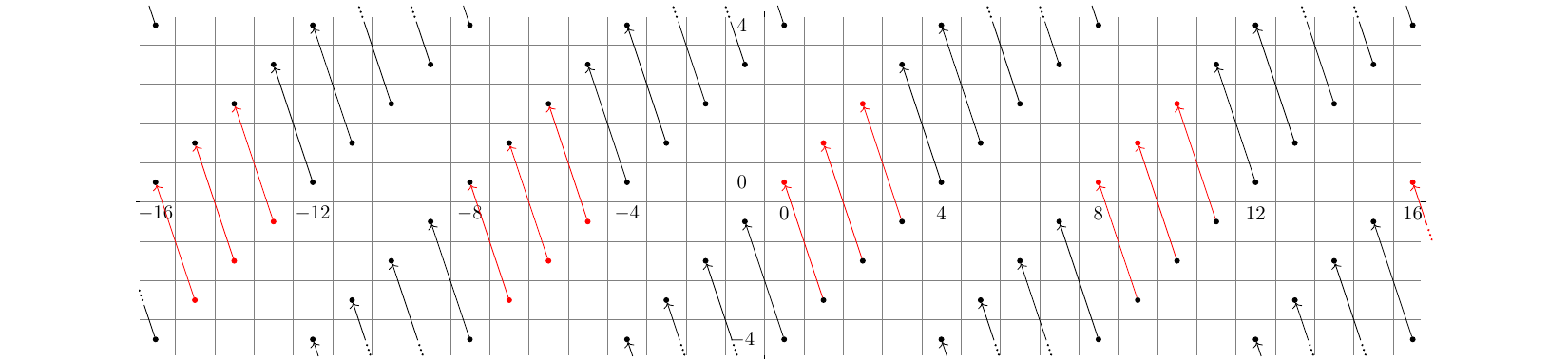}}
\caption{The $d_3$-differentials in the Tate spectral sequence and classes from induced slice cells contributing to the $bo$-pattern.}
\hfill
\label{fig:E2C4TateSSd3}
\end{center}
\end{figure}

In the argument that follows, we will focus on the bidegree of the elements and will only occasionally specify an explicit formula for a class in that bidegree when necessary. This approach emphasizes that explicit computations with formulas are not required for determining higher differentials; it suffices to just know the bidegree and to use the computational structures in the slice spectral sequence described above.


\subsection{Higher differentials}\hfill
\vspace{0.1in}

\noindent \textbf{$d_5$-differentials.}  The $d_5$-differentials can be determined by comparing the classes in the 32-stem with those in the 0-stem. All the classes in the 0-stem lie in bidegree $(0, 0)$ and are torsion-free. Since every class in the 32-stem with positive filtration is torsion, they must all be killed by differentials.

In particular, the classes $\dtone^{16} u_{16\sigma}a_{16\lambda}$ and $2\dtone^{16} u_{16\sigma}a_{16\lambda}$ at $(32, 32)$ must both be killed. The only way for this to happen is for $2\dtone^{16} u_{16\sigma}a_{16\lambda}$ to be hit by a $d_5$-differential from $(33, 27)$, and for $\dtone^{16} u_{16\sigma}a_{16\lambda}$ to be hit by a $d_{13}$-differential from $(33, 19)$. The $d_5$-differential is
\[d_5(\dtone^{15}u_{2\lambda}u_{14\sigma}a_{13\lambda}a_\sigma) = 2\dtone^{16} u_{16\sigma}a_{16\lambda} \hspace{0.2in} (d_5(33, 27) = 2(32, 32)).\]
Applying the $(8, 8)$-Tate periodicity given by the element $\dtone^4 u_{4\sigma} a_{4\lambda}$, we obtain the following $d_5$-differential:
\[d_5(\dtone^{3}u_{2\lambda}u_{2\sigma}a_{\lambda}a_\sigma) = 2\dtone^{4} u_{4\sigma}a_{4\lambda} \hspace{0.2in} (d_5(9, 3) = 2(8, 8)).\]
We will now apply the Leibniz rule to this $d_5$-differential with respect to the following classes:
\begin{enumerate}
\item The class $\dtone a_\lambda a_\sigma$ at $(1, 3)$ is a permanent cycle. Applying the Leibniz rule using the differential $d_5(9, 3) = 2(8, 8)$ and this class produces the following $d_5$-differential:
\[ d_5(\dtone^2 u_{2\lambda} u_{2\sigma}) = \dtone^3 u_\lambda u_{2\sigma} a_{2\lambda} a_\sigma \hspace{0.2in} (d_5(8, 0) = (7, 5)).\]

\item The class $\dtone^2 u_{2\sigma} a_{2\lambda}$ at $(4, 4)$ supports the $d_5$-differential 
\[d_5(\dtone^2 u_{2\sigma} a_{2\lambda}) = \dtone^3 a_{3\lambda} a_{3\sigma} \hspace{0.2in} (d_5(4,4) = (3, 9)),\] 
which is obtained from the transchromatic isomorphism. The target of this differential kills all classes below the line of slope 1, due to the gold relation $u_{\lambda} a_{3\sigma} = 2 u_{2\sigma} a_{\lambda} a_\sigma = 0$.

\item The class $\dtone u_{\lambda} a_\sigma$ at $(3, 1)$ cannot support a $d_5$-differential, since doing so would contradict the differential $d_5(4, 4) = (3, 9)$ above. Therefore, this class is a surviving permanent cycle. Applying the Leibniz rule using the differential $d_5(8, 0) = (7, 5)$ and this class produces the following $d_5$-differential: 
\[d_5(\dtone^3 u_{3\lambda}u_{2\sigma} a_\sigma) = 2 \dtone^4 u_{\lambda}u_{4\sigma} a_{3\lambda} \hspace{0.2in} (d_5(11,1) = 2(10, 6)).\]

\item The class $\dtone^4 u_{4\lambda} u_{4\sigma}$ at $(16, 0)$ is the square of the class $\dtone^2 u_{2\lambda} u_{2\sigma}$ at $(8, 0)$, and is therefore a $d_5$-cycle. 
\end{enumerate}

Take the following three $d_5$-differentials that we have obtained: 
\begin{alignat*}{2}
d_5(\dtone^2 u_{2\lambda} u_{2\sigma}) &= \dtone^3 u_\lambda u_{2\sigma} a_{2\lambda} a_\sigma \hspace{0.2in} &&(d_5(8, 0) = (7, 5)) \\
d_5(\dtone^{3}u_{2\lambda}u_{2\sigma}a_{\lambda}a_\sigma) &= 2\dtone^{4} u_{4\sigma}a_{4\lambda} \hspace{0.2in}&& (d_5(9, 3) = 2(8, 8)) \\
d_5(\dtone^3 u_{3\lambda}u_{2\sigma} a_\sigma) &= 2 \dtone^4 u_{\lambda}u_{4\sigma} a_{3\lambda} \hspace{0.2in}&& (d_5(11,1) = 2(10, 6))
\end{alignat*}
By repeatedly applying the Leibniz rule to these differentials, using the classes $\dtone^2 u_{2\sigma} a_{2\lambda}$ at $(4, 4)$ and $\dtone^4 u_{4\lambda} u_{4\sigma}$ at $(16, 0)$, we obtain all the $d_5$-differentials in the positive cone.  Applying Tate periodicity and duality then produces all the $d_5$-differentials in the negative cone.  See \cref{fig:E2C4d5}. 

\begin{figure}
\begin{center}
\makebox[\textwidth]{ \includegraphics[trim={0cm 0cm 0cm 0cm}, clip, scale = 0.45]{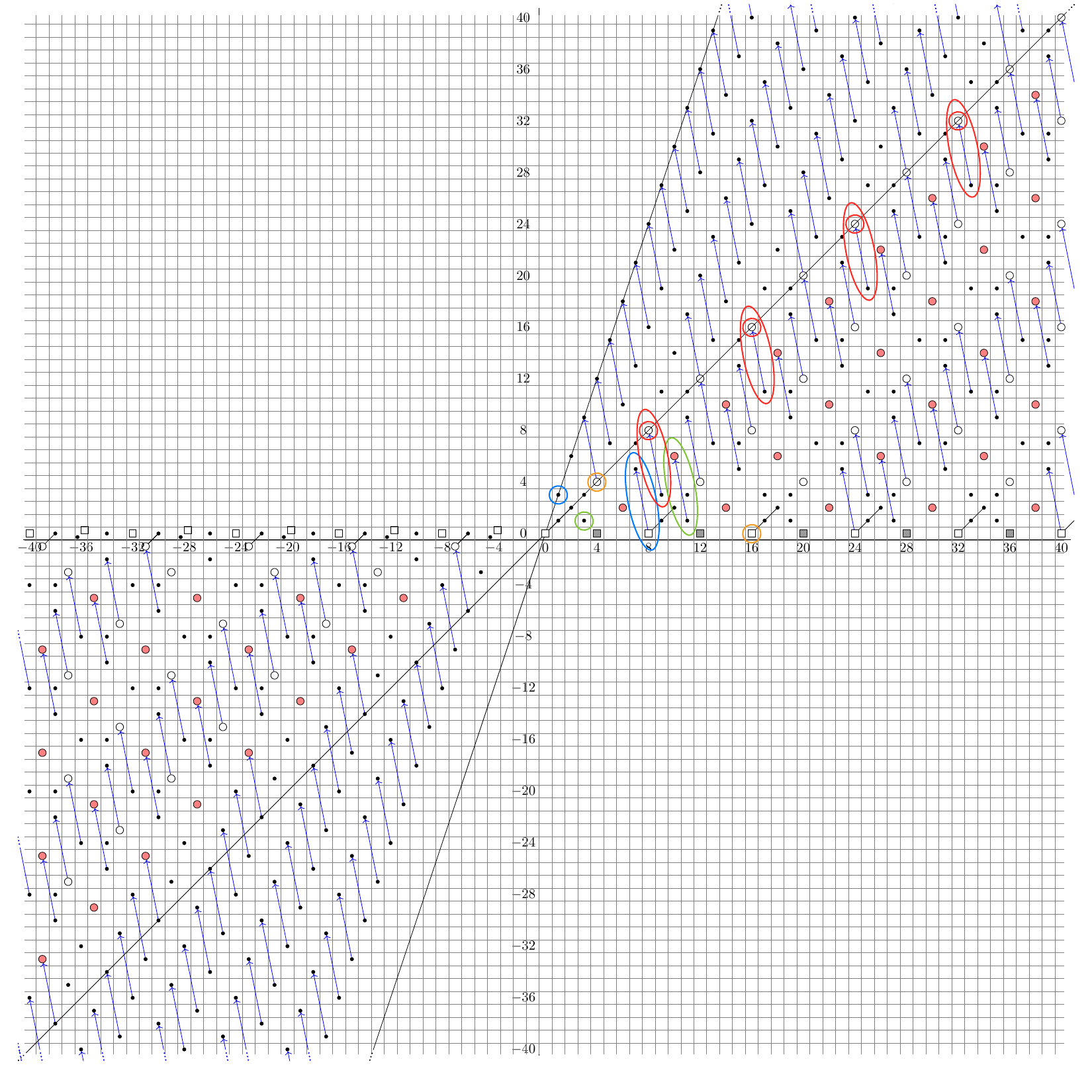}}
\caption{$d_5$-differentials. The $d_5$-differentials between the lines of slope 1 and 3 are obtained using the transchromatic isomorphism.}
\hfill
\label{fig:E2C4d5}
\end{center}
\end{figure}

\vspace{0.1in}

\noindent \textbf{The rest of the differentials.}  To determine the rest of the differentials, we make further use of the 32-periodicity in $\pi_*$. We begin by examining the classes in stems $29$ through $32$. Since $\pi_{-3} = \pi_{-2} = \pi_{-1} = 0$ and $\pi_0$ is torsion-free, all classes in stems $29$, $30$, and $31$, as well as the classes in positive filtration in the 32-stem, must die.

\begin{enumerate}
\item The class at $(30, 2)$ must die. The only possibility is that it supports a differential of length 13, hitting the class at $(29, 15)$.

\item The class at $(30, 18)$ must also die, and the only possibility is that it is the target of a $d_{11}$-differential supported by the class at $(31, 7)$.

\item The classes at $(31, 15)$, $(31, 23)$, and $(31, 31)$ must all die. These classes cannot support differentials, so the only possibility is that they are the targets of $d_7$-differentials originating from $(32, 8)$, $(32, 16)$, and $(32, 24)$, respectively. This because by the strong vanishing line result, the maximal differential length in the spectral sequence is 13.

\item The class at $(32, 32)$ must die.  The only possibility is for it to be killed by a $d_{13}$-differential originating from $(33, 19)$.
\end{enumerate}
These differentials can then be propagated using the Tate periodicity classes $\dtone^4 u_{4\sigma} a_{4\lambda}$ at $(8, 8)$ and $\dtone^8 u_{8\lambda} u_{8\sigma}$ at $(32, 0)$.  See \cref{fig:E2C4HigherStep1}.

\begin{figure}
\begin{center}
\makebox[\textwidth]{\includegraphics[trim={0cm 0cm 0cm 0cm}, clip, scale = 0.45]{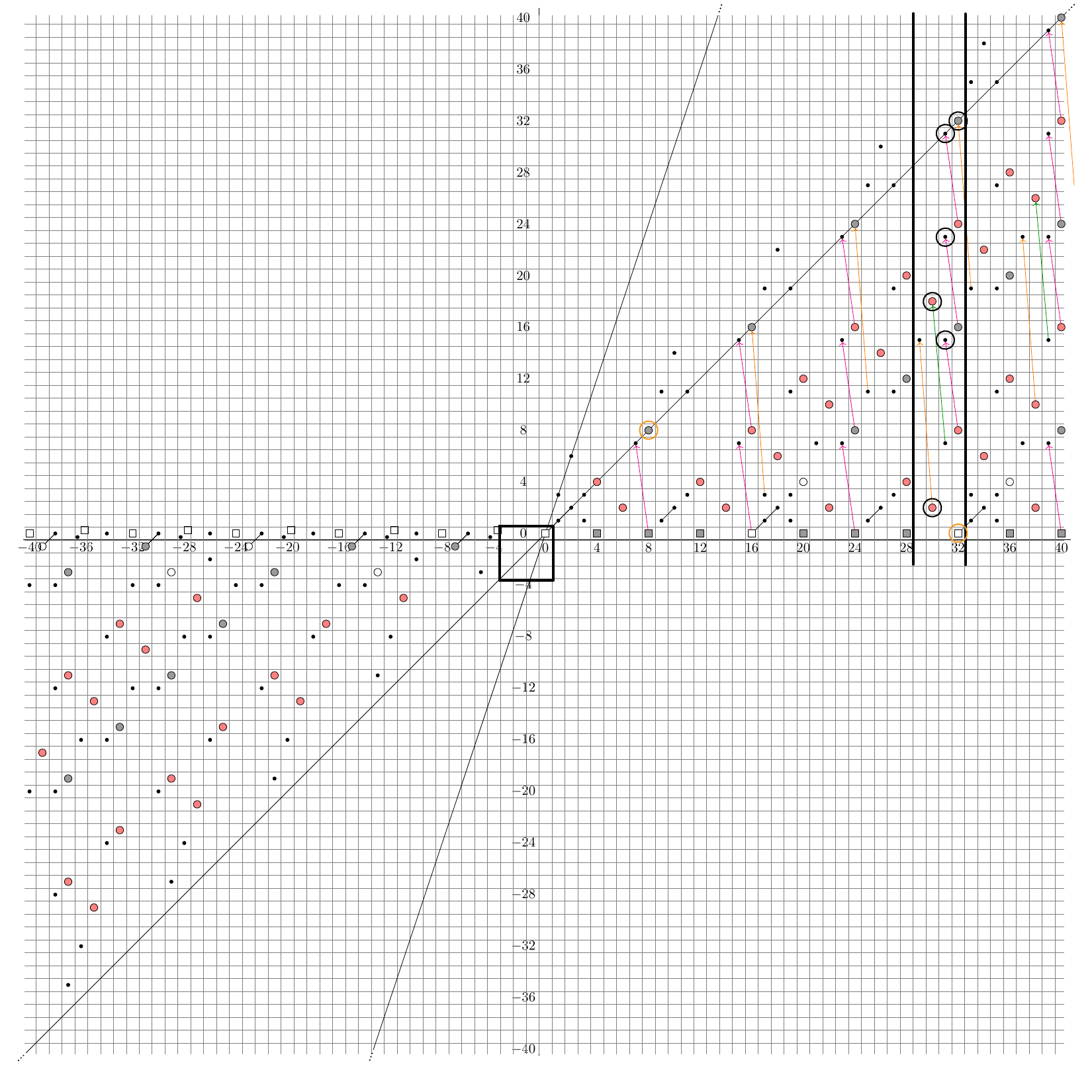}}
\caption{The higher differentials in the positive cone that are obtained by considering the classes in stems 29, 30, 31, and 32.}
\hfill
\label{fig:E2C4HigherStep1}
\end{center}
\end{figure}

Next, we will consider the classes in stems $25$, $26$, $27$, and $28$.
\begin{enumerate}
\item There are three classes in the 25-stem, located at $(25, 1)$, $(25, 11)$, and $(25, 27)$. The class at $(25, 11)$ supports a $d_{13}$-differential, obtained by propagating the differential $d_{13}(33, 19) = (32, 32)$ above via the $(8, 8)$-Tate periodicity. The class at $(25, 27)$ lies in filtration at least 13, which is above the horizontal vanishing line, and therefore it must die. The only possibility is that it is the target of a $d_{13}$-differential supported by the class at $(26, 14)$. The class at $(25, 1)$ neither supports nor is killed by any differential, and therefore it survives to the $\mathcal{E}_\infty$-page. Under the 32-periodicity in $\pi_*$, it corresponds to the class at $(-7, -1)$.

\item There are two classes remaining in the 26-stem, located at $(26, 2)$ and $(26, 30)$. The class at $(26, 30)$ must die because it lies above the horizontal vanishing line. By degree reasons, it must be the target of a $d_{11}$-differential supported by the class at $(27, 19)$ (it cannot be the target of a $d_3$-differential from $(29, 27)$, since all the $d_3$-differentials have already been determined). The class at $(26, 2)$ survives to the $\mathcal{E}_\infty$-page and corresponds to the class at $(-6, 0)$ under the 32-periodicity in $\pi_*$.

\item There are three remaining classes in the 27-stem, located at $(27, 3)$, $(27, 11)$, and $(27, 27)$. In stem $(-5)$, there is a single class at $(-5, -3)$, which survives. This class must correspond to the class at $(27, 3)$, which also survives by degree reasons. Therefore, the classes at $(27, 11)$ and $(27, 27)$ must both die. The only possibility is for them to be killed by $d_7$-differentials supported by the classes at $(28, 4)$ and $(28, 20)$, respectively.

\item There is one remaining class in the 28-stem, located at $(28, 12)$. By degree reasons, it must survive to the $\mathcal{E}_\infty$-page. It corresponds to the class at $(-4, 0)$ under the 32-periodicity in $\pi_*$.
\end{enumerate}
These differentials can then be propagated using the $(8, 8)$- and $(32, 0)$-Tate periodicity classes described above.  See \cref{fig:E2C4HigherStep2}.

To finish the computation of differentials, we just need to consider the remaining classes in the 36-stem: the class $\dtone^{10}u_{8\lambda}u_{10\sigma}a_{2\lambda}$ at $(36, 4)$ (4-torsion) and the class $\dtone^{14}u_{4\lambda}u_{14\sigma}a_{10\lambda}$ at $(36, 20)$ (2-torsion). There is only one class in stem 4, a 2-torsion class at $(4, 4)$, which survives to the $\mathcal{E}_\infty$-page. This implies that the generator at $(36, 4)$ must die. The only possibility is for this class to support a $d_7$-differential, hitting the class at $(35, 11)$. The class at $(36, 20)$ must also die because it lies above the horizontal vanishing line, and the only possibility is for it to be killed by a $d_{13}$-differential supported by the class at $(37, 7)$. The class $2\dtone^{10}u_{8\lambda}u_{10\sigma}a_{2\lambda}$ at $(36, 4)$ survives to the $\mathcal{E}_\infty$-page and corresponds to the class at $(4, 0)$ under the 32-periodicity in $\pi_*$.

Propagating these two differentials using the $(8, 8)$- and $(32, 0)$-Tate periodicity classes produces all the remaining differentials in the positive cone. Applying Tate periodicity and duality then produces all the differentials in the negative cone.  See \cref{fig:E2C4SSSFullAllDifferentials} and \cref{fig:E2C4SSSFullEinfty}.

\begin{figure}
\begin{center}
\makebox[\textwidth]{\includegraphics[trim={0cm 0cm 0cm 0cm}, clip, scale = 0.45]{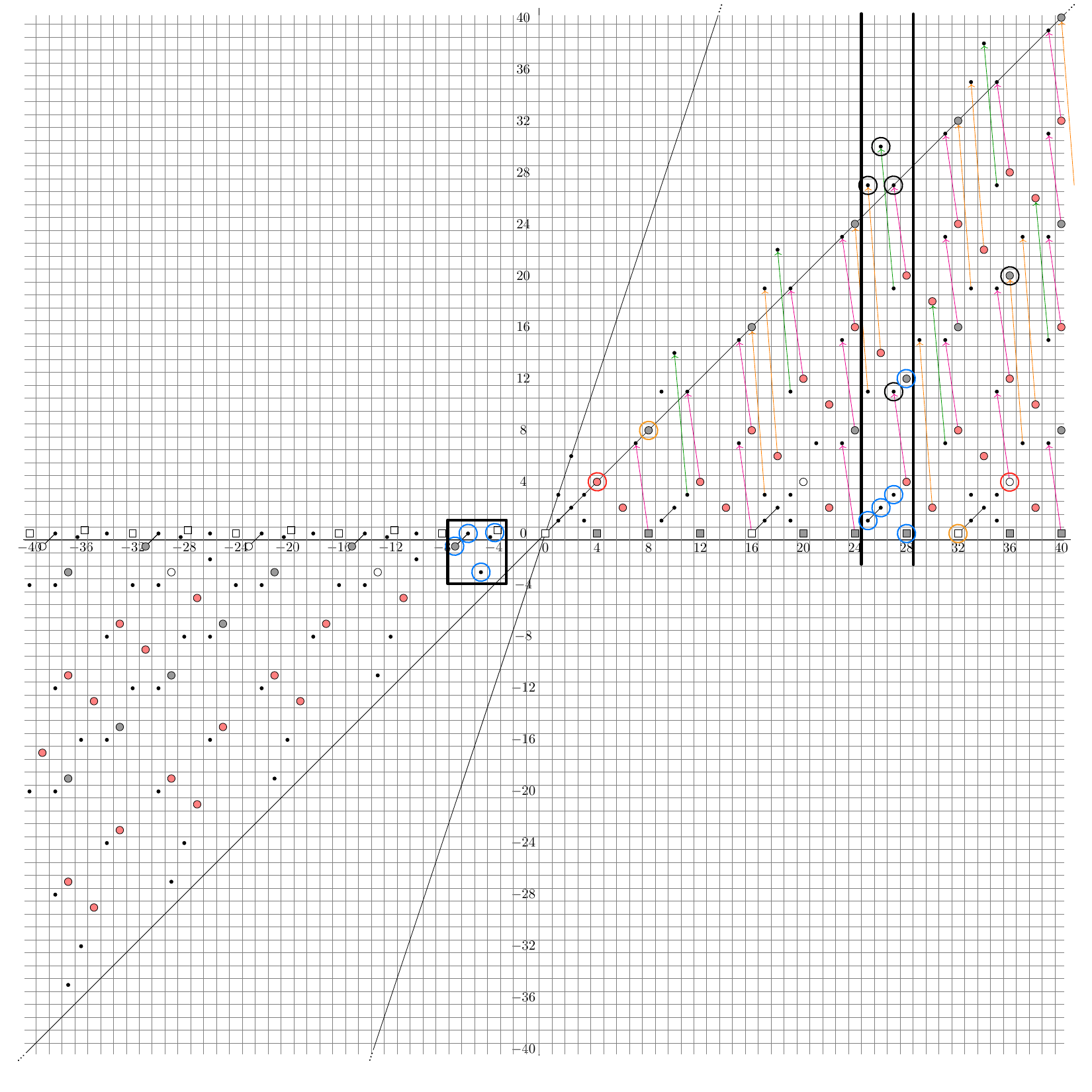}}
\caption{The rest of the higher differentials in the positive cone, obtained by considering the classes in stems 25, 26, 27, 28, and 36.}
\hfill
\label{fig:E2C4HigherStep2}
\end{center}
\end{figure}

\begin{figure}
\begin{center}
\makebox[\textwidth]{\includegraphics[trim={0cm 0cm 0cm 0cm}, clip, scale = 0.45, page = 1]{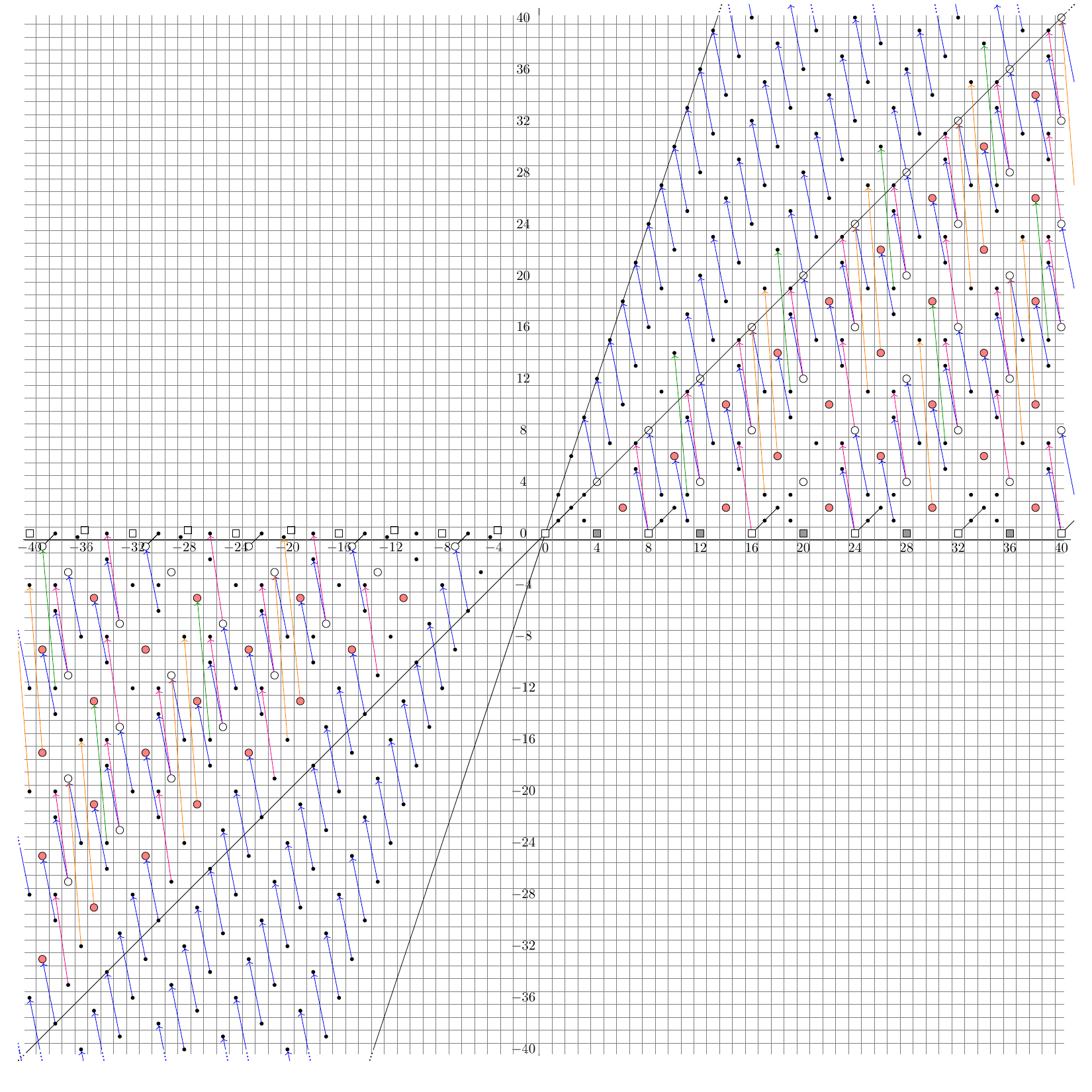}}
\caption{All the higher differentials in the slice spectral sequence of $D^{-1}\BPCfourone$.}
\hfill
\label{fig:E2C4SSSFullAllDifferentials}
\end{center}
\end{figure}

\begin{figure}
\begin{center}
\makebox[\textwidth]{\includegraphics[trim={0cm 0cm 0cm 0cm}, clip, scale = 0.45, page = 6]{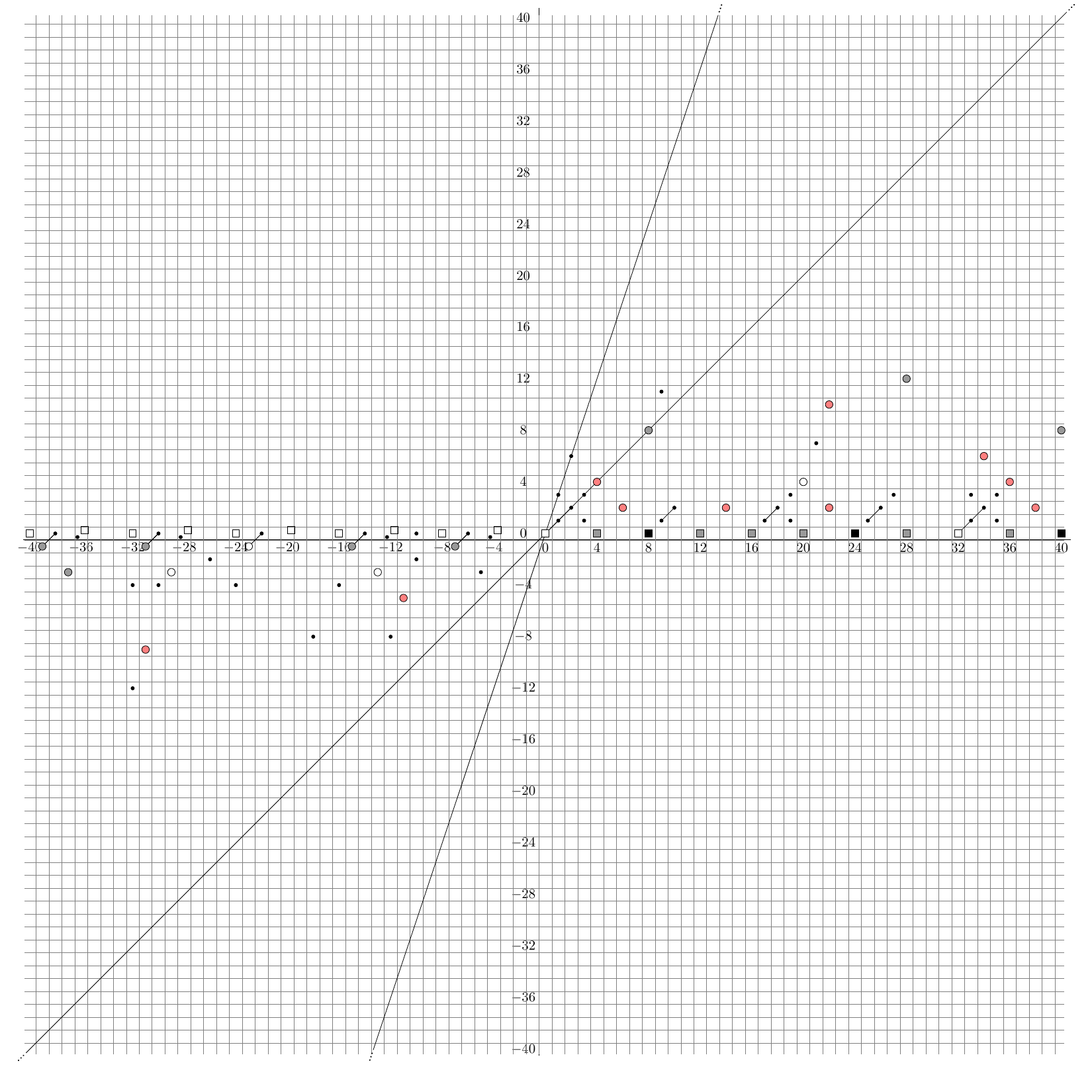}}
\caption{The $\mathcal{E}_\infty$-page of the slice spectral sequence of $D^{-1}\BPCfourone$ (without the repeating $bo$-patterns). }
\hfill
\label{fig:E2C4SSSFullEinfty}
\end{center}
\end{figure}

\subsection{2-extensions}
We will now solve all the 2-extensions on the $\mathcal{E}_\infty$-page. Our main tool will be the $32$-periodicity in $\pi_*$. For the arguments below, refer to \cref{fig:E2C4SSS2Extensions}.

\begin{figure}
\begin{center}
\makebox[\textwidth]{\includegraphics[trim={0cm 13.5cm 0cm 0cm}, clip, scale = 0.45, page = 1]{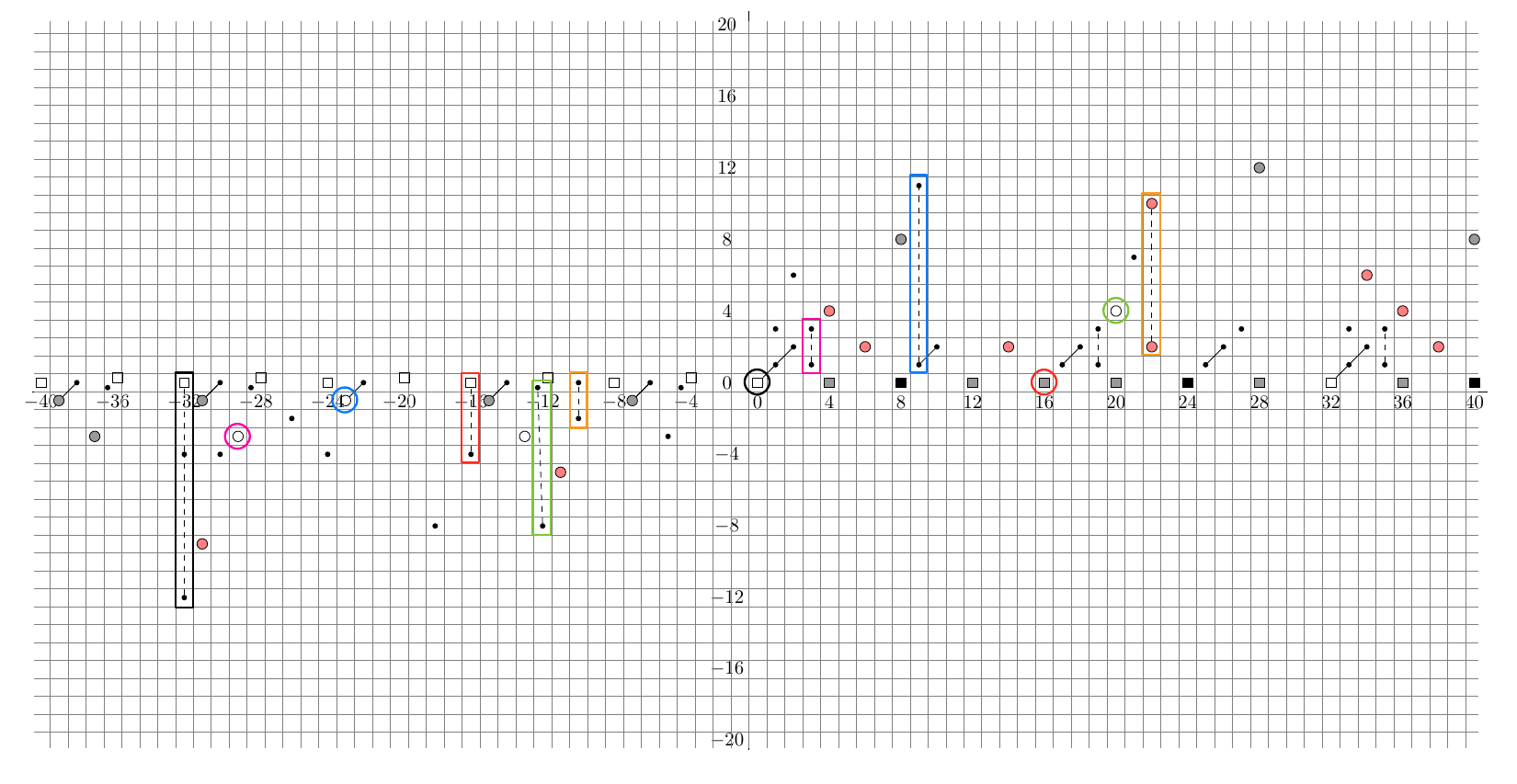}}
\caption{2-extensions on the $\mathcal{E}_\infty$-page of the slice spectral sequence of $D^{-1}\BPCfourone$. }
\hfill
\label{fig:E2C4SSS2Extensions}
\end{center}
\end{figure}

\begin{enumerate}
\item $\pi_{-32}$ and $\pi_0$: in the $(-32)$-stem, there is a $\mathbb{Z}$-class at $(-32, 0)$, a $\mathbb{Z}/2$-class at $(-32, -4)$, and a $\mathbb{Z}/2$-class at $(-32, -12)$. In the $0$-stem, there is only a $\mathbb{Z}$-class at $(0, 0)$.  Therefore, by the 32-periodicity, there are two 2-extensions, from $(-32, -12)$ to $(-32, -4)$ and from $(-32, -4)$ to $(-32, 0)$. 

\item $\pi_1$, $\pi_2$, $\pi_{-31}$, and $\pi_{-30}$: in the 1-stem, there are two $\mathbb{Z}/2$-classes.  They are the class $a = \dtone u_{\lambda} a_{\sigma}$ at $(1, 1)$ and the class $b = \dtone a_{\lambda} a_\sigma$ at $(1, 3)$.  In the 2-stem, there are two $\mathbb{Z}/2$-classes, the class $a^2 = \dtone^2 u_{2\lambda} a_{2\sigma}$ at $(2, 2)$ and the class $b^2 = \dtone^2 a_{2\lambda} a_{2\sigma}$ at $(2, 6)$.  

We will first show that there cannot be a 2-extension in the 1-stem from $a$ to $b$.  Suppose the contrary, that $b = 2a$ in $\pi_1$.  Then we must have $b^2 = 4a^2$ in $\pi_2$.  This is impossible.  Therefore the classes in $\pi_1$ and $\pi_{-31}$ are all 2-torsion.  In the 2-stem, there cannot be a 2-extension from $a^2$ to $b^2$, because if $b^2 = 2a^2$, then $2a$ must be nonzero in $\pi_1$, which is a contradiction to classes in $\pi_1$ all being 2-torsion.  

\item $\pi_{-29}$ and $\pi_3$: in the $(-29)$-stem, there is a $\mathbb{Z}/4$-class at $(-29, -3)$.  In the 3-stem, there are two $\mathbb{Z}/2$-classes, at $(3, 1)$ and $(3, 3)$.  By 32-periodicity, there is a 2-extension from $(3, 1)$ to $(3, 3)$.

\item $\pi_{-24}$ and $\pi_{8}$: in the 8-stem, the $\mathbb{Z}$-class at $(8, 0)$ cannot have a 2-extension to the $\mathbb{Z}/2$-class at $(8, 8)$.  Therefore by 32-periodicity, there is no 2-extension between from $\mathbb{Z}/2$-class at $(-24, -4)$ to the $\mathbb{Z}$-class at $(-24, 0)$ in the $(-24)$-stem.

\item $\pi_{-23}$ and $\pi_{9}$: in the $(-23)$-stem, there is a $\mathbb{Z}/4$-class.  In the $9$-stem, there are two $\mathbb{Z}/2$-classes, at $(9,1)$ and $(9, 11)$. By 32-periodicity, there is a 2-extension from $(9, 1)$ to $(9, 11)$.

\item $\pi_{-16}$ and $\pi_{16}$: in the $(-16)$-stem, there is a $\mathbb{Z}$-class at $(-16, 0)$ and a $\mathbb{Z}/2$-class at $(-16, -4)$.  In the $16$-stem, there is only a $\mathbb{Z}$-class at $(16, 0)$.  By 32-periodicity, there is a 2-extension from $(-16, -4)$ to $(-16, 0)$.  

\item $\pi_{-13}$ and $\pi_{19}$: in the $(-13)$-stem, there is a $\mathbb{Z}/4$-class at $(-13, -3)$.  In the $19$-stem, there are two $\mathbb{Z}/2$-classes, at $(19, 1)$ and $(19, 3)$.  By 32-periodicity, there is a 2-extension from $(19, 1)$ to $(19, 3)$. 

\item $\pi_{-12}$ and $\pi_{20}$: in the $(-12)$-stem, there are two $\mathbb{Z}/2$-classes, at $(-12, 0)$ and $(-12, -8)$.  In the $20$-stem, there is one $\mathbb{Z}/4$-class. By 32-periodicity, there is a 2-extension from $(-12, -8)$ to $(-12, 0)$. 

\item $\pi_{-10}$ and $\pi_{22}$: there are two $\mathbb{Z}/2$-classes in the $(-10)$-stem, $(-10, -2)$ and $(-10, 0)$, and two $\mathbb{Z}/2$-classes in the $22$-stem, $(22, 2)$ and $(22, 10)$.  We will show that there are 2-extensions from $(-10, -2)$ to $(-10, 0)$ and from $(22, 2)$ to $(22, 10)$.  First, we will show that even though the class $(22, 2)$ restricts to 0 on the $\mathcal{E}_2$-page, its restriction is nonzero in homotopy.  To do so, we make use of part of the computation of the $RO(G)$-graded spectral sequence to show that it is not divisible by $a_\sigma$.  By a computation of the $\mathcal{E}_2$-page, there are no classes at $(22+\sigma, 0)$ on the $\mathcal{E}_2$-page.  Moreover, the class at $(22+\sigma, 1)$ supports a nontrivial $d_5$-differential, and hence do not survive to the $\mathcal{E}_\infty$-page.  This shows that $(22, 2)$ is not divisible by $a_\sigma$, and therefore must have nontrivial restriction.

In the Tate spectral sequence, the class $(22, 2)$ is killed by a $d_5$-differential supported by a class at $(23, -3)$, which, after applying the 32-periodicity, corresponds to the class at $(-10, -2)$ in the slice spectral sequence.  Similarly, the class $(22, 10)$ is killed by a $d_{11}$-differential that is supported by the class $(23, -1)$, which, after applying the 32-periodicity, corresponds to the class at $(-10, 0)$ in the slice spectral sequence.  Since we have established that the class $(22, 2)$ has nontrivial restriction in homotopy, the class $(-10, -2)$ must also have nontrivial restriction in homotopy, and the only possible image is the class at $(-10, 0)$ on the $\mathcal{E}_\infty$-page in the $C_2$-spectral sequence by degree reasons.  The transfer of this class is exactly the class $(-10, 0)$ in the $C_4$-spectral sequence. Therefore, we have a $2$-extension from $(-10,-2)$ to $(-10,0)$ since $\tr\circ \res =2$.  By 32-periodicity, there is also a 2-extension from $(22, 2)$ to $(22, 10)$.

\item The classes in the $bo$-patterns described in \cref{subsec:boPatterns} admit no further 2-extensions for degree reasons.
\end{enumerate}

\bibliographystyle{alpha}
\bibliography{ref}

@article {Adams60,
    AUTHOR = {Adams, J. F.},
     TITLE = {On the non-existence of elements of {H}opf invariant one},
   JOURNAL = {Ann. of Math. (2)},
  FJOURNAL = {Annals of Mathematics. Second Series},
    VOLUME = {72},
      YEAR = {1960},
     PAGES = {20--104},
      ISSN = {0003-486X},
   MRCLASS = {55.40},
  MRNUMBER = {141119},
MRREVIEWER = {M. A. Kervaire},
       DOI = {10.2307/1970147},
       URL = {https://doi.org/10.2307/1970147},
}

@article {Adams62,
    AUTHOR = {Adams, J. F.},
     TITLE = {Vector fields on spheres},
   JOURNAL = {Ann. of Math. (2)},
  FJOURNAL = {Annals of Mathematics. Second Series},
    VOLUME = {75},
      YEAR = {1962},
     PAGES = {603--632},
      ISSN = {0003-486X},
   MRCLASS = {57.30},
  MRNUMBER = {139178},
MRREVIEWER = {M. F. Atiyah},
       DOI = {10.2307/1970213},
       URL = {https://doi.org/10.2307/1970213},
}

@article {AdamsAtiyah66,
    AUTHOR = {Adams, J. F. and Atiyah, M. F.},
     TITLE = {{$K$}-theory and the {H}opf invariant},
   JOURNAL = {Quart. J. Math. Oxford Ser. (2)},
  FJOURNAL = {The Quarterly Journal of Mathematics. Oxford. Second Series},
    VOLUME = {17},
      YEAR = {1966},
     PAGES = {31--38},
      ISSN = {0033-5606},
   MRCLASS = {55.40 (57.30)},
  MRNUMBER = {198460},
MRREVIEWER = {E. Dyer},
       DOI = {10.1093/qmath/17.1.31},
       URL = {https://doi.org/10.1093/qmath/17.1.31},
}

@article {Adams66ImgJ,
    AUTHOR = {Adams, J. F.},
     TITLE = {On the groups {$J(X)$}. {IV}},
   JOURNAL = {Topology},
  FJOURNAL = {Topology. An International Journal of Mathematics},
    VOLUME = {5},
      YEAR = {1966},
     PAGES = {21--71},
      ISSN = {0040-9383},
   MRCLASS = {55.40},
  MRNUMBER = {198470},
MRREVIEWER = {E. Dyer},
       DOI = {10.1016/0040-9383(66)90004-8},
       URL = {https://doi.org/10.1016/0040-9383(66)90004-8},
}

@article {Atiyah66KR,
    AUTHOR = {Atiyah, M. F.},
     TITLE = {{$K$}-theory and reality},
   JOURNAL = {Quart. J. Math. Oxford Ser. (2)},
  FJOURNAL = {The Quarterly Journal of Mathematics. Oxford. Second Series},
    VOLUME = {17},
      YEAR = {1966},
     PAGES = {367--386},
      ISSN = {0033-5606},
   MRCLASS = {55.30 (57.30)},
  MRNUMBER = {206940},
MRREVIEWER = {J. F. Adams},
       DOI = {10.1093/qmath/17.1.367},
       URL = {https://doi.org/10.1093/qmath/17.1.367},
}

@article {ABS1964,
    AUTHOR = {Atiyah, M. F. and Bott, R. and Shapiro, A.},
     TITLE = {Clifford modules},
   JOURNAL = {Topology},
  FJOURNAL = {Topology. An International Journal of Mathematics},
    VOLUME = {3},
      YEAR = {1964},
    NUMBER = {suppl, suppl. 1},
     PAGES = {3--38},
      ISSN = {0040-9383},
   MRCLASS = {57.30 (55.30)},
  MRNUMBER = {167985},
MRREVIEWER = {J. F. Adams},
       DOI = {10.1016/0040-9383(64)90003-5},
       URL = {https://doi.org/10.1016/0040-9383(64)90003-5},
}

@article{AtiyahSingerI,
  author    = {Atiyah, Michael F. and Singer, Isadore M.},
  title     = {The index of elliptic operators. {I}},
  journal   = {Annals of Mathematics},
  volume    = {87},
  year      = {1968},
  number    = {3},
  pages     = {484--530},
}

@article{Buj2012,
      title={Finite subgroups of extended Morava stabilizer groups}, 
      author={Cedric Bujard},
      year={2012},
      journal={arXiv: 1206.1951},
      archivePrefix={arXiv},
      primaryClass={math.AT}
}

@incollection{Bau08,
 author = {Bauer, Tilman},
 title = {Computation of the homotopy of the spectrum tmf},
 booktitle = {Proceedings of the conference on groups, homotopy and configuration spaces, University of Tokyo, Japan, July 5--11, 2005 in honor of the 60th birthday of Fred Cohen},
 pages = {11--40},
 year = {2008},
 publisher = {Coventry: Geometry \& Topology Publications},
 language = {English},
 zbMATH = {5261766},
 Zbl = {1147.55005}
}

@incollection{BB20,
 Author = {Beaudry, Agn{\`e}s and Barthel, Tobias},
 Title = {Chromatic structures in stable homotopy theory},
 BookTitle = {Handbook of homotopy theory},
 ISBN = {978-0-8153-6970-7; 978-1-351-25162-4},
 Pages = {163--220},
 Year = {2020},
 Publisher = {Boca Raton, FL: CRC Press},
 Language = {English},
 DOI = {10.1201/9781351251624-5},
 zbMATH = {7303338},
 Zbl = {1476.55026}
}

@article{BBHS20,
author = {Agn{\`e}s Beaudry and Irina Bobkova and Michael Hill and Vesna Stojanoska},
title = {{Invertible $K(2)$–local $E$–modules in $C_4$–spectra}},
volume = {20},
journal = {Algebraic {\&} Geometric Topology},
number = {7},
publisher = {MSP},
pages = {3423 -- 3503},
year = {2020},
doi = {10.2140/agt.2020.20.3423},
URL = {https://doi.org/10.2140/agt.2020.20.3423}
}

@article{BDGNS16,
      title={Parametrized higher category theory and higher algebra: A general introduction}, 
      author={Clark Barwick and Emanuele Dotto and Saul Glasman and Denis Nardin and Jay Shah},
      year={2016},
      eprint={1608.03654},
      archivePrefix={arXiv},
      primaryClass={math.AT},
      url={https://arxiv.org/abs/1608.03654}, 
}

@article {BG18,
    AUTHOR = {Bobkova, Irina and Goerss, Paul G.},
     TITLE = {Topological resolutions in {$K(2)$}-local homotopy theory at
              the prime 2},
   JOURNAL = {J. Topol.},
  FJOURNAL = {Journal of Topology},
    VOLUME = {11},
      YEAR = {2018},
    NUMBER = {4},
     PAGES = {918--957},
      ISSN = {1753-8416},
   MRCLASS = {55P42 (55Q10 55Q40 55Q45)},
  MRNUMBER = {3989433},
MRREVIEWER = {Constanze Roitzheim},
       DOI = {10.1112/topo.12076},
       URL = {https://doi-org.ep.fjernadgang.kb.dk/10.1112/topo.12076},
}

@article{BGHS2022,
  author       = {Agn{\`e}s Beaudry and Paul G. Goerss and Michael J. Hopkins and Vesna Stojanoska},
  title        = {Dualizing spheres for compact $p$-adic analytic groups and duality in chromatic homotopy},
  journal      = {Inventiones Mathematicae},
  volume       = {229},
  number       = {4},
  pages        = {1301--1434},
  year         = {2022},
  doi          = {10.1007/s00222-022-01120-1},
  eprint       = {arXiv:2010.09518},
  archivePrefix = {arXiv},
  primaryClass = {math.AT},
}

@article {Behrens2006,
    AUTHOR = {Behrens, Mark},
     TITLE = {A modular description of the {$K(2)$}-local sphere at the
              prime 3},
   JOURNAL = {Topology},
  FJOURNAL = {Topology. An International Journal of Mathematics},
    VOLUME = {45},
      YEAR = {2006},
    NUMBER = {2},
     PAGES = {343--402},
      ISSN = {0040-9383},
   MRCLASS = {55Q40 (11F23 14H52 55N34 55Q51 55S05)},
  MRNUMBER = {2193339},
MRREVIEWER = {Mark\ Hovey},
       DOI = {10.1016/j.top.2005.08.005},
       URL = {https://doi.org/10.1016/j.top.2005.08.005},
}

@article {Beaudry2015,
    AUTHOR = {Beaudry, Agn\`es},
     TITLE = {The algebraic duality resolution at {$p=2$}},
   JOURNAL = {Algebr. Geom. Topol.},
  FJOURNAL = {Algebraic \& Geometric Topology},
    VOLUME = {15},
      YEAR = {2015},
    NUMBER = {6},
     PAGES = {3653--3705},
      ISSN = {1472-2747,1472-2739},
   MRCLASS = {55Q45 (55N22 55P60 55T25)},
  MRNUMBER = {3450774},
MRREVIEWER = {Vidhyanath\ K.\ Rao},
       DOI = {10.2140/agt.2015.15.3653},
       URL = {https://doi.org/10.2140/agt.2015.15.3653},
}

@article {BGH2022,
    AUTHOR = {Beaudry, Agn\`es and Goerss, Paul G. and Henn, Hans-Werner},
     TITLE = {Chromatic splitting for the {$K(2)$}-local sphere at {$p =
              2$}},
   JOURNAL = {Geom. Topol.},
  FJOURNAL = {Geometry \& Topology},
    VOLUME = {26},
      YEAR = {2022},
    NUMBER = {1},
     PAGES = {377--476},
      ISSN = {1465-3060,1364-0380},
   MRCLASS = {55P60 (55P42 55Q51)},
  MRNUMBER = {4404881},
MRREVIEWER = {Katsumi\ Shimomura},
       DOI = {10.2140/gt.2022.26.377},
       URL = {https://doi.org/10.2140/gt.2022.26.377},
}

@Article{BH20,
 Author = {Blumberg, Andrew J. and Hill, Michael A.},
 Title = {{{\(G\)}}-symmetric monoidal categories of modules over equivariant commutative ring spectra},
 FJournal = {Tunisian Journal of Mathematics},
 Journal = {Tunis. J. Math.},
 ISSN = {2576-7658},
 Volume = {2},
 Number = {2},
 Pages = {237--286},
 Year = {2020},
 Language = {English},
 DOI = {10.2140/tunis.2020.2.237},
 zbMATH = {7119004},
 Zbl = {1427.55008}
}

@article{BHHM20,
  author  = {Behrens, Mark and Hill, Michael and Hopkins, Michael J. and Mahowald, Mark},
  title   = {Detecting exotic spheres in low dimensions using coker J},
  journal = {Journal of the London Mathematical Society},
  series  = {2},
  volume  = {101},
  number  = {3},
  pages   = {1173--1218},
  year    = {2020},
  doi     = {10.1112/jlms.12301},
  eprint  = {1708.06854},
  archivePrefix = {arXiv}
}

@article {BHSZ21,
    AUTHOR = {Beaudry, Agn\`es and Hill, Michael A. and Shi, XiaoLin Danny and
              Zeng, Mingcong},
     TITLE = {Models of {L}ubin-{T}ate spectra via real bordism theory},
   JOURNAL = {Adv. Math.},
  FJOURNAL = {Advances in Mathematics},
    VOLUME = {392},
      YEAR = {2021},
     PAGES = {Paper No. 108020, 58},
      ISSN = {0001-8708},
   MRCLASS = {55N22 (11S31 55P42)},
  MRNUMBER = {4313964},
       DOI = {10.1016/j.aim.2021.108020},
       URL = {https://doi-org.ep.fjernadgang.kb.dk/10.1016/j.aim.2021.108020},
}

@article{BMQ20,
  author  = {Behrens, Mark and Mahowald, Mark and Quigley, Gregory},
  title   = {The tmf-hurewicz homomorphism and $v_2^{32}$-periodic families},
  journal = {Geometry \& Topology},
  volume  = {24},
  number  = {2},
  pages   = {781--822},
  year    = {2020},
  doi     = {10.2140/gt.2020.24.781}
}

@article {BO16,
    AUTHOR = {Behrens, Mark and Ormsby, Kyle},
     TITLE = {On the homotopy of {$Q(3)$} and {$Q(5)$} at the prime 2},
   JOURNAL = {Algebr. Geom. Topol.},
  FJOURNAL = {Algebraic \& Geometric Topology},
    VOLUME = {16},
      YEAR = {2016},
    NUMBER = {5},
     PAGES = {2459--2534},
      ISSN = {1472-2747},
   MRCLASS = {55Q45 (11F33 55Q51)},
  MRNUMBER = {3572338},
MRREVIEWER = {Haynes R. Miller},
       DOI = {10.2140/agt.2016.16.2459},
       URL = {https://doi.org/10.2140/agt.2016.16.2459},
}

@article {Bott59,
    AUTHOR = {Bott, Raoul},
     TITLE = {The stable homotopy of the classical groups},
   JOURNAL = {Ann. of Math. (2)},
  FJOURNAL = {Annals of Mathematics. Second Series},
    VOLUME = {70},
      YEAR = {1959},
     PAGES = {313--337},
      ISSN = {0003-486X},
   MRCLASS = {57.00},
  MRNUMBER = {110104},
MRREVIEWER = {M. A. Kervaire},
       DOI = {10.2307/1970106},
       URL = {https://doi.org/10.2307/1970106},
}

@article{Car25,
      title={Slice spectral sequences through synthetic spectra}, 
      author={Christian Carrick},
      year={2025},
      eprint={2510.19501},
      archivePrefix={arXiv},
      primaryClass={math.AT},
      url={https://arxiv.org/abs/2510.19501}, 
}

@article {COCTALOS,
	AUTHOR = {Hopkins, Michael},
	TITLE = {Complex oriented cohomology theories and the Language of Stacks},
YEAR = {1999},
JOURNAL = {Course Notes}
}

@book{DFHH14,
  title={Topological modular forms},
  author={Douglas, Christopher L and Francis, John and Henriques, Andr{\'e} G and Hill, Michael A},
  volume={201},
  year={2014},
  publisher={American Mathematical Soc.}
}

@article{DHS88,
  author  = {Ethan S. Devinatz and Michael J. Hopkins and Jeffrey H. Smith},
  title   = {Nilpotence and stable homotopy theory {I}},
  journal = {Annals of Mathematics},
  volume  = {128},
  number  = {2},
  pages   = {207--241},
  year    = {1988},
}

@article{DKLLW24,
author={Duan, Zhipeng
and Kong, Hana Jia
and Li, Guchuan
and Lu, Yunze
and Wang, Guozhen},
title={{RO(G)}-Graded Homotopy Fixed Point Spectral Sequence for Height 2 {M}orava {E}-Theory},
journal={Peking Mathematical Journal},
year={2024},
month={May},
day={27},
issn={2524-7182},
doi={10.1007/s42543-024-00087-7},
url={https://doi.org/10.1007/s42543-024-00087-7}
}

@article{DLS2022,
author = {Duan, Zhipeng and Li, Guchuan and Shi, XiaoLin Danny},
 title = {Vanishing lines in chromatic homotopy theory},
 fjournal = {Geometry \& Topology},
 journal = {Geom. Topol.},
 issn = {1465-3060},
 volume = {29},
 number = {2},
 pages = {903--930},
 year = {2025},
 language = {English},
 doi = {10.2140/gt.2025.29.903},
 zbMATH = {8038281}
}

@article {Furuta01,
    AUTHOR = {Furuta, M.},
     TITLE = {Monopole equation and the {$\frac{11}8$}-conjecture},
   JOURNAL = {Math. Res. Lett.},
  FJOURNAL = {Mathematical Research Letters},
    VOLUME = {8},
      YEAR = {2001},
    NUMBER = {3},
     PAGES = {279--291},
      ISSN = {1073-2780},
   MRCLASS = {57R15 (53C27 57R57)},
  MRNUMBER = {1839478},
       DOI = {10.4310/MRL.2001.v8.n3.a5},
       URL = {https://doi.org/10.4310/MRL.2001.v8.n3.a5},
}

@incollection {GH04,
    AUTHOR = {Goerss, P. G. and Hopkins, M. J.},
     TITLE = {Moduli spaces of commutative ring spectra},
 BOOKTITLE = {Structured ring spectra},
    SERIES = {London Math. Soc. Lecture Note Ser.},
    VOLUME = {315},
     PAGES = {151--200},
 PUBLISHER = {Cambridge Univ. Press, Cambridge},
      YEAR = {2004},
   MRCLASS = {55P43},
  MRNUMBER = {2125040},
MRREVIEWER = {Birgit Richter},
       DOI = {10.1017/CBO9780511529955.009},
       URL = {https://doi-org.ep.fjernadgang.kb.dk/10.1017/CBO9780511529955.009},
}

@article {GHMR05,
    AUTHOR = {Goerss, P. and Henn, H.-W. and Mahowald, M. and Rezk, C.},
     TITLE = {A resolution of the {$K(2)$}-local sphere at the prime 3},
   JOURNAL = {Ann. of Math. (2)},
  FJOURNAL = {Annals of Mathematics. Second Series},
    VOLUME = {162},
      YEAR = {2005},
    NUMBER = {2},
     PAGES = {777--822},
      ISSN = {0003-486X,1939-8980},
   MRCLASS = {55Q45 (55P60)},
  MRNUMBER = {2183282},
MRREVIEWER = {Mark\ J.\ Behrens},
       DOI = {10.4007/annals.2005.162.777},
       URL = {https://doi.org/10.4007/annals.2005.162.777},
}

@incollection {Henn2007,
    AUTHOR = {Henn, Hans-Werner},
     TITLE = {On finite resolutions of {$K(n)$}-local spheres},
 BOOKTITLE = {Elliptic cohomology},
    SERIES = {London Math. Soc. Lecture Note Ser.},
    VOLUME = {342},
     PAGES = {122--169},
 PUBLISHER = {Cambridge Univ. Press, Cambridge},
      YEAR = {2007},
   MRCLASS = {55P43 (55N15 55P60)},
  MRNUMBER = {2330511},
MRREVIEWER = {Mark J. Behrens},
       DOI = {10.1017/CBO9780511721489.008},
       URL = {https://doi.org/10.1017/CBO9780511721489.008},
}

@incollection {Henn2019,
    AUTHOR = {Henn, Hans-Werner},
     TITLE = {The centralizer resolution of the {$K(2)$}-local sphere at the
              prime 2},
 BOOKTITLE = {Homotopy theory: tools and applications},
    SERIES = {Contemp. Math.},
    VOLUME = {729},
     PAGES = {93--128},
 PUBLISHER = {Amer. Math. Soc., [Providence], RI},
      YEAR = {[2019] \copyright 2019},
   MRCLASS = {55T15},
  MRNUMBER = {3959597},
MRREVIEWER = {Agn\`es Beaudry},
       DOI = {10.1090/conm/729/14692},
       URL = {https://doi.org/10.1090/conm/729/14692},
}

@article {Hew95,
    AUTHOR = {Hewett, Thomas},
     TITLE = {Finite subgroups of division algebras over local fields},
   JOURNAL = {J. Algebra},
  FJOURNAL = {Journal of Algebra},
    VOLUME = {173},
      YEAR = {1995},
    NUMBER = {3},
     PAGES = {518--548},
      ISSN = {0021-8693},
   MRCLASS = {16K20 (11S45 12E15 16U60)},
  MRNUMBER = {1327867},
MRREVIEWER = {Timothy J. Ford},
       DOI = {10.1006/jabr.1995.1101},
       URL = {https://doi-org.ep.fjernadgang.kb.dk/10.1006/jabr.1995.1101},
}

@article {Hew99,
    AUTHOR = {Hewett, Thomas},
     TITLE = {Normalizers of finite subgroups of division algebras over
              local fields},
   JOURNAL = {Math. Res. Lett.},
  FJOURNAL = {Mathematical Research Letters},
    VOLUME = {6},
      YEAR = {1999},
    NUMBER = {3-4},
     PAGES = {271--286},
      ISSN = {1073-2780},
   MRCLASS = {11S45 (12E15 16K20 16U60)},
  MRNUMBER = {1713129},
MRREVIEWER = {Timothy J. Ford},
       DOI = {10.4310/MRL.1999.v6.n3.a2},
       URL = {https://doi.org/10.4310/MRL.1999.v6.n3.a2},
}

@article {HHR16,
    AUTHOR = {Hill, M. A. and Hopkins, M. J. and Ravenel, D. C.},
     TITLE = {On the nonexistence of elements of {K}ervaire invariant one},
   JOURNAL = {Ann. of Math. (2)},
  FJOURNAL = {Annals of Mathematics. Second Series},
    VOLUME = {184},
      YEAR = {2016},
    NUMBER = {1},
     PAGES = {1--262},
      ISSN = {0003-486X},
   MRCLASS = {55P91 (55N22 55P42 55Q45 55T15 55U35 57R15)},
  MRNUMBER = {3505179},
MRREVIEWER = {Paul G. Goerss},
       DOI = {10.4007/annals.2016.184.1.1},
       URL = {https://doi-org.ep.fjernadgang.kb.dk/10.4007/annals.2016.184.1.1},
}

@article {HHR17,
    AUTHOR = {Hill, Michael A. and Hopkins, Michael J. and Ravenel, Douglas
              C.},
     TITLE = {The slice spectral sequence for the {$C_4$} analog of real
              {$K$}-theory},
   JOURNAL = {Forum Math.},
  FJOURNAL = {Forum Mathematicum},
    VOLUME = {29},
      YEAR = {2017},
    NUMBER = {2},
     PAGES = {383--447},
      ISSN = {0933-7741},
   MRCLASS = {55Q10 (55P42 55Q91 55R45 55T05)},
  MRNUMBER = {3619120},
MRREVIEWER = {Vigleik Angeltveit},
       DOI = {10.1515/forum-2016-0017},
       URL = {https://doi-org.ep.fjernadgang.kb.dk/10.1515/forum-2016-0017},
}

@book{HHR21,
 author = {Hill, Michael A. and Hopkins, Michael J. and Ravenel, Douglas C.},
 title = {Equivariant stable homotopy theory and the {Kervaire} invariant problem},
 fseries = {New Mathematical Monographs},
 series = {New Math. Monogr.},
 volume = {40},
 isbn = {978-1-108-83144-4; 978-1-108-91727-8},
 year = {2021},
 publisher = {Cambridge: Cambridge University Press},
 language = {English},
 doi = {10.1017/9781108917278},
 zbMATH = {7298515},
 Zbl = {1484.55001}
}

@Article{HLS21,
 Author = {Heard, Drew and Li, Guchuan and Shi, XiaoLin Danny},
 Title = {Picard groups and duality for real {Morava} {{\(E\)}}-theories},
 FJournal = {Algebraic \& Geometric Topology},
 Journal = {Algebr. Geom. Topol.},
 ISSN = {1472-2747},
 Volume = {21},
 Number = {6},
 Pages = {2703--2760},
 Year = {2021},
 Language = {English},
 DOI = {10.2140/agt.2021.21.2703},
 zbMATH = {7444682},
 Zbl = {1481.19008}
}

@article {HM17,
    AUTHOR = {Hill, Michael A. and Meier, Lennart},
     TITLE = {The {$C_2$}-spectrum {${\rm Tmf}_1(3)$} and its invertible
              modules},
   JOURNAL = {Algebr. Geom. Topol.},
  FJOURNAL = {Algebraic \& Geometric Topology},
    VOLUME = {17},
      YEAR = {2017},
    NUMBER = {4},
     PAGES = {1953--2011},
      ISSN = {1472-2747},
   MRCLASS = {55N34 (55P42)},
  MRNUMBER = {3685599},
MRREVIEWER = {Jorge Andres Devoto},
       DOI = {10.2140/agt.2017.17.1953},
       URL = {https://doi-org.ep.fjernadgang.kb.dk/10.2140/agt.2017.17.1953},
}

@article{HS98,
  author  = {Michael J. Hopkins and Jeffrey H. Smith},
  title   = {Nilpotence and stable homotopy theory {II}},
  journal = {Annals of Mathematics},
  volume  = {148},
  number  = {1},
  pages   = {1--49},
  year    = {1998},
}

@article {HS20,
    AUTHOR = {Hahn, Jeremy and Shi, XiaoLin Danny},
     TITLE = {Real orientations of {L}ubin-{T}ate spectra},
   JOURNAL = {Invent. Math.},
  FJOURNAL = {Inventiones Mathematicae},
    VOLUME = {221},
      YEAR = {2020},
    NUMBER = {3},
     PAGES = {731--776},
      ISSN = {0020-9910},
   MRCLASS = {55P43},
  MRNUMBER = {4132956},
MRREVIEWER = {Drew Heard},
       DOI = {10.1007/s00222-020-00960-z},
       URL = {https://doi-org.proxy.lib.umich.edu/10.1007/s00222-020-00960-z},
}

@article {HSWX23,
    AUTHOR = {Hill, Michael A. and Shi, XiaoLin Danny and Wang, Guozhen and
              Xu, Zhouli},
     TITLE = {The slice spectral sequence of a {$C_4$}-equivariant height-4
              {L}ubin-{T}ate theory},
   JOURNAL = {Mem. Amer. Math. Soc.},
  FJOURNAL = {Memoirs of the American Mathematical Society},
    VOLUME = {288},
      YEAR = {2023},
    NUMBER = {1429},
     PAGES = {v+119},
      ISSN = {0065-9266,1947-6221},
      ISBN = {978-1-4704-7468-3; 978-1-4704-7571-0},
   MRCLASS = {55P91 (55P92 55Q10 55Q40 55Q91)},
  MRNUMBER = {4627086},
       DOI = {10.1090/memo/1429},
       URL = {https://doi.org/10.1090/memo/1429},
}

@article {HLSX2022,
    AUTHOR = {Hopkins, Michael J. and Lin, Jianfeng and Shi, XiaoLin Danny
              and Xu, Zhouli},
     TITLE = {Intersection forms of spin 4-manifolds and the {${\rm
              Pin}(2)$}-equivariant {M}ahowald invariant},
   JOURNAL = {Commun. Am. Math. Soc.},
  FJOURNAL = {Communications of the American Mathematical Society},
    VOLUME = {2},
      YEAR = {2022},
     PAGES = {22--132},
   MRCLASS = {55P91 (57K40 57K41)},
  MRNUMBER = {4385297},
MRREVIEWER = {J. P. C. Greenlees},
       DOI = {10.1090/cams/4},
       URL = {https://doi.org/10.1090/cams/4},
}

@article{HuKriz,
title = "Real-oriented homotopy theory and an analogue of the {A}dams-{N}ovikov spectral sequence ",
journal = "Topology ",
volume = "40",
number = "2",
pages = "317 - 399",
year = "2001",
note = "",
issn = "0040-9383",
doi = "http://dx.doi.org/10.1016/S0040-9383(99)00065-8",
url = "http://www.sciencedirect.com/science/article/pii/S0040938399000658",
author = "Po Hu and Igor Kriz",
}

@article{HY18,
 author = {Hill, Michael A. and Yarnall, Carolyn},
 title = {A new formulation of the equivariant slice filtration with applications to {{\(C_p\)}}-slices},
 fjournal = {Proceedings of the American Mathematical Society},
 journal = {Proc. Am. Math. Soc.},
 issn = {0002-9939},
 volume = {146},
 number = {8},
 pages = {3605--3614},
 year = {2018},
 language = {English},
 doi = {10.1090/proc/13906},
 zbMATH = {6880242},
 Zbl = {1395.55014}
}

@article{KervaireMilnor63,
  author    = {Kervaire, Michel A. and Milnor, John W.},
  title     = {Groups of homotopy spheres: I},
  journal   = {Annals of Mathematics},
  series    = {Second Series},
  volume    = {77},
  number    = {3},
  year      = {1963},
  pages     = {504--537}
}

@article{Lur17,
    author ={Lurie, Jacob} ,
    title ={Higher Algebra} ,
    journal = {https://www.math.ias.edu/~lurie/papers/HA.pdf},
    year ={2017} 
}

@article{Lur18,
  title={Elliptic cohomology {II}: orientations},
  author={Lurie, Jacob},
  journal= {https://www.math.ias.edu/~lurie/papers/Elliptic-II.pdf},
  year={2018},
}

@article {LurieChromatic,
	AUTHOR = {Lurie, Jacob},
	TITLE = {Chromatic Homotopy Theory},
	JOURNAL = {Course Notes},
	YEAR = {2010},
}

@article{LiuShiYan2025,
  title={The generalized {T}ate diagram of the equivariant slice filtration},
  author={Liu, Yutao and Shi, XiaoLin Danny and Yan, Guoqi},
  journal={preprint},
  year={2025}
}

@article {LSWX2019,
    AUTHOR = {Li, Guchuan and Shi, XiaoLin Danny and Wang, Guozhen and Xu, Zhouli},
     TITLE = {Hurewicz images of real bordism theory and real {J}ohnson-{W}ilson theories},
   JOURNAL = {Adv. Math.},
  FJOURNAL = {Advances in Mathematics},
    VOLUME = {342},
      YEAR = {2019},
     PAGES = {67--115},
      ISSN = {0001-8708},
   MRCLASS = {55N22 (55Q91 55T15)},
  MRNUMBER = {3877362},
MRREVIEWER = {Hans-Werner Henn},
       DOI = {10.1016/j.aim.2018.11.002},
       URL = {https://doi.org/10.1016/j.aim.2018.11.002},
}

@ARTICLE{LWX24,
       author = {{Lin}, Weinan and {Wang}, Guozhen and {Xu}, Zhouli},
        title = "{On the Last Kervaire Invariant Problem}",
      journal = {arXiv e-prints},
         year = 2024,
        month = dec,
          eid = {arXiv:2412.10879},
        pages = {arXiv:2412.10879},
          doi = {10.48550/arXiv.2412.10879},
archivePrefix = {arXiv},
       eprint = {2412.10879},
 primaryClass = {math.AT},
       adsurl = {https://ui.adsabs.harvard.edu/abs/2024arXiv241210879L}
}

@article {MahowaldImageJ,
    AUTHOR = {Mahowald, Mark},
     TITLE = {The image of {$J$} in the {$EHP$} sequence},
   JOURNAL = {Ann. of Math. (2)},
  FJOURNAL = {Annals of Mathematics. Second Series},
    VOLUME = {116},
      YEAR = {1982},
    NUMBER = {1},
     PAGES = {65--112},
      ISSN = {0003-486X},
   MRCLASS = {55Q40},
  MRNUMBER = {662118},
MRREVIEWER = {Donald M. Davis},
       DOI = {10.2307/2007048},
       URL = {https://doi.org/10.2307/2007048},
}

@article{MahowaldRezk09,
 author = {Mahowald, Mark and Rezk, Charles},
 title = {Topological modular forms of level 3},
 fjournal = {Pure and Applied Mathematics Quarterly},
 journal = {Pure Appl. Math. Q.},
 issn = {1558-8599},
 volume = {5},
 number = {2},
 pages = {853--872},
 year = {2009},
 language = {English},
 doi = {10.4310/PAMQ.2009.v5.n2.a9},
 zbMATH = {5567766},
 Zbl = {1192.55006}
}

@article{Mor24,
      title={Duals of higher real {$K$}-theories at $p=2$}, 
      author={Juan C. Moreno Del Angel},
      year={2024},
      eprint={2410.10726},
      archivePrefix={arXiv},
      primaryClass={math.AT},
      url={https://arxiv.org/abs/2410.10726}, 
}

@book{MM02,
 author = {Mandell, M. A. and May, J. P.},
 title = {Equivariant orthogonal spectra and {{\(S\)}}-modules},
 fseries = {Memoirs of the American Mathematical Society},
 series = {Mem. Am. Math. Soc.},
 issn = {0065-9266},
 volume = {755},
 isbn = {978-0-8218-2936-3; 978-1-4704-0348-5},
 year = {2002},
 publisher = {Providence, RI: American Mathematical Society (AMS)},
 language = {English},
 doi = {10.1090/memo/0755},
 zbMATH = {1812424},
 Zbl = {1025.55002}
}

@article{MRW77,
  author  = {Haynes R. Miller and Douglas C. Ravenel and W. Stephen Wilson},
  title   = {Periodic phenomena in the Adams–Novikov spectral sequence},
  journal = {Annals of Mathematics},
  volume  = {106},
  number  = {3},
  pages   = {469--516},
  year    = {1977},
}

@article{MSZ23,
title = {The localized slice spectral sequence, norms of Real bordism, and the Segal conjecture},
journal = {Advances in Mathematics},
volume = {412},
pages = {108804},
year = {2023},
issn = {0001-8708},
doi = {https://doi.org/10.1016/j.aim.2022.108804},
url = {https://www.sciencedirect.com/science/article/pii/S0001870822006211},
author = {Lennart Meier and XiaoLin Danny Shi and Mingcong Zeng},
}

@article{MSZ24,
      title={Transchromatic phenomena in the equivariant slice spectral sequence}, 
      author={Lennart Meier and XiaoLin Danny Shi and Mingcong Zeng},
      year={2024},
      eprint={2403.00741},
      archivePrefix={arXiv},
      primaryClass={math.AT},
      url={https://arxiv.org/abs/2403.00741}, 
}

@article {Quillen69,
    AUTHOR = {Quillen, Daniel},
     TITLE = {On the formal group laws of unoriented and  complex cobordism
              theory},
   JOURNAL = {Bull. Amer. Math. Soc.},
  FJOURNAL = {Bulletin of the American Mathematical Society},
    VOLUME = {75},
      YEAR = {1969},
     PAGES = {1293--1298},
      ISSN = {0002-9904},
   MRCLASS = {57.10},
  MRNUMBER = {0253350},
MRREVIEWER = {R. E. Stong},
       URL = {https://doi.org/10.1090/S0002-9904-1969-12401-8},
}

@article{Ravenel78,
  author  = {Ravenel, Douglas C.},
  title   = {The non-existence of odd primary {Arf} invariant elements in stable homotopy},
  journal = {Mathematical Proceedings of the Cambridge Philosophical Society},
  volume  = {83},
  number  = {3},
  year    = {1978},
  pages   = {429--443},
  doi     = {10.1017/S0305004100054712},
}

@book{RavenelOrange,
  author    = {Douglas C. Ravenel},
  title     = {Nilpotence and Periodicity in Stable Homotopy Theory},
  series    = {Annals of Mathematics Studies},
  volume    = {128},
  publisher = {Princeton University Press},
  year      = {1992},
}

@book{RavenelGreen,
  title={Complex cobordism and stable homotopy groups of spheres},
  author={Ravenel, Douglas C},
  year={2003},
  publisher={American Mathematical Soc.}
}

@incollection {Rez98,
    AUTHOR = {Rezk, Charles},
     TITLE = {Notes on the {H}opkins-{M}iller theorem},
 BOOKTITLE = {Homotopy theory via algebraic geometry and group
              representations ({E}vanston, {IL}, 1997)},
    SERIES = {Contemp. Math.},
    VOLUME = {220},
     PAGES = {313--366},
 PUBLISHER = {Amer. Math. Soc., Providence, RI},
      YEAR = {1998},
   MRCLASS = {55N22 (55S99)},
  MRNUMBER = {1642902},
       DOI = {10.1090/conm/220/03107},
       URL = {http://dx.doi.org/10.1090/conm/220/03107},
}

@book {Ull13,
    AUTHOR = {Ullman, John Richard},
     TITLE = {On the {R}egular {S}lice {S}pectral {S}equence},
      NOTE = {Thesis (Ph.D.)--Massachusetts Institute of Technology},
 PUBLISHER = {ProQuest LLC, Ann Arbor, MI},
      YEAR = {2013},
     PAGES = {(no paging)},
   MRCLASS = {Thesis},
  MRNUMBER = {3211466},
       URL =
              {http://gateway.proquest.com.ep.fjernadgang.kb.dk/openurl?url_ver=Z39.88-2004&rft_val_fmt=info:ofi/fmt:kev:mtx:dissertation&res_dat=xri:pqm&rft_dat=xri:pqdiss:0829532},
}

\end{document}